\documentclass{article}
\usepackage{fullpage}

\usepackage{amsmath,amsfonts,amsthm,amssymb}
\usepackage{mathrsfs}
\usepackage{color,xcolor}
\usepackage{ifpdf}
\usepackage{psfrag}
\usepackage{graphicx,graphics,subfigure,epsfig}
\usepackage{url,algorithm,algorithmic}
\usepackage{hyperref}
\usepackage{xspace}
\usepackage{xargs}
\usepackage[colorinlistoftodos,prependcaption,textsize=tiny]{todonotes}

\newtheorem{theorem}{Theorem}[section]

\newtheorem{lemma}[theorem]{Lemma}

\newtheorem{proposition}[theorem]{Proposition}
\newtheorem{remark}[theorem]{Remark}

\definecolor{OliveGreen}{rgb}{0.33, 0.42, 0.18}
\definecolor{Plum}{rgb}{0.8, 0.6, 0.8}
\newcommandx{\unsure}[2][1=]{\todo[linecolor=red,backgroundcolor=red!25,bordercolor=red,#1]{#2}}
\newcommandx{\change}[2][1=]{\todo[linecolor=blue,backgroundcolor=blue!25,bordercolor=blue,#1]{#2}}
\newcommandx{\info}[2][1=]{\todo[linecolor=OliveGreen,backgroundcolor=OliveGreen!25,bordercolor=OliveGreen,#1]{#2}}
\newcommandx{\improvement}[2][1=]{\todo[linecolor=Plum,backgroundcolor=Plum!25,bordercolor=Plum,#1]{#2}}
\newcommandx{\thiswillnotshow}[2][1=]{\todo[disable,#1]{#2}}
\newcommand{\dpar}[2]{\dfrac{\partial #1}{\partial #2}}

 \newcommand{\R}{\mathbb R}
 \newcommand{\Z}{\mathbb Z}
 \newcommand{\N}{\mathbb N}

\renewcommand{\SS}{\mathcal{\mathbb S}}

\newcommand{\bbf}{{\mathbf {f}}}

\newcommand{\bu}{\mathbf{u}}
\newcommand{\bv}{\mathbf{v}}

\newcommand{\hf}{\hat{\mathbf{f}}}

\newcommand{\bx}{\mathbf{x}}

\newcommand{\remi}[1]{{#1}}
\newcommand{\remI}[1]{{#1}}
\newcommand{\remiIII}[1]{\textcolor{green}{#1}}
\begin{document}
\title{A combination of Residual Distribution   and the Active Flux formulations or a new class of schemes that can combine several writings of the same hyperbolic problem: application to the 1D Euler equations}
\author{R. Abgrall\\
Institute of Mathematics, University of Z\"urich\\
Winterhurerstrasse 190, CH 8057 Z\"urich \\
email: remi.abgrall@math.uzh.ch
}
\date{\today}
\maketitle
\begin{abstract}
\remi{We show how to combine in a natural way (i.e. without any test nor switch) the conservative and non conservative formulations of an hyperbolic system that has a conservative form. This is inspired from two different class of schemes: the Residual Distribution one \cite{MR4090481}, and the Active Flux formulations \cite{AF1, AF3, AF4,AF5,RoeAF}. The solution is globally continuous, and as in the active flux method, described by a combination of point values and average values.
Unlike the "classical" active flux methods, the meaning of the point-wise and cell average degrees of freedom is different, and hence follow different form of PDEs: it is a conservative version of the cell average, and a possibly non conservative one for the points.
This new class of scheme is proved to satisfy a Lax-Wendroff like theorem. We also develop a method to perform non linear stability. We illustrate the behaviour on several benchmarks, some quite challenging.  }
\end{abstract}


\section{Introduction}
The notion of conservation is essential in the numerical approximation of hyperbolic systems of conservation: if it is violated, there is no chance, in practice, to compute the right weak solution in the limit of mesh refinement. This statement is known since the celebrated work of Lax and Wendroff \cite{laxwendroff}, and what happens when conservation is violated has been discussed by Hou and Le Floch \cite{hou}. This conservation requirement imposes the use of the conservation form of the system. However, in many practical situations, this is not really the one one would like to deal with, since in addition to conservation constraints, one also seeks for the preservation of additional features, like contacts for fluid mechanics, or entropy decrease for shocks.

In this paper, we are interested in compressible fluid dynamics. Several authors have already considered the problem of the correct discretisation of the non conservative form of the system. In the purely Lagrangian framework, when the system is described by the momentum equation and the Gibbs equality, this has been done since decades: one can consider the seminal work of Wilkins, to begin with, and the problem is still of interest: one can consider \cite{Rieben,svetlana1,svetlana2} where high order is sought for. In the case of the Eulerian formulation, there are less work. One can mention \cite{herbin,despre,ksenya} where staggered meshes are used, the thermodynamic variables are localised in the cells, while the kinetic ones are localised at the grid points, or \cite{paola} where a non conservative formulation with correction is used from scratch. The first two references show how to construct at most second order scheme, while the last one shows this for any order. All constructions are quite involved in term of algebra, because one has to transfert information from the original grid and the staggered one.

 In this paper, we aim at showing how the notion of conservation introduced in the residual distribution framework \cite{conservationRD} is flexible enough to allow to deal directly with the non conservative form of the system, while the correct solutions are obtained in the limit of mesh refinement. More precisely, we show how to to deal both with the conservative and non conservative form of the PDE, without any switch, as it was the case in \cite{karni}. We illustrate our strategy on several versions of the non conservative form, and provide first, second order and third order accurate version of the scheme.  More than a particular example, we describe  a general strategy which is quite simple.
 The systems on which we will work are descriptions of the Euler equations for fluid mechanics:
 \begin{itemize}
 \item The conservation one:
 \begin{equation}
 \label{eq:conservative}\dpar{}{t}\begin{pmatrix}\rho \\ \rho u \\ E\end{pmatrix}+\dpar{}{x}\begin{pmatrix} \rho u\\ \rho u^2+p\\ u(E+p)\end{pmatrix}=0\end{equation}
 \item the primitive formulation:
 \begin{equation}\label{eq:primitive}\dpar{}{t}\begin{pmatrix}\rho \\ u \\ p\end{pmatrix}+\begin{pmatrix}\dpar{\rho u}{x}\\
 u\dpar{u}{x}+\frac{1}{\rho}\dpar{p}{x}\\
 u\dpar{p}{x}+(e+p)\dpar{u}{x}
 \end{pmatrix}=0
 \end{equation}
 \item The "entropy" formulation:
 \begin{equation}\label{eq:entropy}
\dpar{}{t}\begin{pmatrix}p \\ u \\ s\end{pmatrix}+\begin{pmatrix}u\dpar{p}{x}+(e+p)\dpar{u}{x}\\
 u\dpar{u}{x}+\frac{1}{\rho}\dpar{p}{x}\\
 u\dpar{s}{x}
 \end{pmatrix}=0
 \end{equation}
 \end{itemize}
 where as usual $\rho$ is the density, $u$ the velocity, $p$ the pressure, $E=e+\tfrac{1}{2}\rho u^2$ is the total energy, $e=(\gamma-1)p$ and $s=\log (p)-\gamma \log(\rho)$ is the entropy. The ratio of specific heats, $\gamma$ is supposed to be constant here, mostly for simplicity.
 
 \remi{This paper has several source of inspirations. The first one is the residual distribution (RD) framework, and in particular \cite{conservationRD}. The second one is the family of active flux \cite{AF1,AF2,AF3,AF4,AF5}, where the solution is represented by a cell average and point values. The conservation is recovered from how the average is updated. Here the difference comes from the fact that in addition several forms of the same system can be conserved, as \eqref{eq:conservative}, \eqref{eq:primitive},  \eqref{eq:entropy} for the point value update while a Lax Wendroff like result can still be shown. If the same system were used, both for the cell average and the point values, this would easily fit into the RD framework, using the structure of the polynomial reconstruction. The difference with Active Flux is that we use only the representation of the solution within one cell, and not a fancy flux evaluation. Another difference is about the way the solution is evolved in time: the AF method uses the method of characteristics to evolve the point value, while here we rely on more standard Runge-Kutta methods.}

The format of the paper is \remiIII{ as follows}. In a first part, we explain the general principles of our method, and justify why, under the assumptions made on the numerical sequence for  the Lax-Wendroff theorem  (boundedness in $L^\infty$ and strong convergence in a $L^p$, $p\geq 1$,  of a subsequence toward a $v\in L^p$, then this $v$ is a weak solution of the problem), \remi{we can also show the convergence of a subsequence to a weak solution of the problem, under the same assumptions.}
{ In the second part, we describe several discretisations of the method, and in  a third part we provide several simulations to illustrate the method.}

In this paper, the letter $C$ denotes a constant, and we uses the standard "algebra", for example $C\times C=C$,  $C+C=C$, or $\alpha C=C$ for any constant $\alpha\in \R$.
\section{The method}
\subsection{Principle}
We consider the problem
\begin{subequations}
\label{eq:1}
\begin{equation}
\label{eq:1:1}
\dpar{\bu}{t}+\dpar{\bbf(\bu)}{x}=0, \qquad x\in \R
\end{equation}
with the initial condition 
\begin{equation}
\label{eq:1:2}
\bu(x,0)=\bu_0(x), \qquad x\in \R
\end{equation}
Here $\bu\in \mathcal{D}_\bu\subset \R^p$. For smooth solutions, we also consider an equivalent formulation in the form
\begin{equation}
\label{eq:1:3}
\dpar{\bv}{t}+J\dpar{\bv}{x}=0
\end{equation}
\end{subequations}
where $\bv=\Psi(\bu)\in \mathcal{D}_\bv$ and $\Psi:\mathcal{D}_\bu\rightarrow \mathcal{D}_\bv$ is assumed to be one-to-one \remI{and $C^1$ (as well as the inverse function).}
For example, if \eqref{eq:1} corresponds to \eqref{eq:conservative}, then
$$\mathcal{D}_\bu=\{ (\rho, \rho u, E)\in \R^3 \text{ such that }\rho>0 \text{ and } E-\frac{1}{2}\rho u^2>0\}.$$
If \eqref{eq:1:3} corresponds to \eqref{eq:primitive}, then
$$\mathcal{D}_\bv=\{(\rho, u, p)\text{ such that }\rho>0 \text{ and } p>0\}$$ and (for a perfect gas) the mapping $\Psi$ corresponds to $(\rho, \rho u, E)\mapsto \big(\rho, u, p=(\gamma-1)\big ( E-\tfrac{1}{2}\rho u^2\big ) \big ),$
while
$$J=\begin{pmatrix}
u & \rho &0\\
0& u & \frac{1}{\rho}\\
0& (e+p) & u
\end{pmatrix}
.$$ For \eqref{eq:entropy}, 
$$\mathcal{D}_\bu=\{(p,u,s)\in \R^3, p>0 \}.$$ 
More generally, we have $J=\big [\nabla_\bu\big (\Psi^{-1}\big )\big ]\nabla_\bu \bbf$.

The idea is to discretise \emph{simultaneously} \eqref{eq:1:1} and \eqref{eq:1:3}.
Forgetting the possible boundary conditions, $\R$ is divided into non overlapping intervals $K_{j+1/2}=[x_{j}, x_{j+1}]$ where $x_j<x_{j+1}$ for all $j\in \Z$. 
We set $\Delta_{j+1/2}=x_{j+1}-x_j$ and $\Delta =\max_j\Delta_{j+1/2}$. At the grid points, we will estimate $\bv_j$ in time, while in the cells we will estimate the average value $$\bar \bu_{j+1/2}=\frac{1}{\Delta_{j+1/2}}\int_{x_j}^{x_{j+1}} \bu(x) \; dx$$
 When needed, we have $\bu_j=\Psi^{-1}(\bv_j)$, however $\bar \bv_{j+1/2}=\Psi(\bar \bu_{j+1/2})$ is meaningless since the $\Psi$ does not commute with the average.
 
In $K_{j+1/2}$ any continuous function can be represented by $\bu_j=\bu(x_j)$, $\bu_{j+1}=\bu(x_{j+1})$ and $\bar \bu_{j+1/2}$: one can consider the polynomial $R_\bu$ defined on $K_{j+1/2}$ by
$$\big (R_\bu\big )_{|K_{j+1/2}}(x)=\bu_j L_{j+1/2}^0+\bu_{j+1}L_{j+1/2}^1+\bar \bu_{j+1/2} L_{j+1/2}^{1/2},$$
with
$$L_{j+1/2}^\xi(x)=\ell_\xi\big (\frac{x-x_{j}}{x_{j+1}-x_j}\big )$$
and
$$\ell_0(s)=(1-s)(1-3s), \quad \ell_1(s)=s(3s-2), \qquad \ell_{1/2}(x)=6s(1-s).$$
We see that 
\begin{equation*}
\begin{split}
\ell_0(0)=1, & \quad \ell_0(1)=0, \quad \int_0^1 \ell_0(s) ds=0\\
\ell_1(1)=1, &  \quad \ell_1(0)=0, \quad \int_0^1\ell_1(s) ds=0\\
  \ell_{1/2}(0)=0,& \quad \ell_{1/2}(1)=0, \quad \int_0^1\ell_{1/2}(s)\; ds=1.
\end{split}
\end{equation*}

\bigskip
How to  evolve $\bar \bu_{j+1/2}$ following \eqref{eq:1:1} and $v_{j}$ following \eqref{eq:1:3} in time?
The solution is simple for the average value: since
$$
\Delta_{j+1/2}  \; \dfrac{d\bar \bu_{j+1/2}}{dt}+  \bbf(\bu_{j+1}(t))-\bbf(\bu_{j}(t))  =0,$$
we simply take
\begin{subequations}
\label{scheme}
\begin{equation}
\label{scheme:1}
\Delta_{j+1/2}\dfrac{d\bar \bu_{j+1/2}}{dt}+\big ( \hf_{j+1/2}-\hf_{j-1/2}\big )=0
\end{equation}
where $\hf_{j+1/2}$ is a \remi{consistent} numerical flux that depends continuously of its arguments. In practice, \remi{since the approximation is continuous}, we take
\begin{equation}
\label{scheme:2}
\hf_{j+1/2}=\bbf(\bu_j)\remi{=\bbf\big (\Psi^{-1}(\bv_j)\big ).}
\end{equation}
For $\bv$, we assume a semi-discrete scheme of the following form:
{\begin{equation}
\label{scheme:3}
 \dfrac{d\bv_j}{dt}+ \overleftarrow{\Phi}^\bv_{j+1/2}+\overrightarrow{\Phi}^\bv_{j-1/2}=0
\end{equation} }
such that\remi{ $\overleftarrow{\Phi}^\bv_{j+1/2}+\overrightarrow{\Phi}^\bv_{j+1/2}$} is a consistent approximation of \remI{$J\dpar{\bv}{x}$ in $K_{j+1/2}$}. We will give examples later, for now we only describe the principles. In general the residuals $\overleftarrow{\Phi}^\bv_{j+1/2}$ and $\overrightarrow{\Phi}^\bv_{j-1/2}$ need to depend on some $\bv_l$ and $\bv_{l+1/2}\approx \bv(x_{l+1/2})$. We can recover the missing informations at the half points in two steps:
\begin{enumerate}
\item From $\bv_j$, we can get $\bu_j=\Psi(\bv_j)$,
\item Then in $[x_{j}, x_{j+1}]$ we approximate $\bu$ by
$$R_\bu(x)=\bu_j\ell_0\big ( \frac{x-x_j}{\Delta_{j+1/2}}\big ) +\bar \bu_{j+1/2}\ell_{1/2}\big ( \frac{x-x_j}{\Delta_{j+1/2}}\big )+\bu_{j+1}
\ell_1\big ( \frac{x-x_j}{\Delta_{j+1/2}}\big ),$$
which enable to provide $\bu_{j+1/2}:=R_\bu(x_{j+1/2})$, i.e
\begin{equation}
\label{scheme:4} \bu_{j+1/2}=\frac{3}{2} \bar \bu_{j+1/2}-\frac{\bu_j+\bu_{j+1}}{4}.
\end{equation}
Note that this relation is simply
$\bar \bu_{j+1/2}=\tfrac{1}{6}\big ( \bu_j+\bu_{j+1}+4 \bu_{j+1/2}\big ),$
i.e. Simpson's formula.
\item Finally, we state
$$\bv_{j+1/2}=\Psi^{-1}(R_\bu(x_{j+1/2}))$$
\end{enumerate}
In some situations, described later, we will also make the approximation:
$$\bv_{j+1/2}=\Psi^{-1}(\bar \bu_{j+1/2})$$
which is nevertheless consistent (but only first order accurate).
\end{subequations}
As written above, the fluctuations $\overleftarrow{\Phi}^\bv_{j+1/2}$ and \remiIII{$\overrightarrow{\Phi}^\bv_{j+1/2}$} are functionals of the form $\Phi\big ( \{\bv_l, \bv_{l+1/2}\}, j-p\leq l\leq j+p)$ for some fixed value of $p$. We will make the following assumptions:
\begin{subequations}\label{consistency}
\begin{enumerate}
\item Lipschiz continuity:
There exists $C$ that depends only on $\bu^0$ and $T$ such that for any $j\in \Z$
\begin{equation}\label{consistency:lipschitz}
\Vert \Phi\big ( \{\bv_l, \bv_{l+1/2}\}, j-p\leq l\leq j+p)\Vert \leq \frac{C}{\Delta_{j+1/2}}\bigg ( \sum_{l=-p}^p \Vert \bv_l-\bv_{l+1/2}\Vert \bigg ),\end{equation}
\item Consistency. Setting $\bv^h=R_\bu$, 
\begin{equation}\label{consistency:consistency}\sum_{j\in \Z}\int_{K_{j+1/2}} \Vert \overleftarrow{\Phi}^\bv_{j+1/2}+\overrightarrow{\Phi}^\bv_{j+1/2}-J\dpar{\bv^h}{x}\Vert \; d\bx \leq C \; \Delta .\end{equation}
\item Regular mesh: the meshes are regular in the \remiIII{finite element sense}.
\end{enumerate}
\end{subequations}
The ODE systems \eqref{scheme} are integrated by a standard ODE solver. We will choose the Euler forward method, and the second order and third order SSP Runge-Kutta scheme.

\subsection{Analysis of the method}
In order to explain why the method can work, we will choose the simplest ODE integrator, namely the Euler forward method. The general case can be done in the same way, with more technical details.
So we integrate \eqref{scheme} by:
\begin{equation}
\label{schemedisc:1}
\bar \bu_{j+1/2}^{n+1}=\bar \bu_{j+1/2}^n-\frac{\Delta t_n}{\Delta_{j+1/2}} \big (\underbrace{ \bbf(\bu_{j+1}^n)-\bbf(\bu_j^n)}_{:=\delta_{j+1/2} \bbf}\big ),
\end{equation}
and
\begin{equation}
\label{schemedisc:2}
\bv_j^{n+1}=\bv_j^n -{\Delta t_n }\big ( \overleftarrow{\Phi}^\bv_{j+1/2}+\overrightarrow{\Phi}^\bv_{j-1/2}\big )
\end{equation}
Setting $\Delta_j$ as the average of $\Delta_{j+1/2}$ and $\Delta_{j-1/2}$, we rewrite \eqref{schemedisc:2} as 
\begin{equation}
\label{schemedisc:3:0}
\bv_j^{n+1}=\bv_j^n -\frac{\Delta t_n }{\Delta_j}\delta_x\bv_j
\end{equation}
and we note that, using the assumption \eqref{consistency:lipschitz} as well as the fact that the mesh is shape regular, that there exists $C>0$ depending only on $\bu^0$ and $T$ such that
$$\remiIII{\Vert \delta_x \bv_j}\Vert \leq C \sum_{j=p-1}^{p+1} \Vert \bv_j-\bv_{j+1/2}\Vert .$$
{
Using the transformation \eqref{scheme:4}, from \eqref{schemedisc:2}, we can evaluate 
$\bu_j^{n+1}=\Psi(\bv_j^{n+1})$, \remi{and then write the update of $\bu$ as}
\begin{equation}\label{schemedisc:3}\remi{\Delta_j \big ( \bu_j^{n+1}-\bu_j^n\big ) +\Delta t_n  \delta_x\bu_{j}=0}\end{equation}
\remi{where }
$$\delta_x\bu_{j}=\dfrac{\Delta_j}{\Delta t_n}\bigg ( \Psi(\bv_j^n-\frac{\Delta t_n}{\Delta_j} \delta \bv_j)-\Psi(\bv_j^n) \bigg ), $$
which, thanks to the assumptions we have made on $\Psi$ satisfies
$$\Vert \delta_x \bu_{j+1/2}\Vert \leq C \Vert \delta_x \bv_j\Vert\leq C\sum_{j=-p}^{p}\Vert \bv_{j+l}-\bv_{j+l+1}\Vert \leq C \sum_{l=-p}^{p}\Vert \bu_{j+l}-\bu_{j+1+l}\Vert$$ for some constants that depends on the gradient of $\Psi$ and the  maximum of the $\bv_i^n$ for $i\in \Z$.}

{ To explain the validity of the approximation,  we start by the Simpson formula, which is exact for quadratic polynomials:
$$
\int_{x_{j}}^{x_{j+1}} f(x) \; dx\approx \dfrac{\Delta_{j+1/2}}{6}\big ( f(x_j)+4 f(x_{j+1/2})+f(x_{j+1})\big ).$$
From the point values $\bu_j$, $\bu_{j+1}$ and $\bu_{j+1/2}$ at times $t_n$ and $t_{n+1}$, we define the quadratic Lagrange interpolant $R_{\bu^n}$ and $R_{\bu^{n+1}}$
and then write
$$\int_{x_{j}}^{x_{j+1}} \remiIII{ \varphi(x,t)} \big ( R_{\bu^{n+1}}-R_{\bu^{n}}\big ) \; dx\approx
\dfrac{\Delta_{j+1/2}}{6}\big (\varphi_{j+1}(\bu_{j+1}^{n+1}-\bu_{j+1}^n)+4\varphi_{j+1/2} (\bu_{j+1/2}^{n+1}-\bu_{j+1/2}^n)+ \varphi_{j}(\bu_{j}^{n+1}-\bu_j^n)\bigg ).$$
Accuracy is not an issue here.
Using \eqref{schemedisc:2} and \eqref{schemedisc:3}, setting $\delta_j^{n+1/2}\bu=\bu_j^{n+1}-\bu_j^n$, we get
\begin{equation*}
\begin{split}
\Sigma:=\sum\limits _{[x_j, x_{j+1}], j\in \Z}\dfrac{\Delta_{j+1/2}}{6}\bigg (& \varphi_{j+1}^n \delta_j^{n+1/2} \bu+4\varphi_{j+1/2}^n \delta_{j+1/2}^{n+1/2}\bu+
 \varphi_{j}^n\delta_j^{n+1/2}\bu\bigg )\\ &=
\sum\limits _{[x_j, x_{j+1}], j\in \Z}\dfrac{\Delta_{j+1/2}}{6}\bigg (\varphi_{j+1}^n \delta_{j+1}^{n+1/2}\bu\\
&\qquad \qquad \qquad \qquad + 4\varphi_{j+1/2}^n\big ( \frac{3}{2} \delta_{j+1/2}^{n+1/2} \bar\bu-\frac{\delta_{j+1}^{n+1/2} \bu+\delta_j^{n+1/2} \bu}{4} \big )+ \varphi_{j}^n \delta_{j}^{n+1/2}\bu\bigg )\\
&=\sum\limits _{[x_j, x_{j+1}], j\in \Z}\Delta_{j+1/2} \varphi_{j+1/2}^n\delta_{j+1/2}^{n+1/2}\bar\bu \\
&\qquad \qquad \qquad \qquad+ \underbrace{\sum\limits_{j\in \Z} \frac{\delta_{j}^{n+1/2} \bu}{6}
\bigg \{\Delta_{j+1/2}\big ( \varphi_j-\varphi_{j+1/2}\big ) +\Delta_{j-1/2}\big (\varphi_j-\varphi_{j-1/2}\big )\bigg \}}_{S_n}.
\end{split}
\end{equation*}
so that we get, using \eqref{schemedisc:3}
\begin{equation}
\label{Master}
\begin{split}
\sum\limits_{n\in \N}\sum\limits _{[x_j, x_{j+1}], j\in \Z}\dfrac{\Delta_{j+1/2}}{6}\bigg (& \varphi_{j+1}^n \delta_j^{n+1/2} \bu+4\varphi_{j+1/2}^n \delta_{j+1/2}^{n+1/2}\bu+
 \varphi_{j}^n\delta_j^{n+1/2}\bu\bigg )\remiIII{-}\sum\limits_{n\in \N} \Delta t_n\sum\limits _{[x_j, x_{j+1}], j\in \Z}
 \varphi_{j+1/2}^n\delta_{j+1/2}\bbf\\
 &-\sum\limits_{n\in \N} S_n=0
 \end{split}
 \end{equation}
Then we  use again \eqref{schemedisc:3}, use the  fact that the mesh is regular, and observe that
$$S_n= \Delta t_n \sum_j \Delta_j \; \delta_x\bu_j+O(\Delta^3).$$  In appendix \ref{appendix}, we will show that in the limit, the contribution of the $S_n$ term  will converges towards  $0$, while the first term of \eqref{Master} will converge to 
$$\int_0^{+\infty}\int_\R \dpar{\varphi}{t}\bu \; dx dt-\int_\R \bu_0 \; dx$$ while the second term will converge towards
$$\int_0^{+\infty}\int_{\R} \dpar{\varphi}{x} \bbf(\bu)\; d\bx.$$ 

This will be shown, 
 using classical arguments, in  the appendix \ref{appendix}, so that  we have   }
\begin{proposition}\label{LxW} We assume that the mesh is regular: there exists $\alpha$ and $\beta$ such that $\alpha\leq \Delta_{j+1/2}/\Delta_{j-1/2}\leq \beta$.
If $\max\limits_{j\in \Z} \Vert \bu_j^n\Vert_\infty$ and $\max\limits_{j\in \Z} \Vert v_{j+1/2}^n\Vert_\infty$ are bounded, and if a subsequence of $\bu_\Delta$ converges in $L^1$ towards $\bu$, then $\bu$ is a weak solution of the \remiIII{problem}.
\end{proposition} 
\begin{remark}
\remI{Indeed, the definition of a precise $\Delta_j$ is not really needed, and we come back to this in the next section.} \remI{What is needed is a spatial scale that relates the updates in $\bv$ and $\bu$ in an incremental form of the finite difference type. This is why the asssumption of mesh regularity is fundemental.}
\end{remark}

\section{Some examples of discretisation}
%
%
We list  possible choices: for $\dpar{\bv}{t}+J\dpar{\bv}{x}=0,$ where $J$ is the Jacobian of $\bbf$ with respect to $\bu$; they have been used in the numerical tests.  The question here is to define
$\overleftarrow{\Phi}_{j+1/2}$ and $\overrightarrow{\Phi}_{j+1/2}$ that are the contributions of \remi{$K_{j\pm 1/2}$} to $J\dpar{\bv}{x}$ so that
$$J\dpar{\bv}{x}(x_i)\approx \overleftarrow{\Phi}^\bv_{j+1/2}+\remi{\overrightarrow{\Phi}^\bv_{j-1/2}}.$$ 
We follow the work of Iserle \cite{iserle} who gives all the possible schemes that guaranty a stable (in $L^2$) semi-discretisation of the convection equation, for a regular grid which we assume. The only difference in his notations and ours is that the grid on which are defined the approximation of the derivative is made of  the mesh points $x_j$ and the half points $x_{j+1/2}$.

{The first list of examples have an upwind flavour:
\begin{equation}\label{fluctuation}
\overleftarrow{\Phi}^\bv_{j+1/2}=\big (\remi{ J(\bv_j)}\big )^- \frac{\delta^-_{j}\bv}{\Delta_{j+1/2}/2} \text{ and } 
\overrightarrow{\Phi}^\bv_{j+1/2}=\big ( J(\bv_{j+1})\big )^+\frac{\delta_{j+1}^+\bv}{\Delta_{j+1/2}/2}
\end{equation}}
where $\delta^\pm_j $ is an approximation of $\Delta_{j+1/2}  \dpar{v}{x}$ obtained from \cite{iserle}\footnote{\remi{The author works on $\dpar{u}{t}=-\dpar{u}{x}$ which is a bit confusing w.r.t. to "modern" habits. It is true that British drive left.}}:
\begin{itemize}
\item First order approximation: we take 
\begin{equation}\label{S1}\delta_j^+\bv=\bv_{j}-\bv_{j-1/2}, \qquad \delta_j^-\bv=\bv_{j+1/2}-\bv_j.\end{equation}
\item Second order: we take
\begin{equation}\label{S2bis}
\begin{split}
\delta_j^-\bv&=-\frac{3}{2}\bv_j+2\bv_{j+1/2}-\frac{\bv_{j+1}}{2}\\
\delta_j^+\bv&=\frac{\bv_{j-1}}{2}-2\bv_{j-1/2}+\frac{3}{2}\bv_{j}
\end{split}
\end{equation}
\item \remi{Third order}:  We take
\begin{equation}\label{S2}
\begin{split}
\remiIII{\delta_j^-=-\frac{v_{i+1}}{6}+v_{i+1/2}-\frac{v_i}{2}-\frac{v_{i-1/2}}{3}},\\
\remiIII{ \delta_j^+=\frac{v_{i-1}}{6}-v_{i-1/2}+\frac{v_i}{2}+\frac{v_{i+1/2}}{3}}
\end{split}
\end{equation}
\item \remi{Fourth order}: The fully centered scheme would be
$$\delta_j^\pm \bv=\remi{\dfrac{\bv_{j+1}-\bv_{j-1}}{12}+2\dfrac{\bv_{j+1/2}-\bv_{j-1/2}}{3}}$$
but we prefer
\begin{equation}\label{S3}
\begin{split}
\delta_j^- \bv&=\frac{\bv_{j-1/2}}{4}+\frac{5}{6}\bv_j-\frac{3}{2}\bv_{j+1/2}+\frac{\bv_{j+1}}{2}-\frac{\bv_{j+3/2}}{12}\\
\delta_j^+\bv&=\frac{\bv_{j+1/2}}{4}+\frac{5}{6}\bv_j-\frac{3}{2}\bv_{j-1/2}+\frac{1}{2}\bv_{j-1}-\frac{\bv_{j-3/2}}{12}
\end{split}
\end{equation}
\item Etc\ldots
\end{itemize}
It can be useful to have more dissipative versions of  a first order scheme. We take:
$$
\bigg (J\dpar{v}{x}\bigg )_j=\overleftarrow{\Phi}_{j+1/2}+\overrightarrow{\Phi}_{j-1/2}$$ with
\begin{equation*}
\begin{split}
\frac{\Delta_{j+1/2}}{2}\overleftarrow{\Phi}_{j+1/2}&=\frac{1}{2} \widehat {J\dpar{\bv}{x}}_j+\alpha \big (\bv_j-\frac{\bv_j+\bv_{j+1/2}}{2}\big )\\
\frac{\Delta_{j+1/2}}{2}\overrightarrow{\Phi}_{j+1/2}&=\frac{1}{2} \widehat {J\dpar{\bv}{x}}_{j+1}+\alpha \big (\bv_{j+1}-\frac{\bv_{j+1}+\bv_{j+1/2}}{2}\big )
\end{split}
\end{equation*}
where 
$\widehat {J\dpar{v}{x}}_l$ is a consistent approximation of $J\dpar{u}{x}$ at $x_l$ and $\alpha$ is an upper-bound of the spectral radius of $J(\bv_j)$, $J(\bv_{j+1/2})$ and $J(\bv_{j+1})$. We take, for simplicity,  $\bv_{j+1/2}=\Psi^{-1}(\bar \bu_{j+1/2})$. For the model \eqref{eq:primitive}, we take
$$\frac{\Delta_{j+1/2}}{2}\widehat {J\dpar{v}{x}}_j=\begin{pmatrix}
(\rho u)_{j+1/2}-(\rho u)_j\\
\frac{1}{2} \big ( u_{j+1/2}^2-u_j^2) + \frac{1}{\tilde{\rho}_{j+1/2}}\big ( p_{j+1/2}-p_j\big )\\
\tilde{u}_{j+1/2}\big (p_{j+1/2}-p_j)+\tilde{\rho c^2}\big (u_{j+1/2}-u_j\big )
\end{pmatrix}$$ where $\tilde{\rho}_{j+1/2}$ is the geometric average of $\rho_j$ and $\rho_{j+1/2}$, $\tilde{u}_{j+1/2}$ is the arithmetic average of $u_j$ and $u_{j+1/2}$, while $\widetilde{\rho c^2}_{j+1/2}=\gamma \frac{p_j+p_{j+1/2}}{2}$.
For the model \eqref{eq:entropy}, we take:
$$\frac{\Delta_{j+1/2}}{2}\widehat {J\dpar{v}{x}}_j=\begin{pmatrix}
\tilde{u}_{j+1/2} \big ( s_{j+1/2}-s_j\big )\\
\frac{1}{2} \big ( u_{j+1/2}^2-u_j^2) + \frac{1}{\tilde{\rho}_{j+1/2}}\big ( p_{j+1/2}-p_j\big )\\
\tilde{u}_{j+1/2}\big (p_{j+1/2}-p_j)+\widetilde{\rho c^2}\big (u_{j+1/2}-u_j\big )
\end{pmatrix}
.$$
All this has a Local Lax-Friedrichs' flavour, and seems to be positivity preserving for the velocity and the pressure.

Using this, the method is :
\begin{subequations}\label{method}
\begin{equation}
\label{methode:1}
\begin{split}
 \dfrac{d\bv_j}{dt}+\overleftarrow{\Phi}_{j+1/2}^\bv+\overrightarrow{\Phi}_{j-1/2}^\bv=0
\end{split}
\end{equation}
combined with 
\begin{equation}
\label{methode:2}
\Delta x\dfrac{d\bar \bu_{j+1/2}}{dt}+{\bbf(\bu_{j+1})-\bbf(\bu_j)}=0.
\end{equation}
\end{subequations}
\remi{We see in \eqref{methode:1} that the time derivative of $\bv$ is obtained by adding two fluctuations, one computed for the interval $K_{j+1/2}=[x_j,x_{j+1}]$ and one for the interval $K_{j-1/2}=[x_{j-1},x_j]$. These fluctuations are obtained from \eqref{fluctuation} with the increments in $\bv$ defined by \eqref{S1}, \eqref{S2}, \eqref{S3}, etc. In the sequel, we denote the  scheme applied on the interval $K_{j+1/2}$  by $S_{j+1/2}(k)$ where the average are integrated by \eqref{methode:2} and $\bv$ by \eqref{methode:1} with the fluctuations \eqref{S1} for $k=1$, \eqref{S2bis} for $k=2$ and \eqref{S2} for $k=3$, etc. To make sure that the first order scheme is positivity preserving (at least experimentaly), we may also consider the case denoted by $k=0$ where $S_{j+1/2}(0)$ is the local Lax Friedrichs scheme defined above. Both fluctuation \eqref{S1} and the local Lax Friedrichs scheme are first order accurate, but the second one is quite dissipative but positivity preserving while the scheme \eqref{S1} is not (experimentaly) positivity preserving.}
\remi{The system \eqref{method} is integrated in time by a Runge-Kutta solver: RK1, RK SSP2 and RK SSP3.  }

\subsection{Error analysis in the scalar case}
Here, the mesh is uniform, so that $\Delta_{j+1/2}=\Delta $ for any $j\in \Z$. 
It is easy to check the consistency, and on figure \ref{fig:error:linear} we show the $L^1$ error on $u$ and $\bar u$ for \eqref{method} with SSPKR2 and SSPRK3 (CFL=$0.4$) for a convection problem 
$$\dpar{u}{t}+\dpar{u}{x}=0$$ with periodic boundary conditions and the initial condition $u_0=\cos(2\pi x)$.
\begin{figure}[h]
\begin{center}
\includegraphics[width=0.7\textwidth]{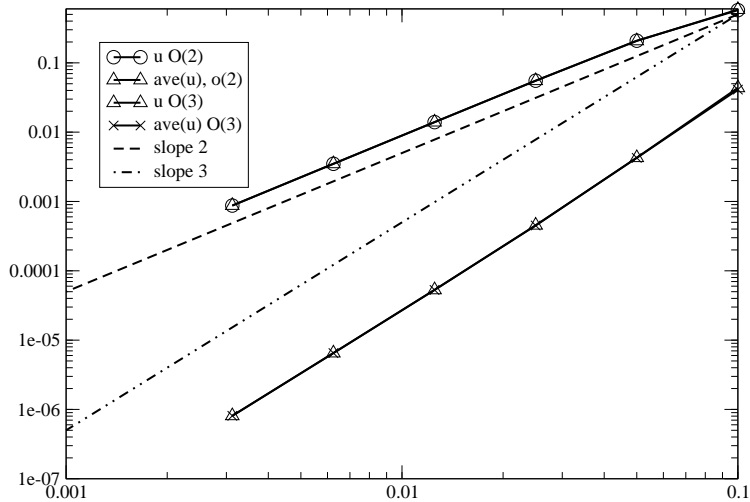}
\end{center}
\caption{\label{fig:error:linear} Error plot for $u$ and $\bar u$ for \eqref{method} with SSPKR2 and SSPRK3 (cfl=$0.4$). Here $f(u)=u$. The second order results are obtained with SSPRK2 with \eqref{methode:1}-\eqref{methode:2}, the third order results is obtained by  \eqref{methode:1}-\eqref{methode:2}.}
\end{figure}
{\begin{remark}[Linear stability]
In the appendix \ref{appendix:linearstability}, we perform the $L^2$ linear stability and we get, with $\lambda=\tfrac{\Delta t}{\Delta }$,
\begin{itemize}
\item First order scheme, $\vert \lambda\vert\leq 0.92$,
\item Second order scheme, $\vert\lambda\vert\leq 0.6$,
\item Third order scheme, $\vert\lambda\vert\leq 0.5$.
\end{itemize}
\end{remark} }

We also have run this scheme for the Burgers equation, and compared it with a standard finite volume (with local Lax-Friedrichs).  The conservative form of  the PDE is used for the average, and the non conservative one for the point values: $J=u$ and $\psi(u)=u$. This is an experimental check of conservation.
The initial condition is $$\bu_0(x)=\sin(2\pi x)+\frac{1}{2}$$ on $[0,1]$, so that there is a moving shock.
\begin{figure}[h]
\begin{center}
\subfigure[]{\includegraphics[width=0.45\textwidth]{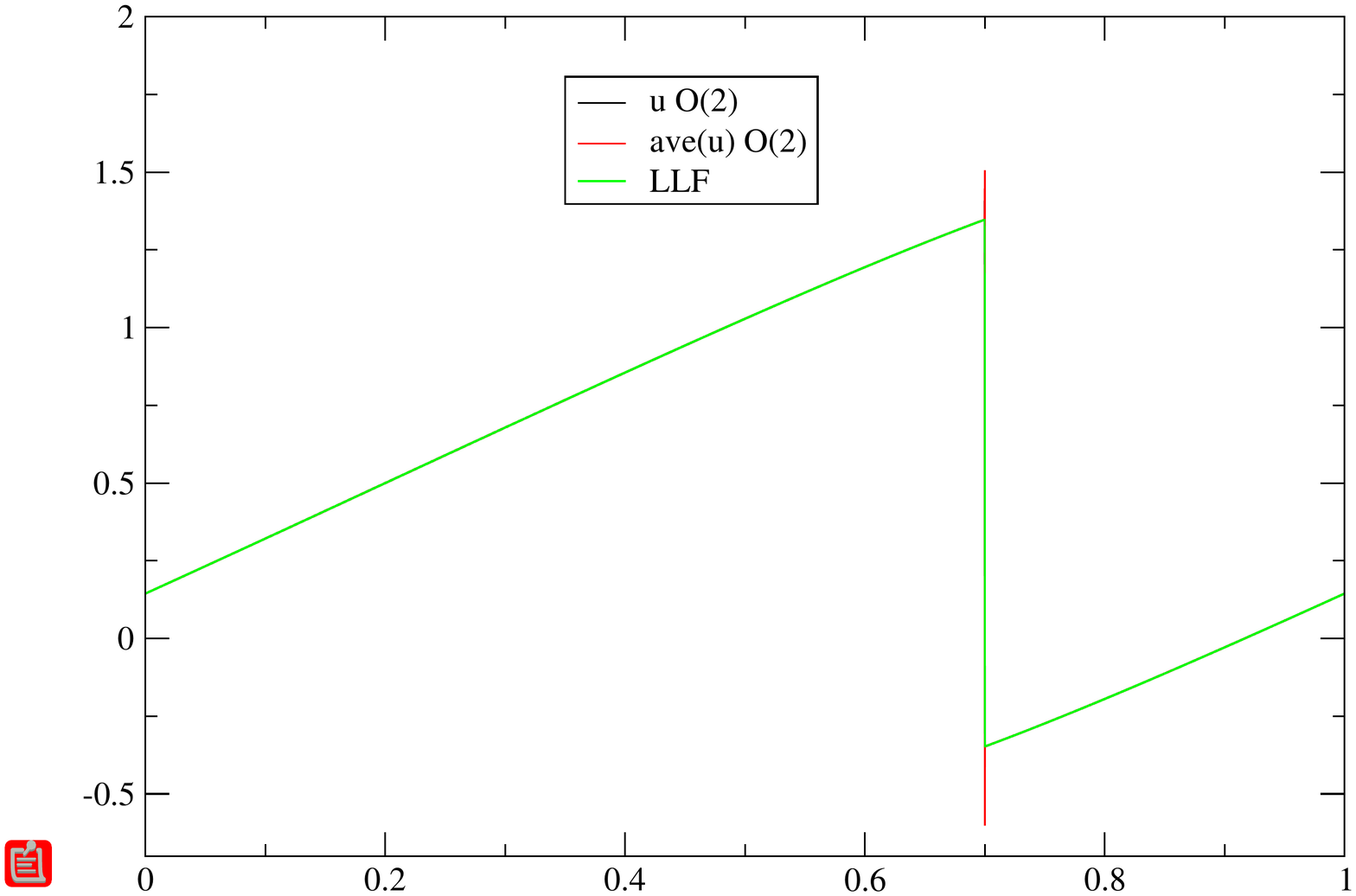}}
\subfigure[]{\includegraphics[width=0.45\textwidth]{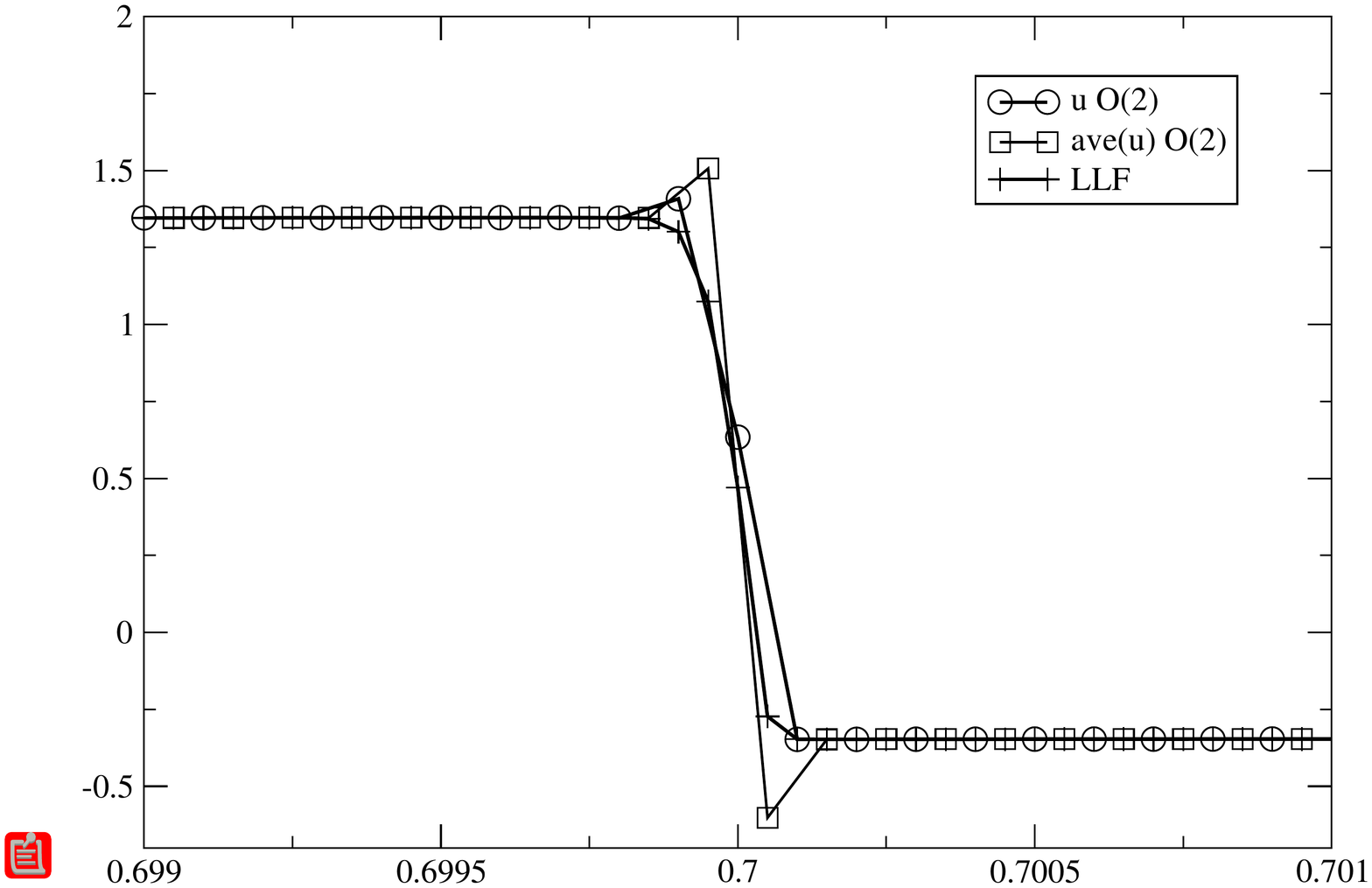}}
\subfigure[]{\includegraphics[width=0.45\textwidth]{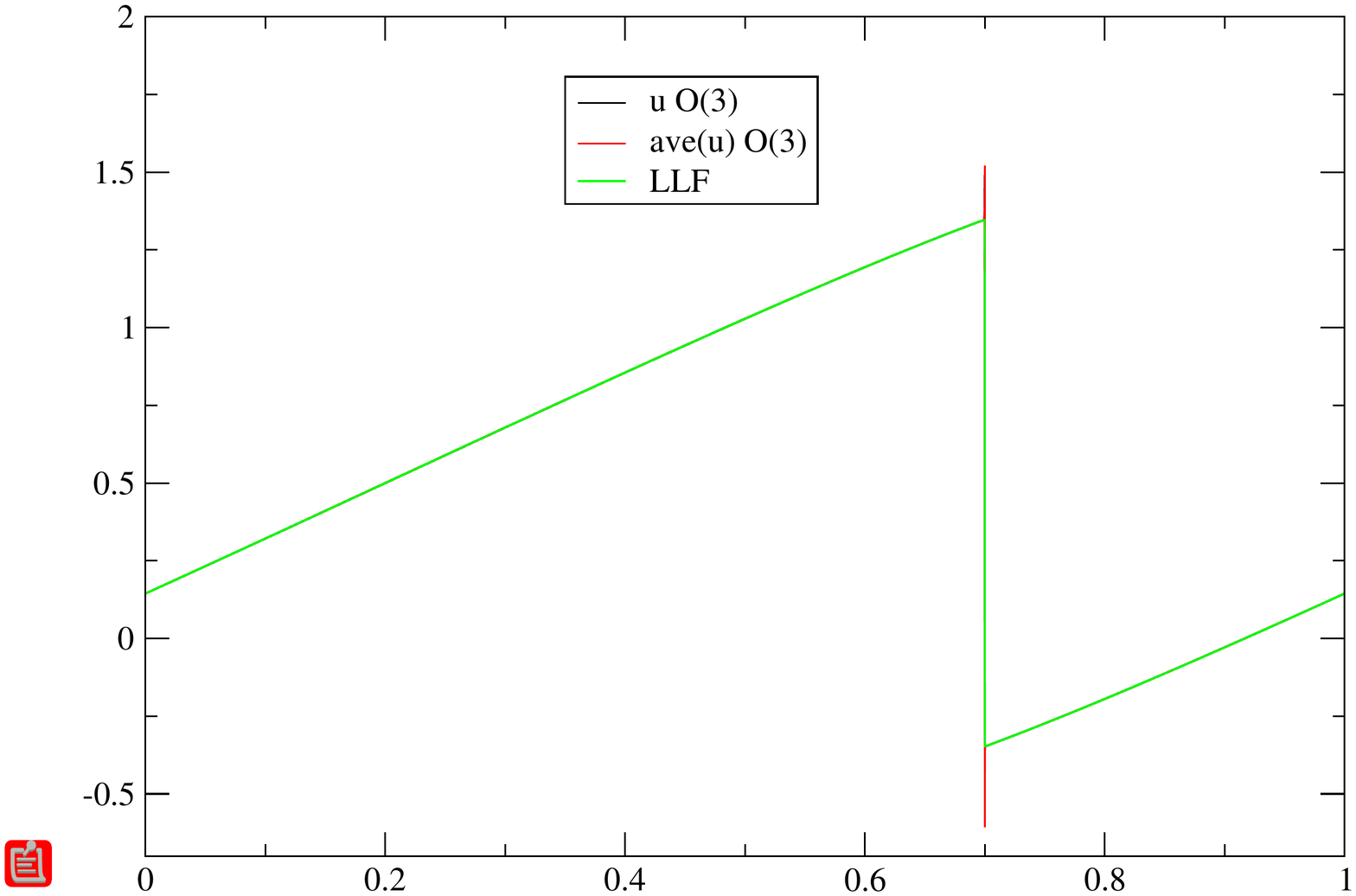}}
\subfigure[]{\includegraphics[width=0.45\textwidth]{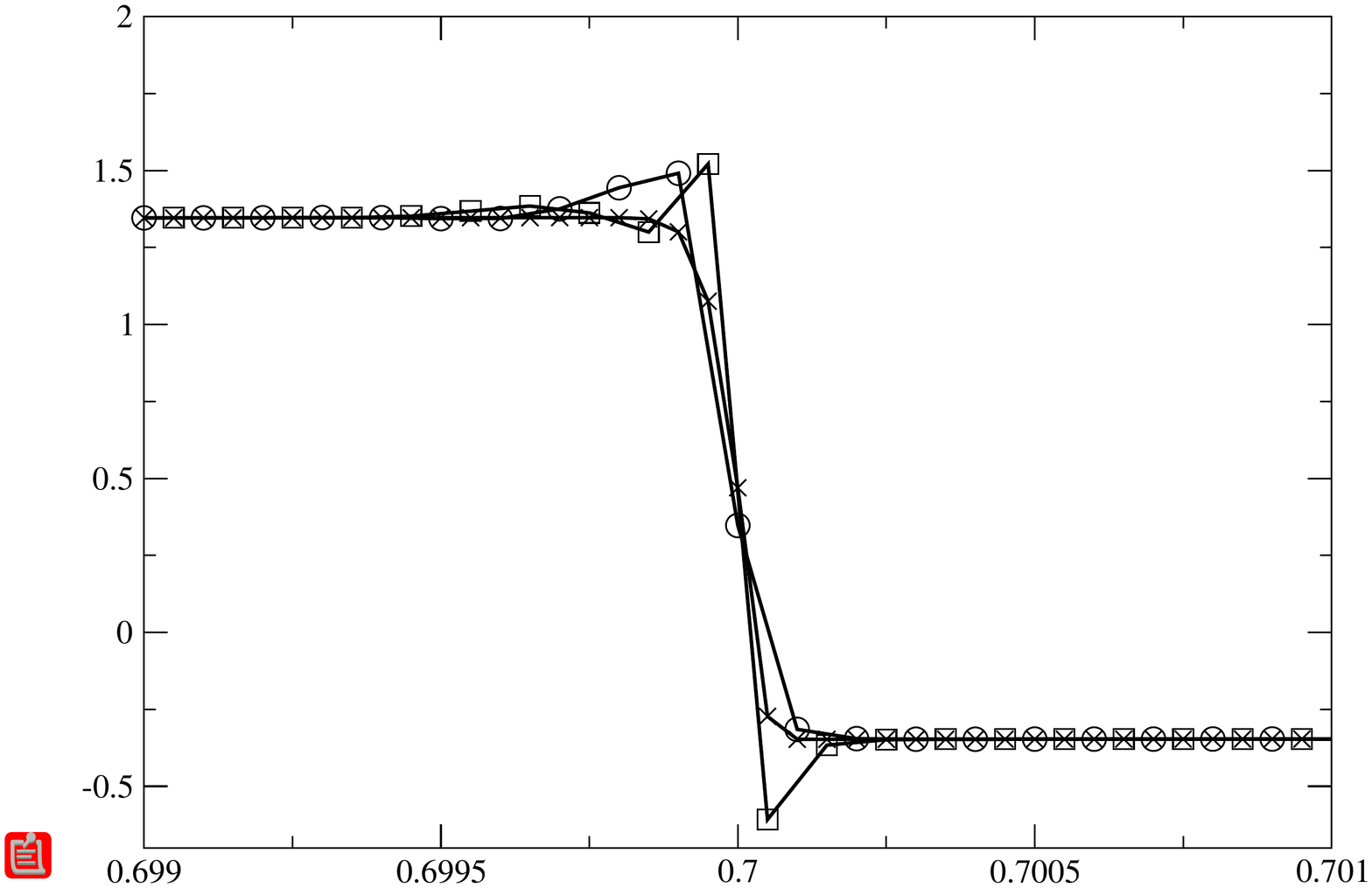}}
\caption{\label{conservationtest} Solution of Burgers with $10\, 000$ points, $t_{fin}=0.4$, $CFL=0.4$ for the second order (a and b) (\eqref{methode:1}-\eqref{methode:2} with SSPRK2) and third order (c and d) (\eqref{methode:1}-\eqref{methode:2} with SSPRK3). The global solution is represented in (a) and (c), and a zoom around the discontinuity is shown in (b) and (c).}
\end{center}
\end{figure}
We can see that the agreement is excellent and that the numerical solution behaves as expected.

\subsection{Non linear stability}\label{sec:mood}
As such, the scheme is at most linearly stable, with a CFL condition based on the fine grid. However, in case of discontinuities or the occurence of gradients that are not resolved by the grid, we have to face oscillations, as usual.

\remi{In order to get high order oscillation free results, a natural option would be to extend the MUSCL approach to the present context. However, it is not very clear how to proceed, } so we have relied on the MOOD paradigm \cite{Mood1,Vilar}. The idea is to work with several schemes ranging from order $p$ to $1$, with the lowest order one able to provide results with positive density and pressure.\remi{ These schemes are the  $S_{j+1/2}(k), k=1, \ldots ,3$ scheme defined above. They }are assumed to work for a given CFL range, and the algorithm is as follows: For each Runge-Kutta sub-step, starting from $U^n=\{\bar \bu_{j+1/2}^n, \bar \bv_j^n\}_{j\in \Z}$, we compute
\begin{equation}
\label{method_mood}
\begin{split}
\tilde{\bar \bu}_{j+1/2}^{n+1}&=\bar \bu_{j+1/2}^n -\lambda_n\big ( \bbf (\bu_{j+1}^n)-\bbf(\bu_j^n)\big ), \qquad \lambda_n=\frac{\Delta t_n}{\Delta_{j+1/2}}\\
\tilde{\bv}_j^{n+1}&=\tilde{\bv}_j^{n}- 2\Delta t_n\overleftarrow{\Phi}^\bv_{j+1/2}\\
\tilde{\bv}_{j+1}^{n+1}&=\tilde{\bv}_{j+1}^{n}-2\Delta t_n\overrightarrow{\Phi}^\bv_{j+1/2}
\end{split}
\end{equation}
by the scheme $S_{j+1/2}(p)$. Then we test the validity of these results in the interval $[x_{j}, x_{j+1}]$ for the density (and possibly the pressure). This is described a little bit later.  The variable  $\bv$ is updated as in \eqref{method_mood}, because  at $t_{n+1}$, the true update of $\bv_j$ is the half sum of $\tilde{\bv}_j^{n+1}$ and $\tilde{\bv}_{j+1}^{n+1}$.

If the test is positive, then we keep the scheme $S_{j+1/2}(p)$ in that interval, else we start again with $S_{j+1/2}(p-1)$, and repeat the procedure unless all the intervals $K_{j+1/2}$ have successfully passed the test. This is described in Algorithm \ref{algo:mood1} where $\mathcal{S}_{j+1/2}$ is the stencil used in $K_{j+1/2}$.
\begin{algorithm}[H]
\caption{ \label{algo:mood1}Description of the MOOD loop. \remi{The algorithm stops because $\SS_{j+1/2}=0$ corresponds to the local Lax Friedrichs scheme for which the test is always true.}}
\begin{algorithmic}
\REQUIRE $U^n =\{\bar u_{j+1/2}^n, \bar v_j^n\}_{j\in \Z}$
\REQUIRE Allocate $\{\SS_{j+1/2}\}_{j\in \Z}$ an array of integers. It is initialized with $\SS_{j+1/2}=S_{j+1/2}(p)$, the maximum order.
\FOR{$k=p, \ldots, 2$}
\FORALL{For all $K_{j+1/2}$}
\STATE Define $\tilde{\bar \bu}_{j+1/2}^{n+1}$, $\tilde{\bv}_j^{n+1}$ and $\tilde{\bv}_{j+1}^{n+1}$ as in \eqref{method_mood}
\STATE Apply the test on $\tilde{\bar \bu}_{j+1/2}^{n+1}$, $\tilde{\bv}_j^{n+1}$ and $\tilde{\bv}_{j+1}^{n+1}$ :
\IF{test=.true.} 
\STATE $\SS_{j+1/2}= S_{j+1/2}(k-1)$
\ENDIF
\ENDFOR
\ENDFOR
\end{algorithmic}
\end{algorithm}
Now, we describe the tests. We do, in the following order, for each element $K_{j+1/2}$, at the iteration $k>0$ of the loop of \ref{algo:mood1}: the tests are performed on variables evaluated from $\bu$ and $\bv$. For the scalar case, they are simply the point values at $x_j, x_{j+1/2}$ and $x_{j+1}$. For the Euler equations they are the density, and possibly the pressure
\begin{enumerate}
\item We check if all the variables are numbers (i.e. not NaN). If \remi{not}, we state  $\SS_{j+1/2}=S_{j+1/2}(k-1)$,
\item (Only for the Euler equations) We check if the density is positive. We can also request to check if the pressure is also positive. If the variable is negative, the we state that $\SS_{j+1/2}=S_{j+1/2}(k-1)$.
\item Then we check if at $t_n$, the solution was not constant in the numerical stencils of the degrees of freedom in $K_{j+1}$, this in order to avoid to detect a fake maximum principle. We follow the procedure of \cite{Vilar}. if we observe that the solution was locally constant, the $\SS_{j+1/2}$ is not modified.
\item Then we apply a discrete maximum principle, even for systems though it is not very rigorous. For the variable $\xi$ (in practice the density, and we may request to do the same on the pressure), we compute $\min_{j+1/2}\xi$ (resp. $\max_{j+1/2}\xi$) the minimum (resp. maximum) of the values of $\xi$ on $K_{j+1/2}$, $K_{j-1/2}$ and $K_{j+3/2}$. We say we have a \remi{potential }maximum if $\tilde{\xi}^{n+1}\not \in [\min_{j+1/2}\xi^n+\varepsilon_{j+1/2}, 
\max_{j+1/2}\xi^n-\varepsilon_{j+1/2}]$, with $\epsilon_{j+1/2}$ estimated as in \cite{Mood1}. Then:
\begin{itemize}
\item If $\tilde{\xi}^{n+1}\in [\min_{j+1/2}\xi^n+\varepsilon_{j+1/2}, 
\max_{j+1/2}\xi^n-\varepsilon_{j+1/2}]$, $\SS_{j+1/2}$ is not modified
\item Else we use the following procedure introduced in \cite{Vilar}. In each $K_{l+1/2}$, we can evaluate a quadratic polynomial $p_{l+1/2}$ that interpolates $\xi$. Note that its derivative is linear in $\xi$. We compute
$$p'_{j-1/2}(x_j), p'_{j+3/2}(x_{j+1}), p'_{j+1/2}(x_j) \text{ and } p'_{j+1/2}(x_{j+1}).$$
\begin{itemize}
\item If 
$$p'_{j+1/2}(x_j) \in [\min(p'_{j-1/2}(x_j),p'_{j+3/2}(x_{j+1})] \text{ and } p'_{j+1/2}(x_{j+1}) \in [\min(p'_{j-1/2}(x_j),p'_{j+3/2}(x_{j+1})]$$
we say it is a true regular extrema and $\SS_{j+1/2}$ will not be modified,
\item Else the extrema is declared not to be regular, and $\SS_{j+1/2}=S_{j+1/2}(k-1)$
\end{itemize}
\end{itemize}
\end{enumerate}

\medskip
As a first application, to show that the oscillations are well controlled without sacrificing the accuracy, we consider the advection problem (with constant speed unity) on $[0,1]$, periodic boundary conditions with initial condition:
$$u_0(x)=
\left \{ \begin{array}{ll}
0&\text{ if } y\in [-1,-0.8[\\
\frac{1}{6}\big (G(y, \beta, z-\delta)+G(y,\beta,z+\delta)+4G(y,\beta,z) & \text{ if }y\in [-0.8,-0.6]\\
1& \text{ if } y\in [-0.4,-0.2]\\
1-\vert10y-1\vert & \text{ if }y\in [0,0.2]\\
\frac{1}{6}\big (F(y, \beta, z-\delta)+G(y,\beta,z+\delta)+4F(y,\beta,z) & \text{ else,}
\end{array} \right. \text{ with } y=2x-1
$$
Here $a=0.5$, $z=-0.7$, $\delta=0.005$, $\alpha=10$, 
$$\beta=\dfrac{\log 2}{36\delta^2}$$
and
$$G(t,\beta,z)=\exp\big ( -\beta(t-z)^2\big ), \qquad F(t,a,\alpha)=\sqrt{\max\big (0, 1-\alpha(t-a)^2\big )}.$$
Using the MOOD procedure with the third order scheme, the results  obtained for $300$ points for $T=10$ are displayed in figure \ref{ShuJiang}. They look  very reasonable.

\begin{figure}[h]
\begin{center}
\includegraphics[width=0.5\textwidth]{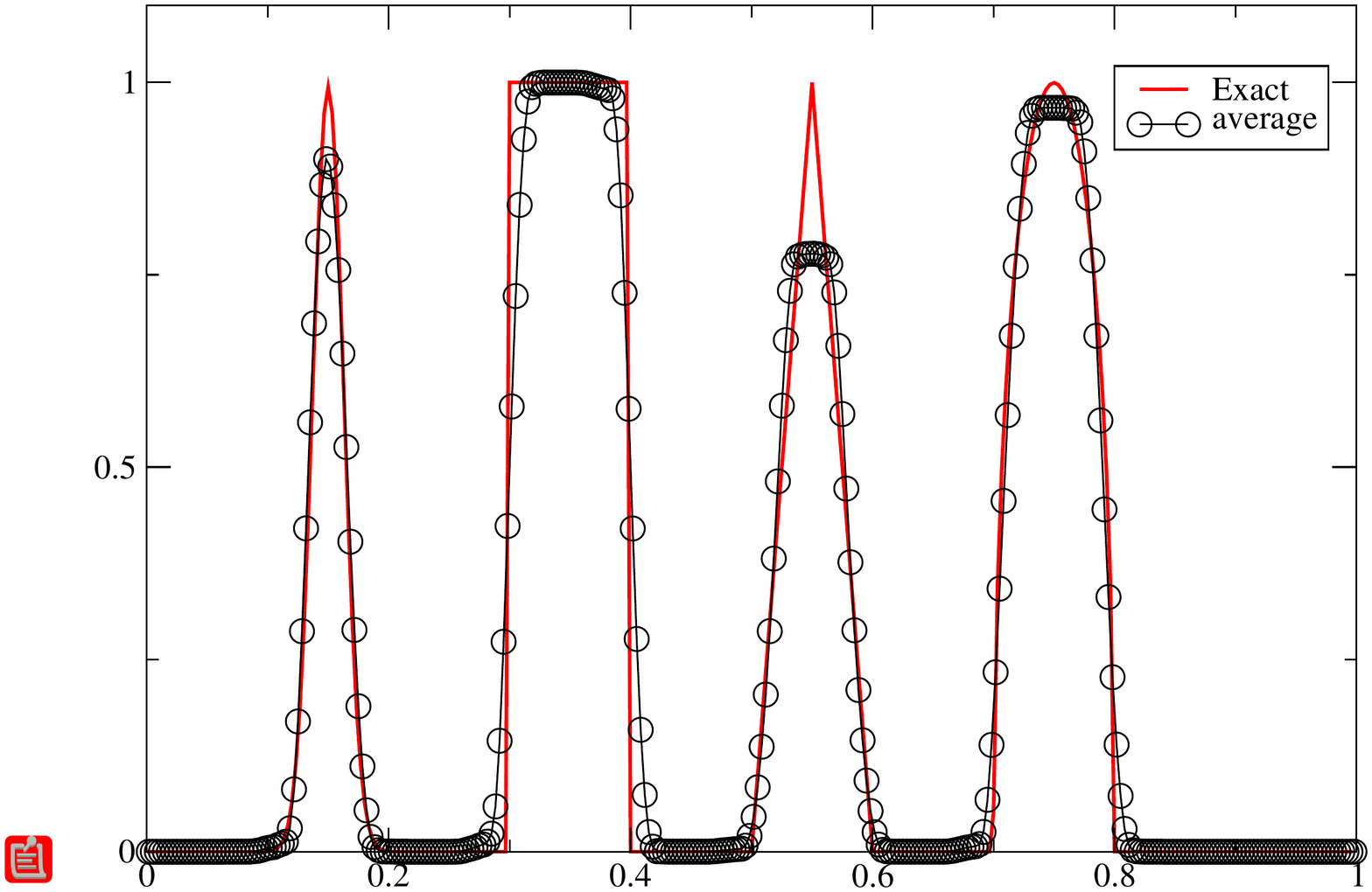}
\end{center}
\caption{\label{ShuJiang} Shu Jiang problem, CFL=0.4, third order scheme with MOOD, 300 points, periodic conditions, 10 periods. The point values and cell average are almost undistinguishable.}
\end{figure}

\section{Numerical results for the Euler equations}
In this section, we show the flexibility of the approach, where conservation is recovered only by the equation \eqref{methode:1}, and so lots of flexibility is possible with the relations on the $\bu_i$. To illustrate this, we consider the Euler equations.  We will consider the conservative formulation \eqref{eq:conservative} for the average value, so $\bu=(\rho, \rho u, E)^T$  and either the form \eqref{eq:primitive}, i.e. $\bv=(\rho, u, p)$ or the form \eqref{eq:entropy} with $\bv=(p,u,s)^T$.

\subsection{Sod test case}
The Sod case is defined for $[0,1]$, the initial condition is
$$(\rho,u,p)^T=\left \{ \begin{array}{ll}
(1,0,1)^T & \text{ for } x<0.5\\
(0.125,0,0.1)^T \text{ else.}
\end{array}\right .
$$
The final time is $T=0.16$. The problem is solved with  \eqref{eq:conservative}-\eqref{eq:primitive} and displayed in figures \ref{fig:sod_100O2}, \ref{fig:sod_100O2Mood}, \ref{fig:sod_100O3} and \ref{fig:sod_100O3Mood}, while the solution obtained with the combination \eqref{eq:conservative}-\eqref{eq:entropy} is shown on figure \ref{fig:sod:entro_100} and \eqref{fig:sod:entro_10000}. When the MOOD procedure is on, it is applied with $\rho$ and $p$ and all the test are performed. 
\begin{figure}[h]
\begin{center}
\subfigure[]{\includegraphics[width=0.45\textwidth]{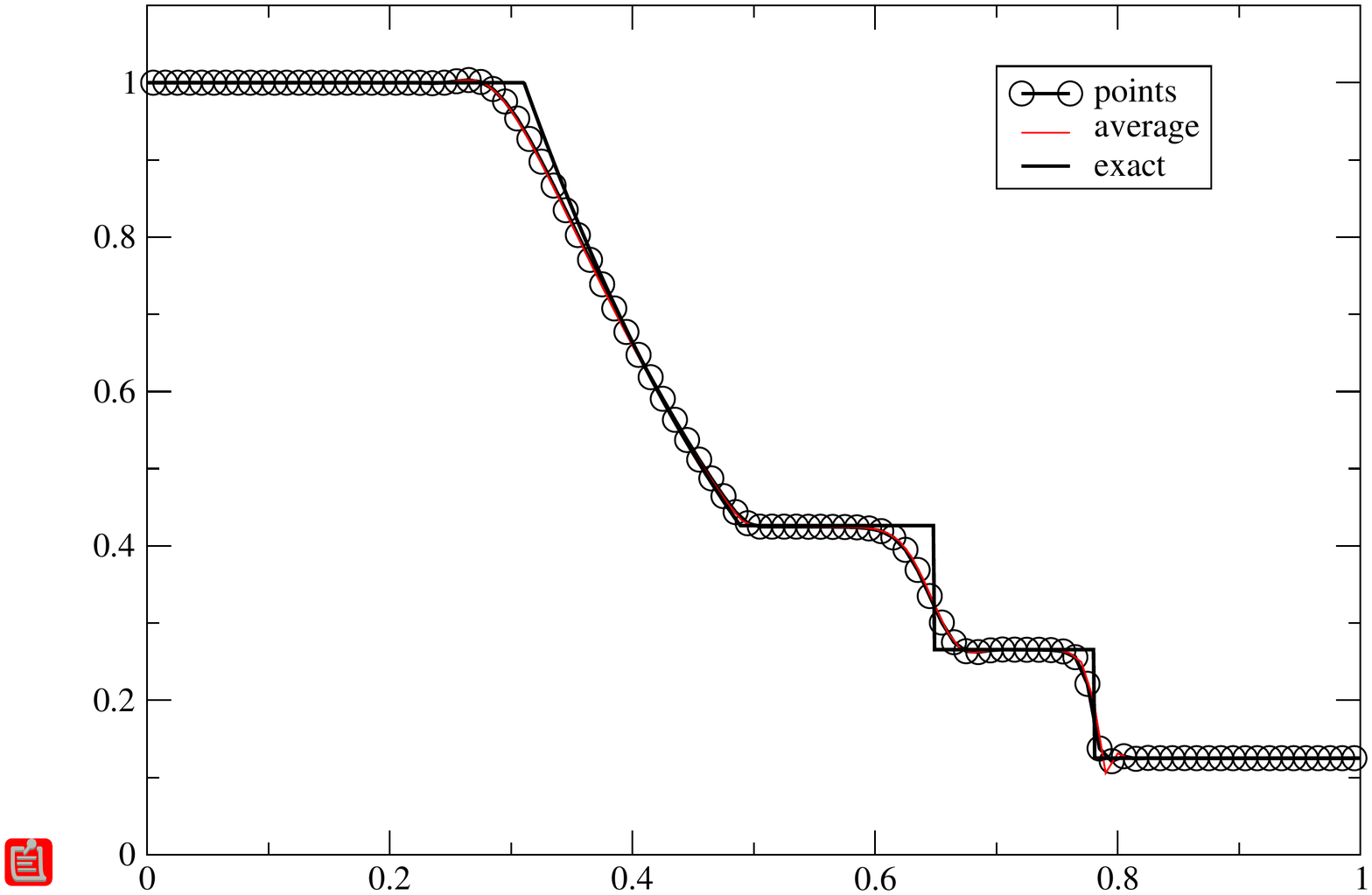}}
\subfigure[]{\includegraphics[width=0.45\textwidth]{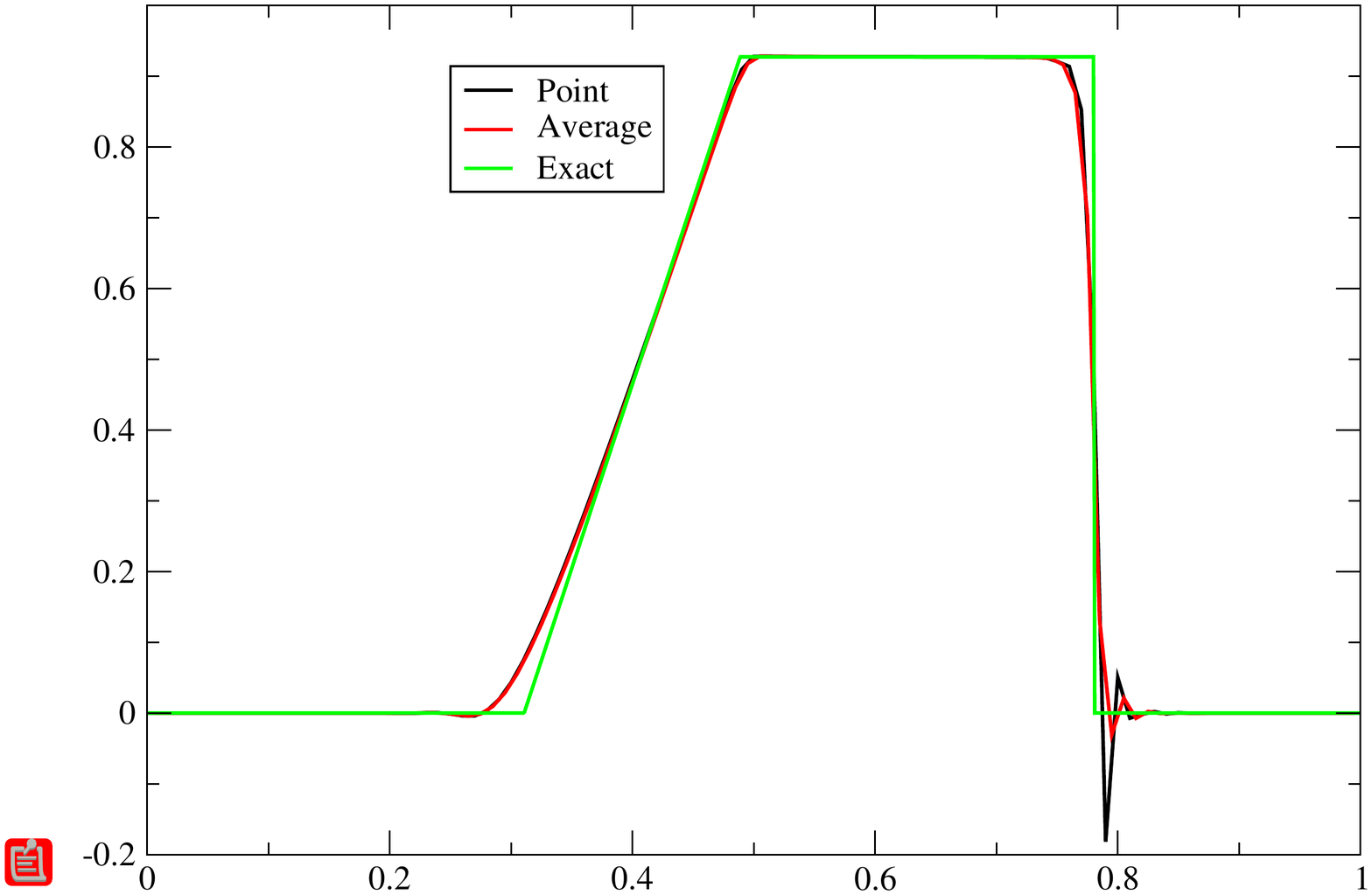}}
\subfigure[]{\includegraphics[width=0.45\textwidth]{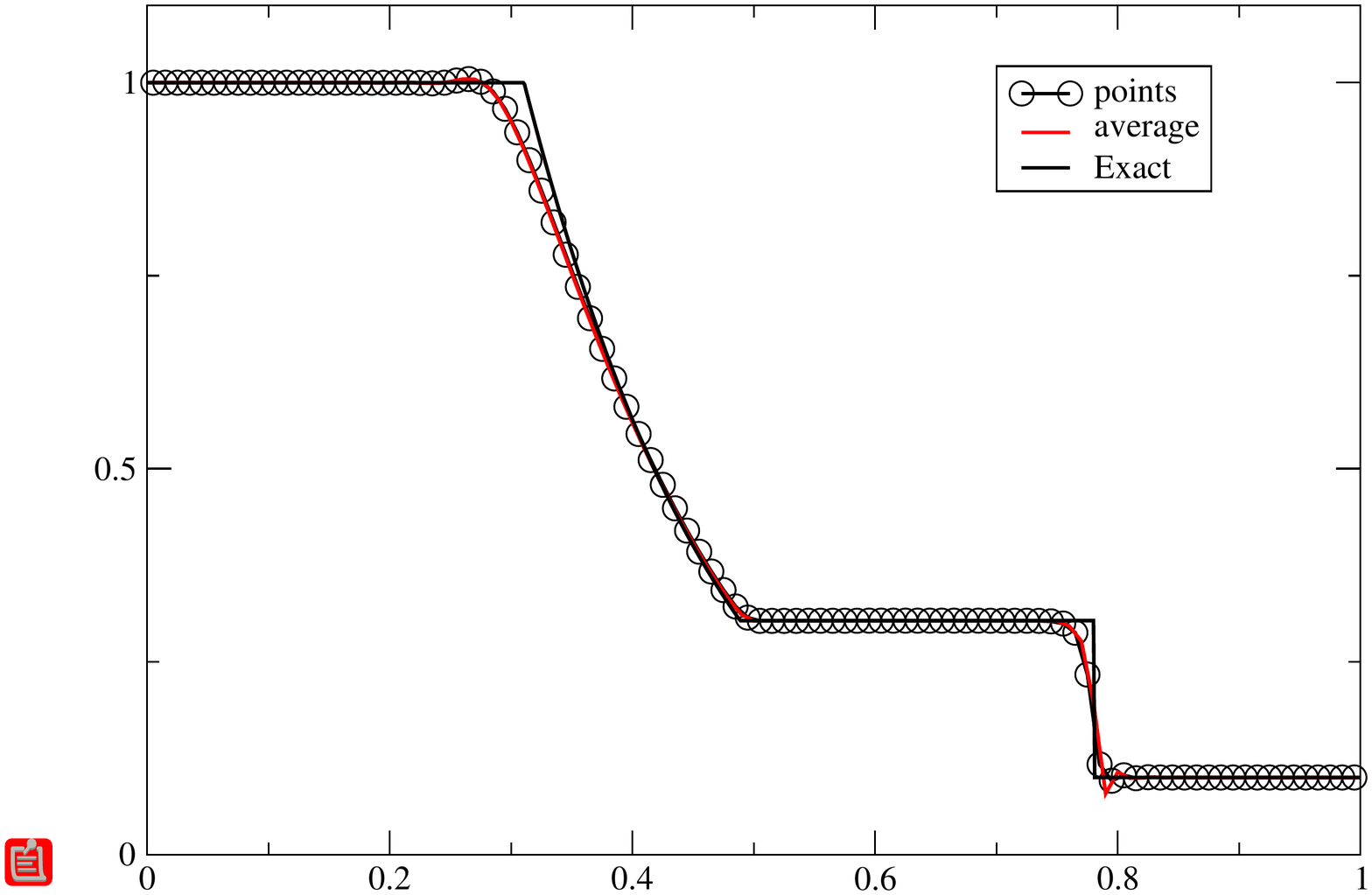}}
\end{center}
\caption{\label{fig:sod_100O2} 100 grid points, and the \remiIII{second order SSPRK2} scheme with CFL=0.1. (a): density, (b): velocity, (c): pressure. }
\end{figure}
\begin{figure}[h]
\begin{center}
\subfigure[]{\includegraphics[width=0.45\textwidth]{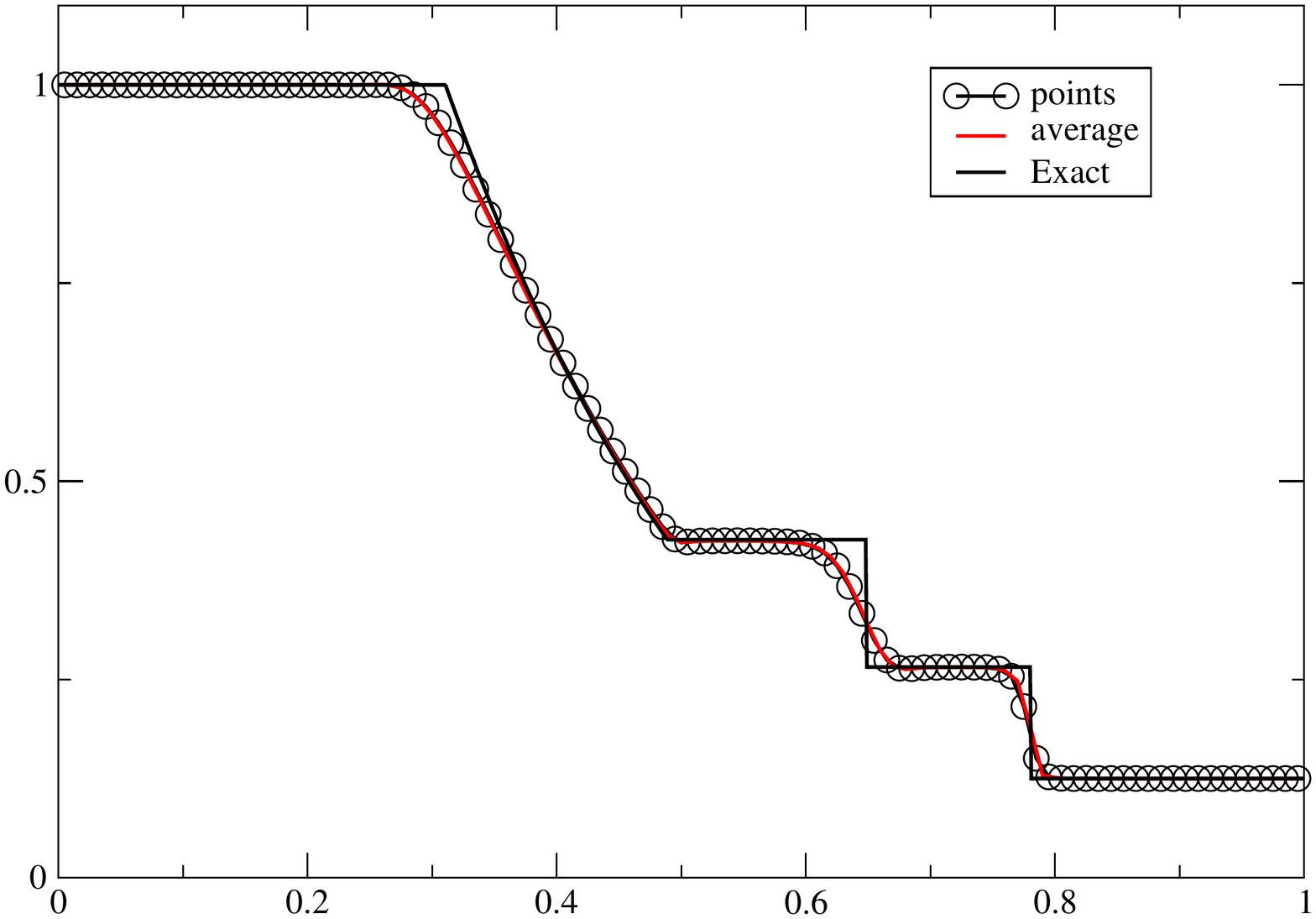}}
\subfigure[]{\includegraphics[width=0.45\textwidth]{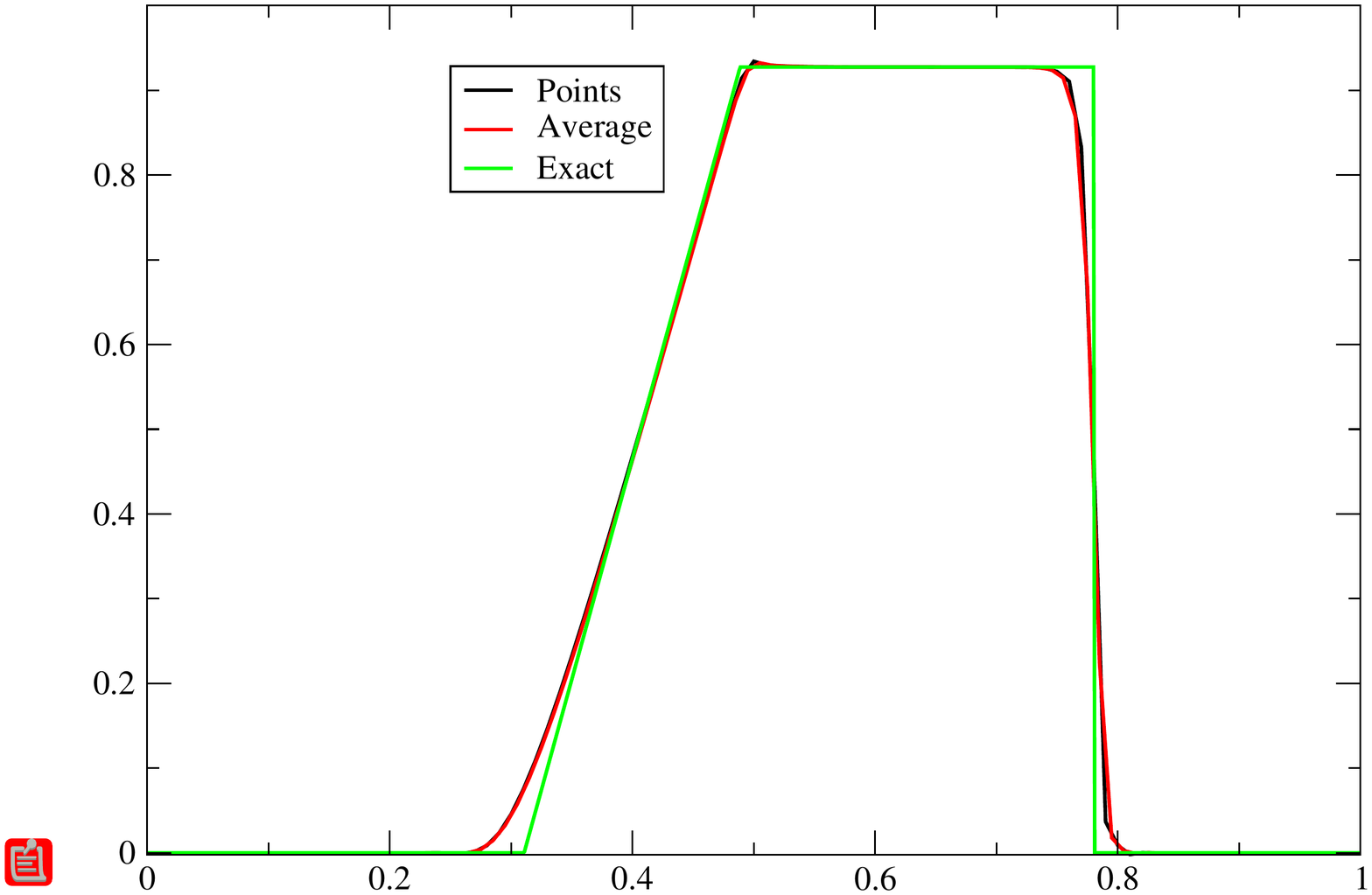}}
\subfigure[]{\includegraphics[width=0.45\textwidth]{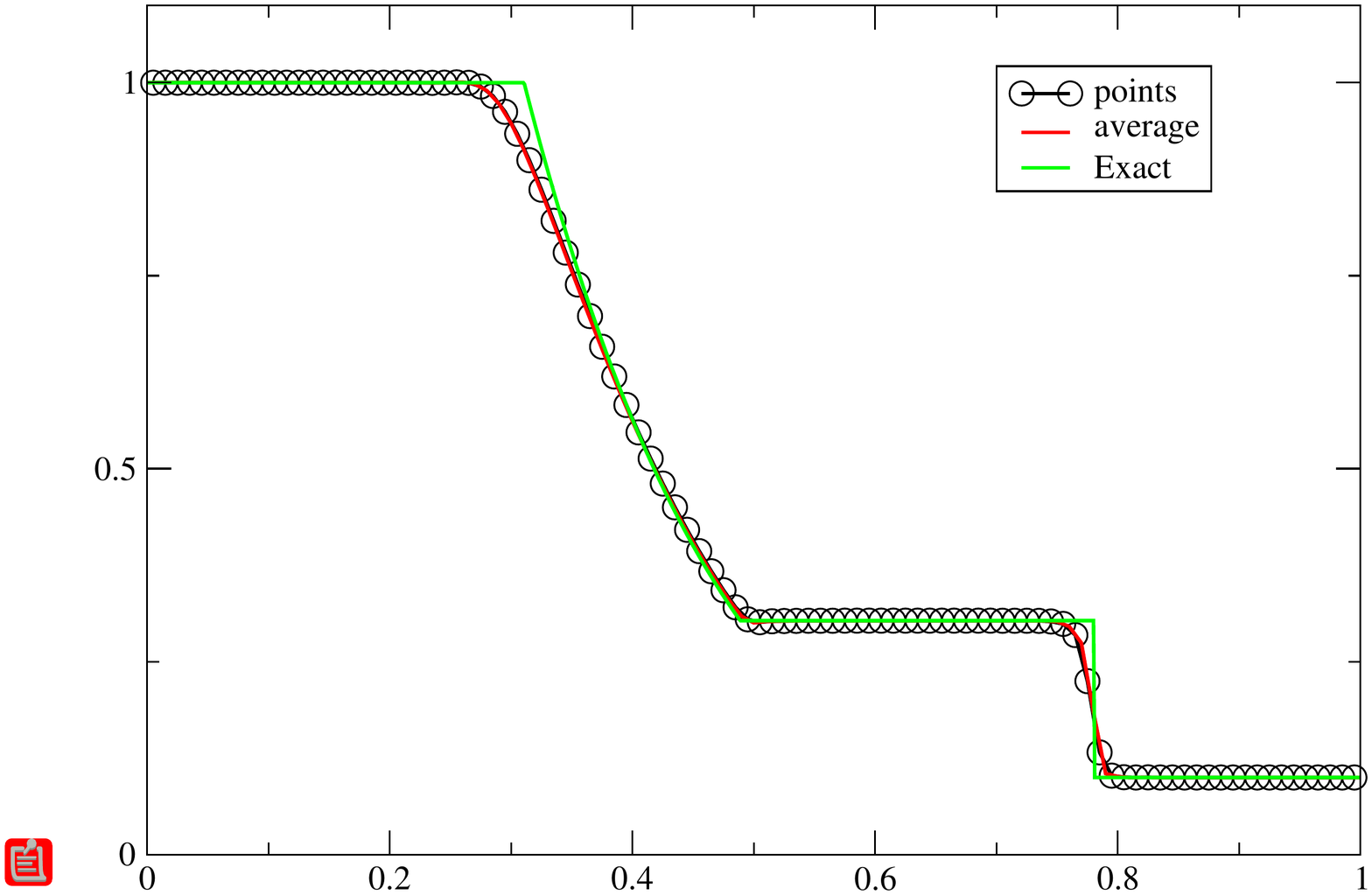}}
\end{center}
\caption{\label{fig:sod_100O2Mood} 100 grid points, and the \remiIII{second order SSPRK2} scheme with CFL=0.1. (a): density, (b): velocity, (c): pressure. MOOD test made on $\rho$ and $p$}
\end{figure}
The exact solution is also shown every time. Different order in time/space are tested. The results are good, eventhough the MOOD procedure is not perfect.
\begin{figure}[h]
\begin{center}
\subfigure[]{\includegraphics[width=0.45\textwidth]{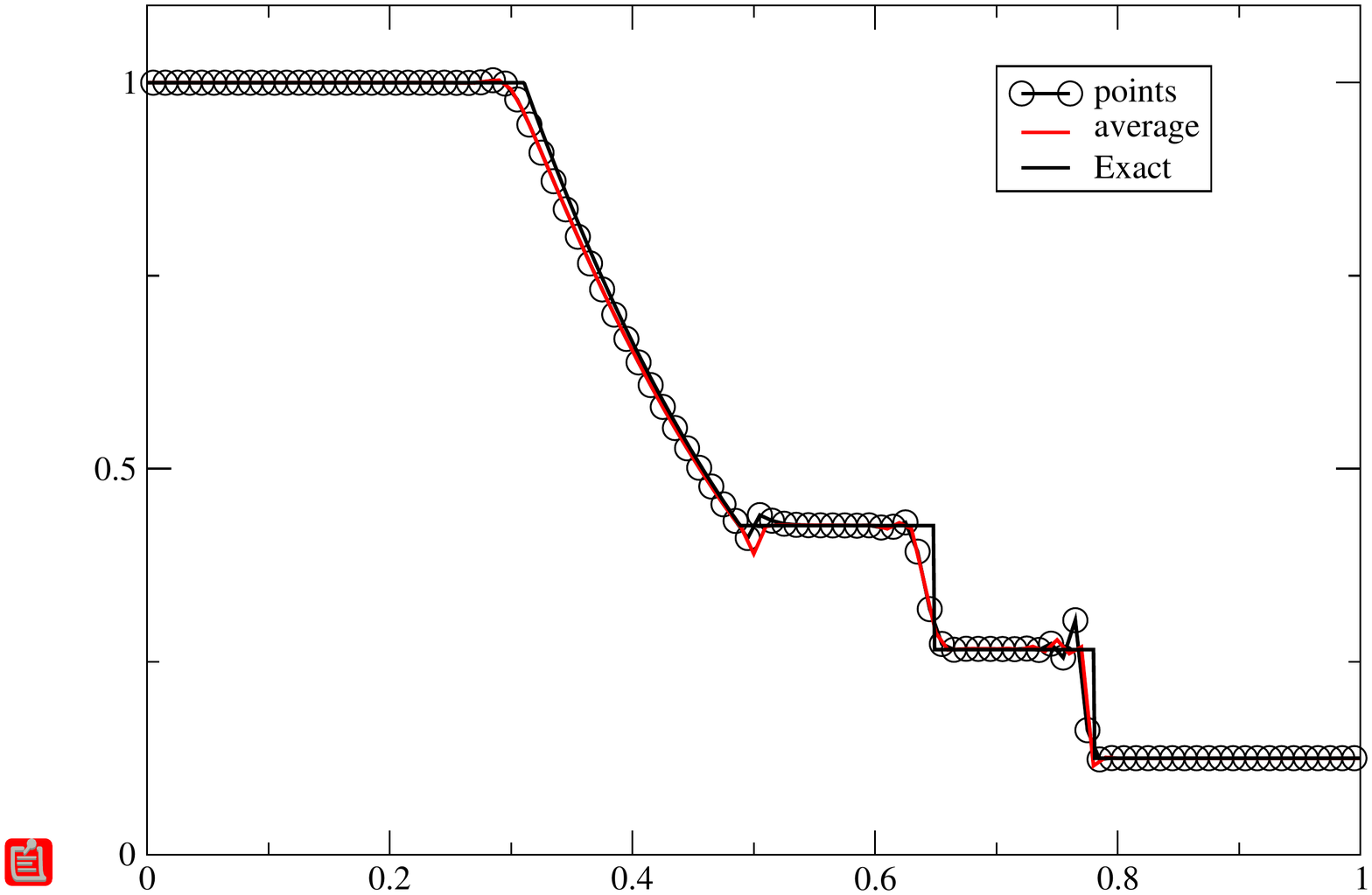}}
\subfigure[]{\includegraphics[width=0.45\textwidth]{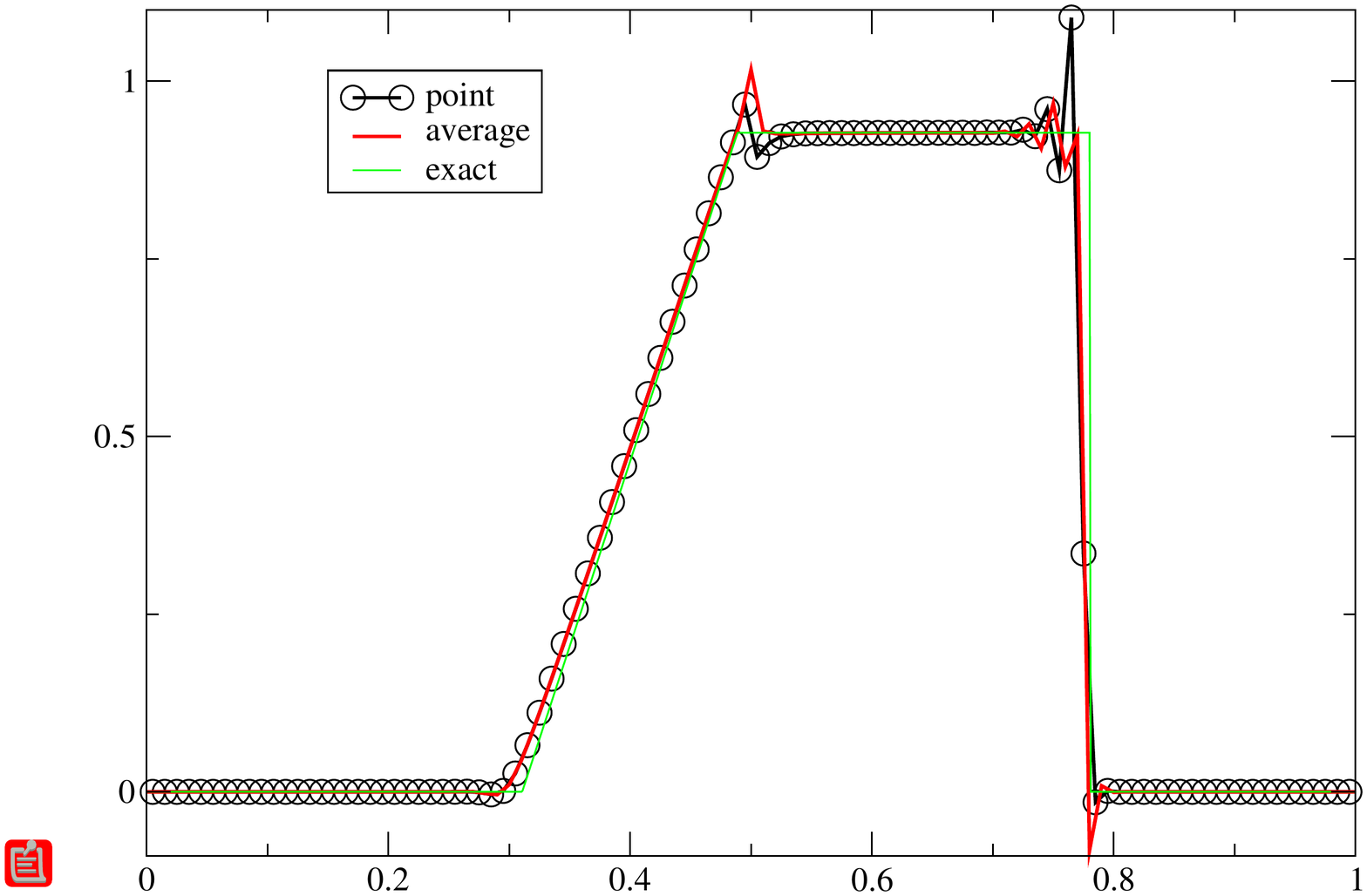}}
\subfigure[]{\includegraphics[width=0.45\textwidth]{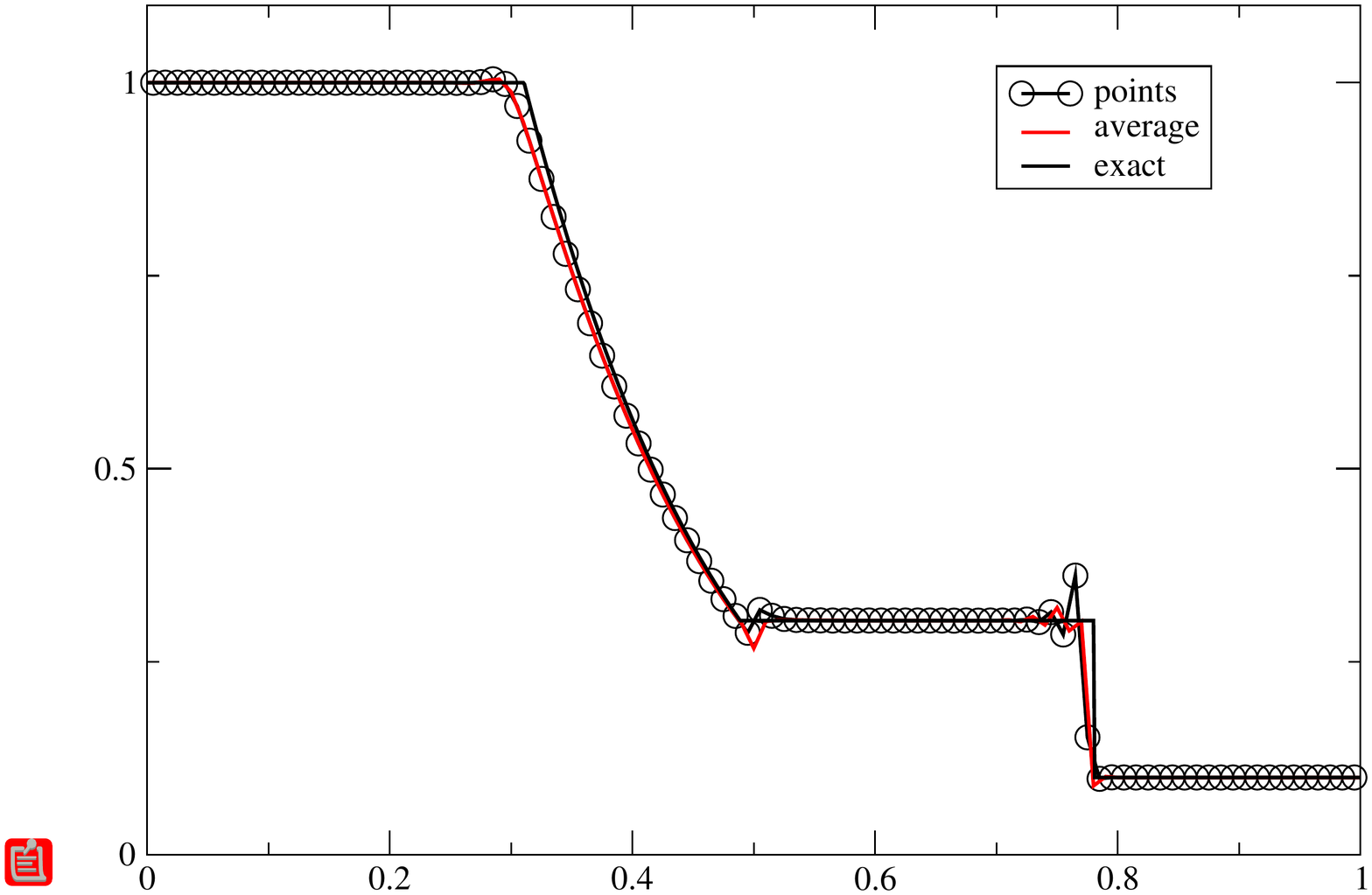}}
\end{center}
\caption{\label{fig:sod_100O3} 100 grid points, and the third order SSPRK3 scheme with CFL=0.1. (a): density, (b): velocity, (c): pressure. }
\end{figure}
\begin{figure}[h]
\begin{center}
\subfigure[]{\includegraphics[width=0.45\textwidth]{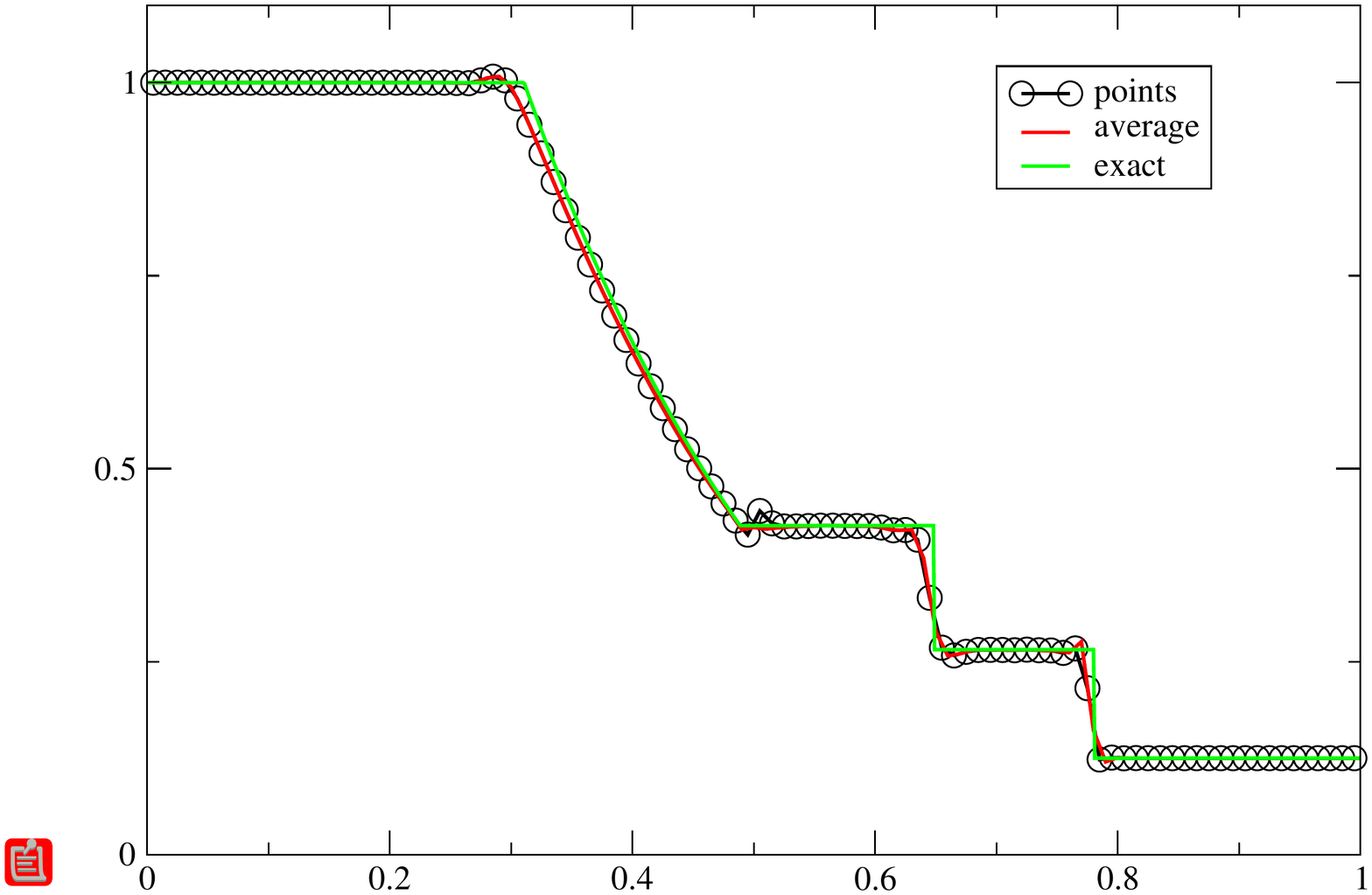}}
\subfigure[]{\includegraphics[width=0.45\textwidth]{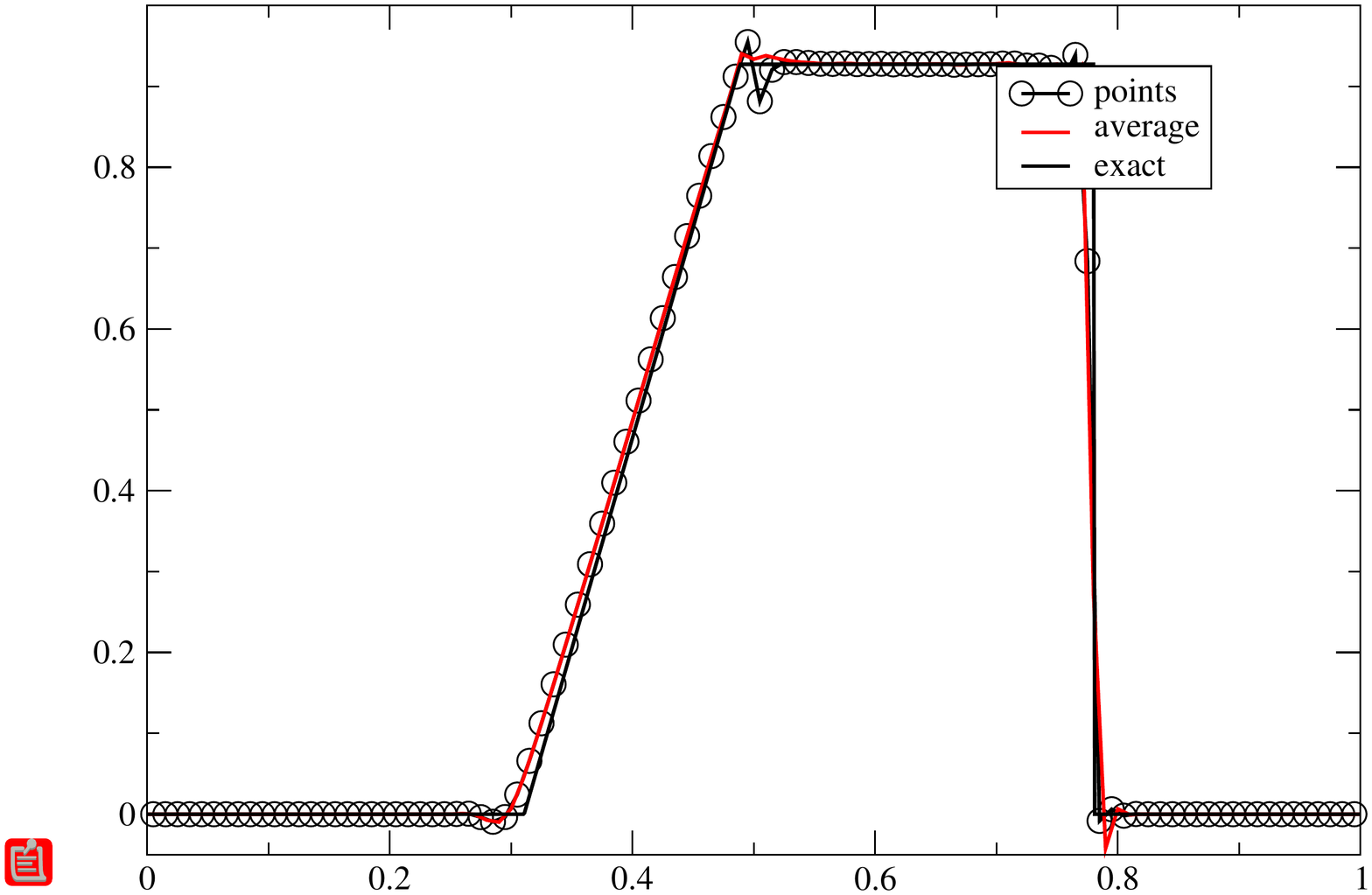}}
\subfigure[]{\includegraphics[width=0.45\textwidth]{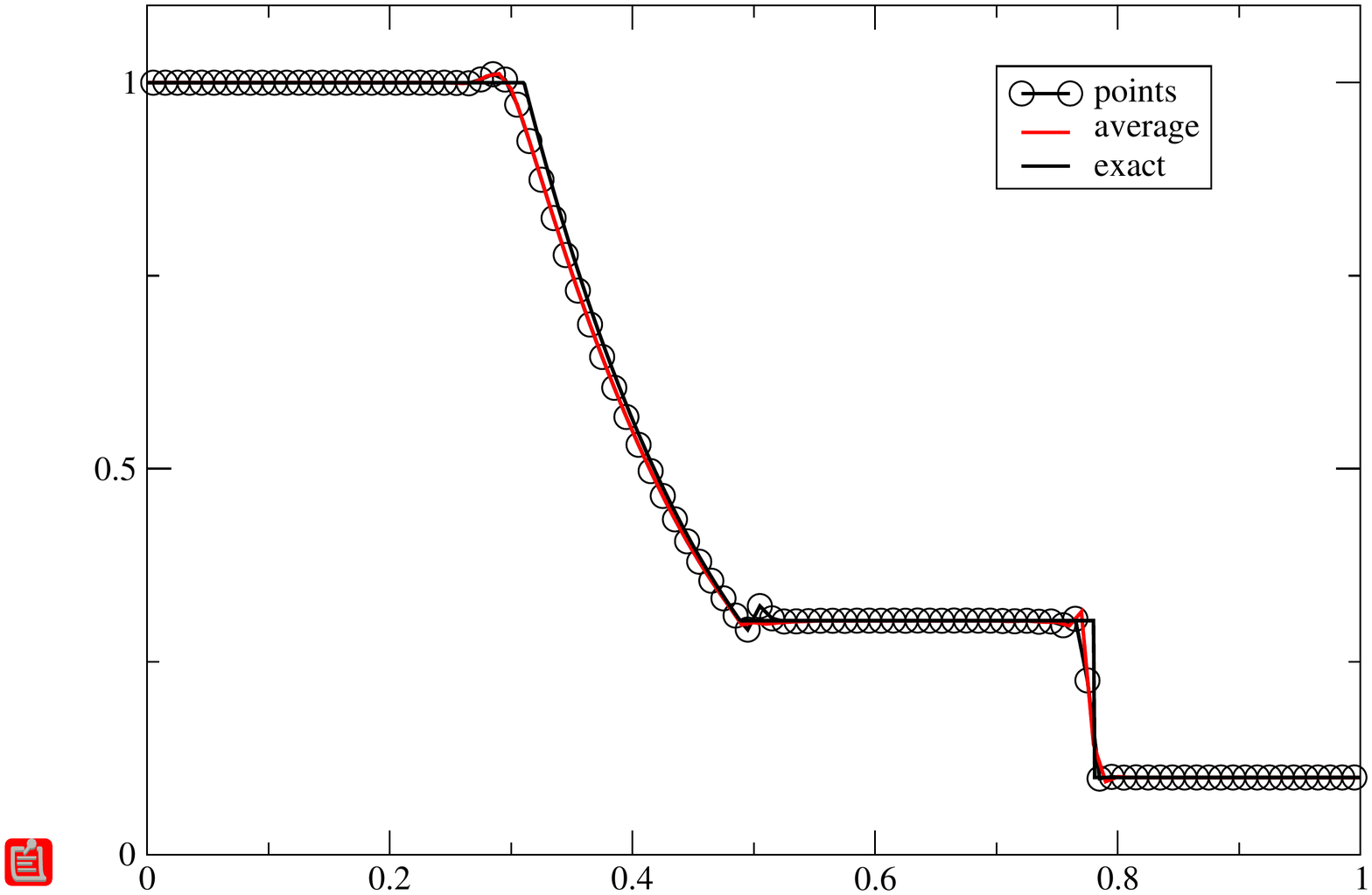}}
\end{center}
\caption{\label{fig:sod_100O3Mood} 100 grid points, and the third order SSPRK3 scheme with CFL=0.1. (a): density, (b): velocity, (c): pressure. Mood Test made on $\rho$ and $p$}
\end{figure}
This the use of the combination \eqref{eq:conservative}-\eqref{eq:entropy} seems more challenging, we have performed a convergence study (with 10000 points). This is shown on figure \eqref{fig:sod:entro_10000}, and a zoom around the contact discontinuity is also shown.
\begin{figure}[h]
\begin{center}
\subfigure[$\rho$]{\includegraphics[width=0.45\textwidth]{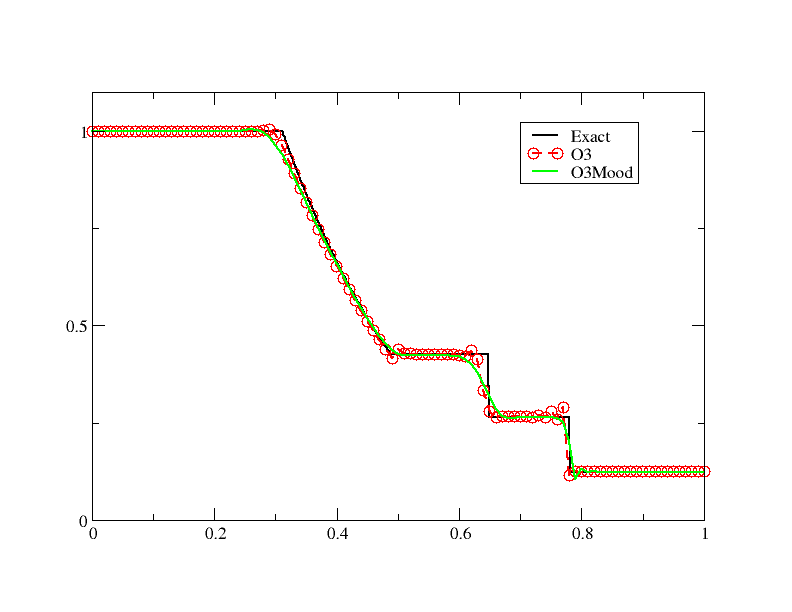}}
\subfigure[$p$]{\includegraphics[width=0.45\textwidth]{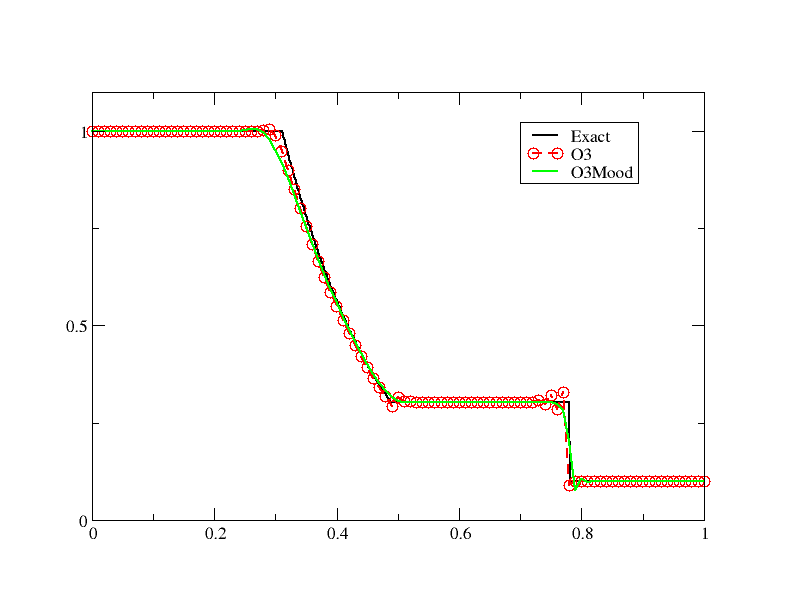}}
\subfigure[$u$]{\includegraphics[width=0.45\textwidth]{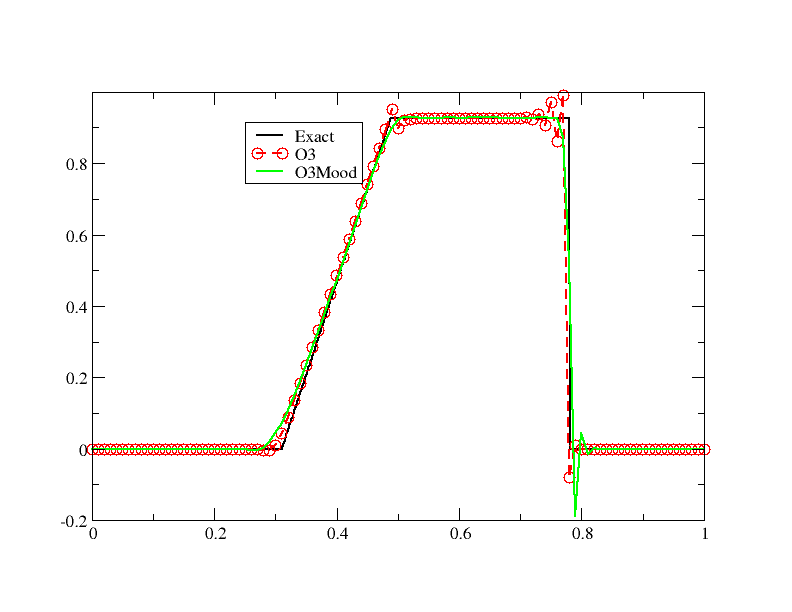}}
\subfigure[$s$]{\includegraphics[width=0.45\textwidth]{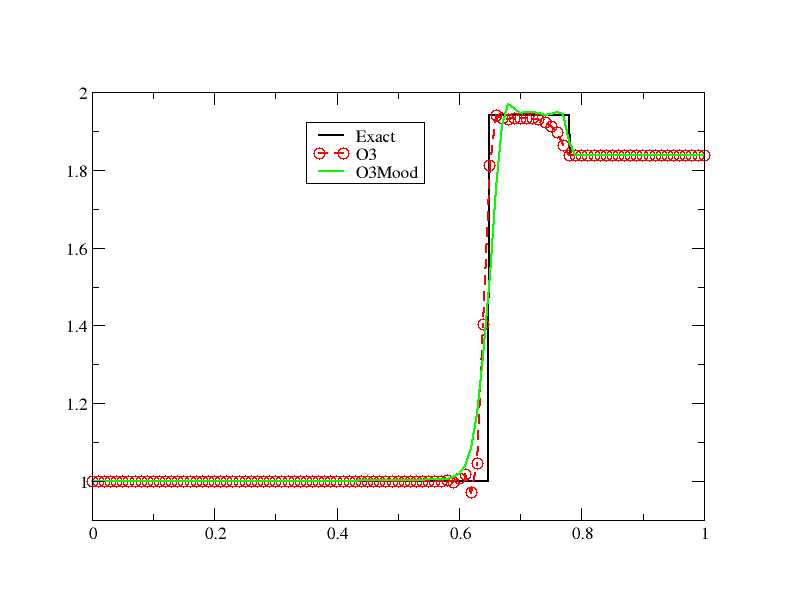}}
\end{center}
\caption{\label{fig:sod:entro_100} Solution with the variables (s,u,p) for 100 points, comparison with the exact solution, third order in time/space with Mood and non Mood. Mood is done on $\rho$ and $p$. Cfl=0.2}
\end{figure}
\begin{figure}[h]
\begin{center}
\subfigure[$\rho$]{\includegraphics[width=0.45\textwidth]{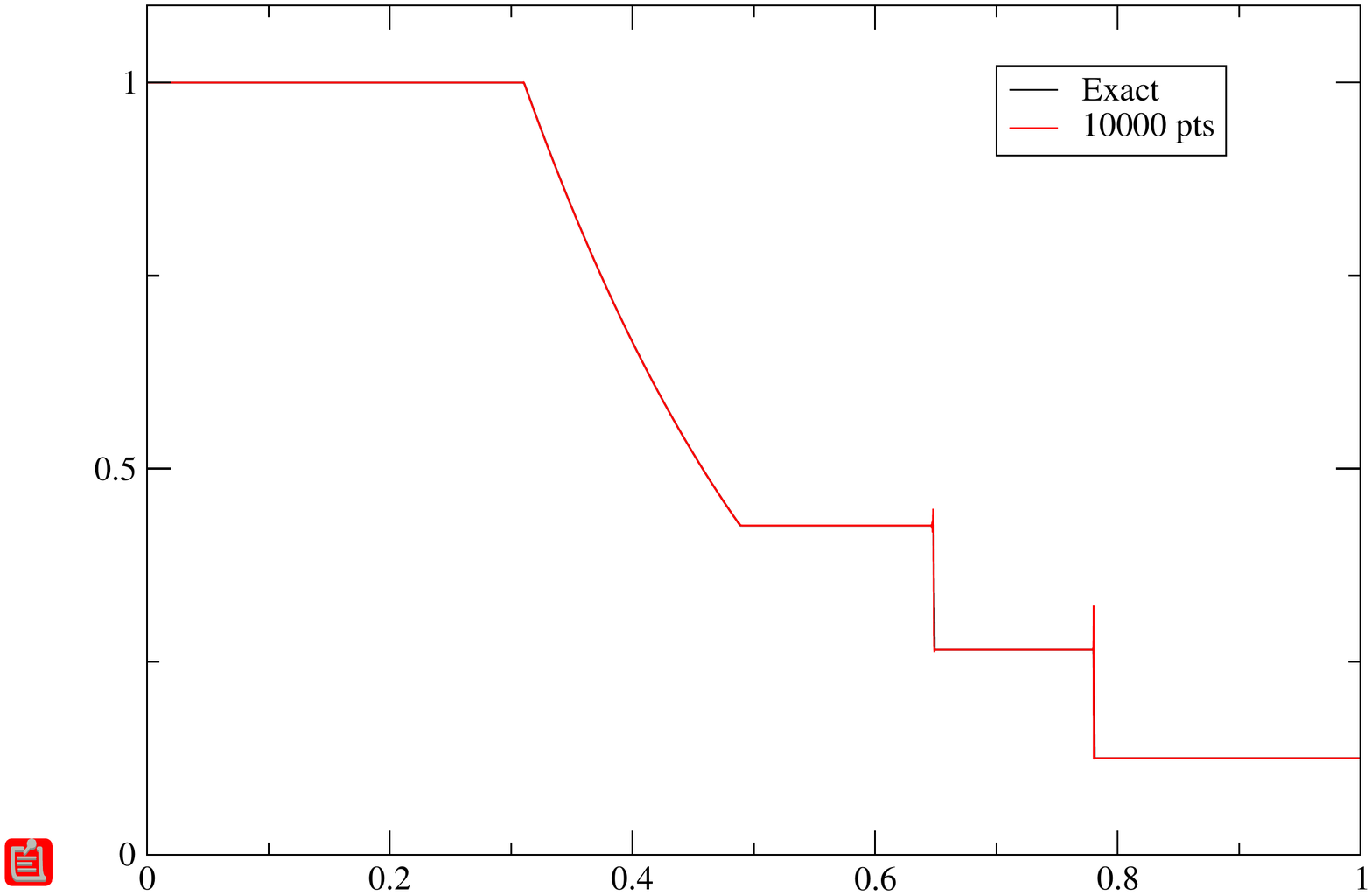}}
\subfigure[$p$]{\includegraphics[width=0.45\textwidth]{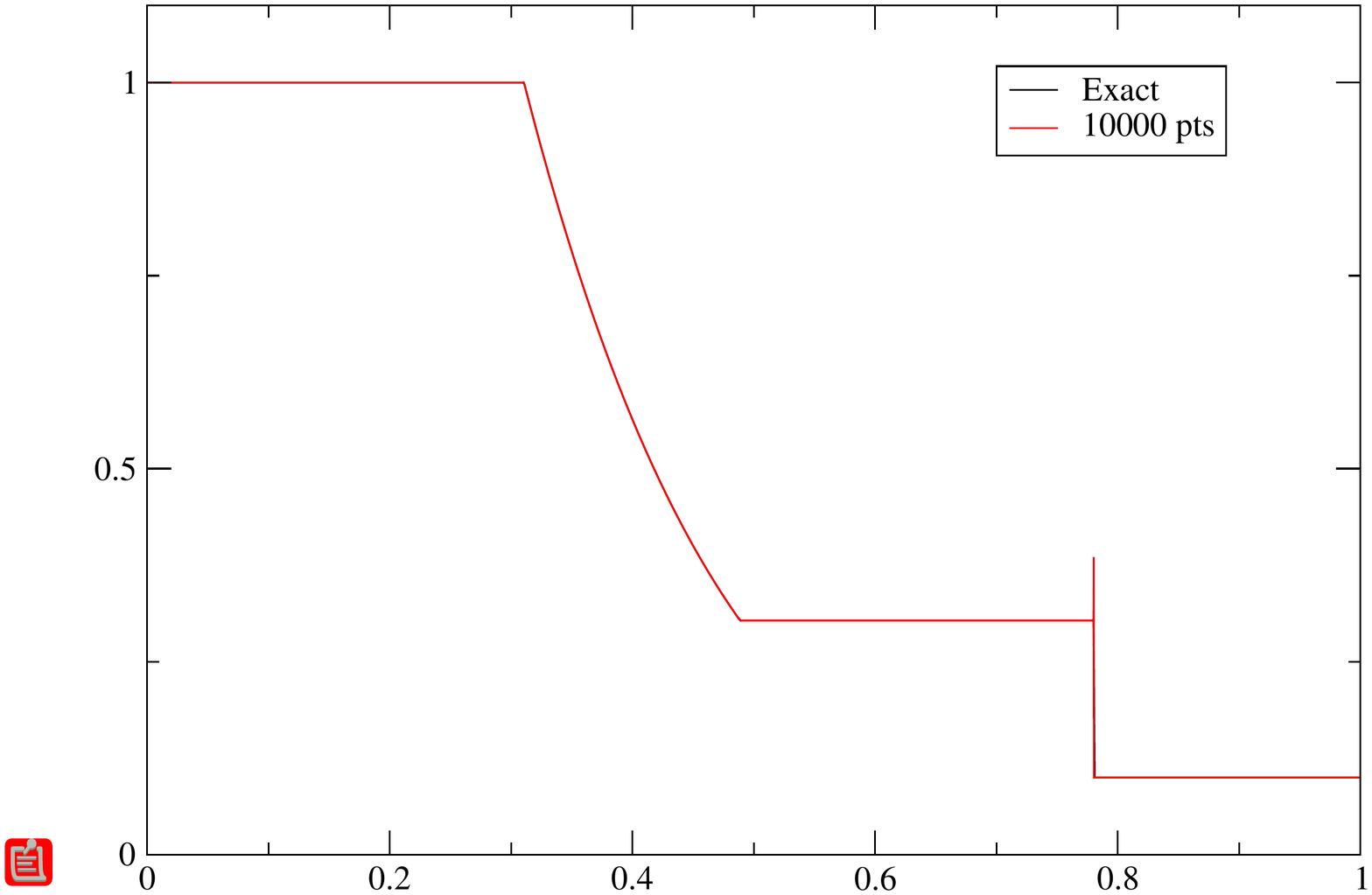}}
\subfigure[$u$]{\includegraphics[width=0.45\textwidth]{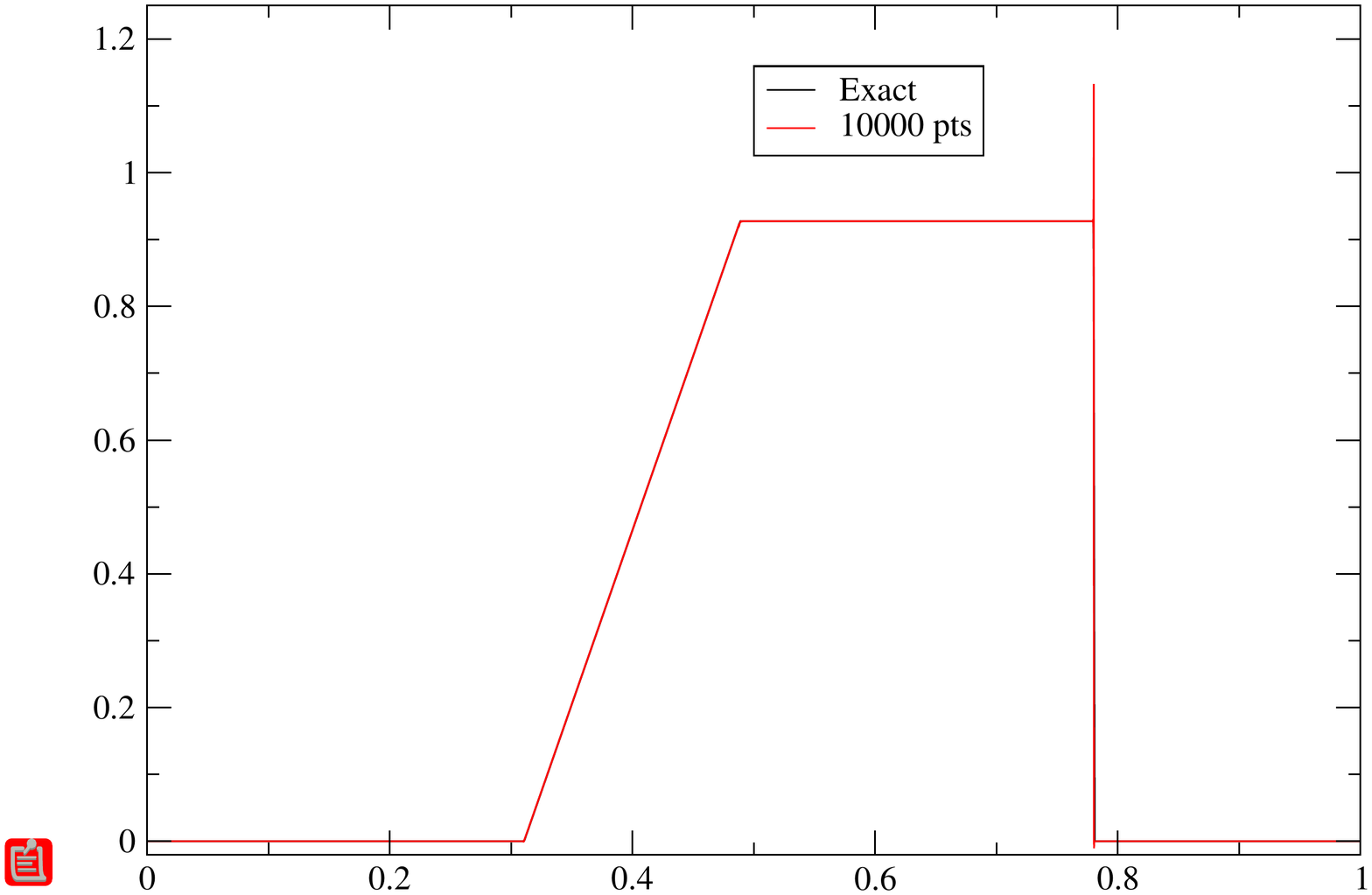}}
\subfigure[$p$ zoom]{\includegraphics[width=0.45\textwidth]{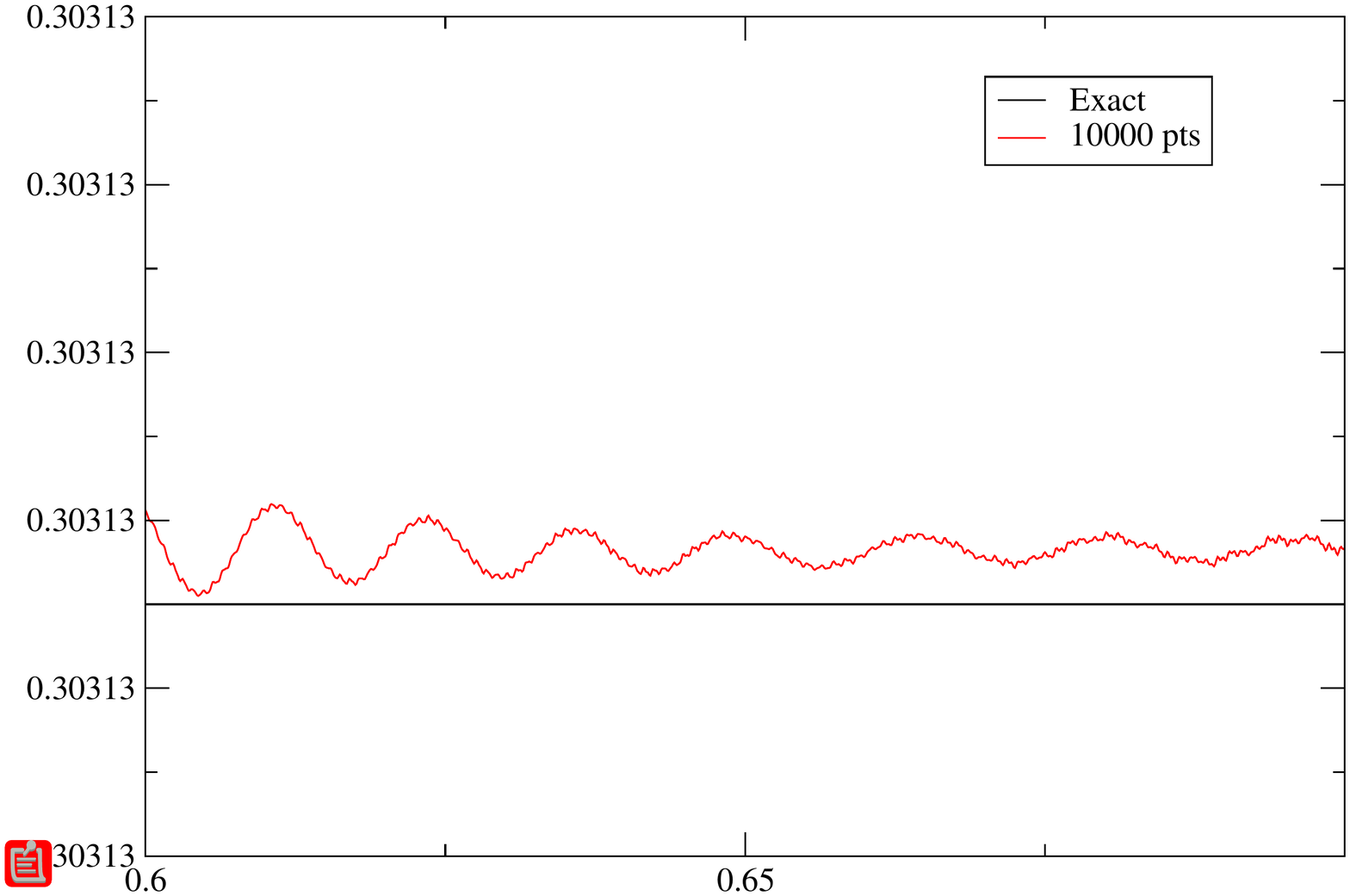}}
\subfigure[$u$ zoom]{\includegraphics[width=0.45\textwidth]{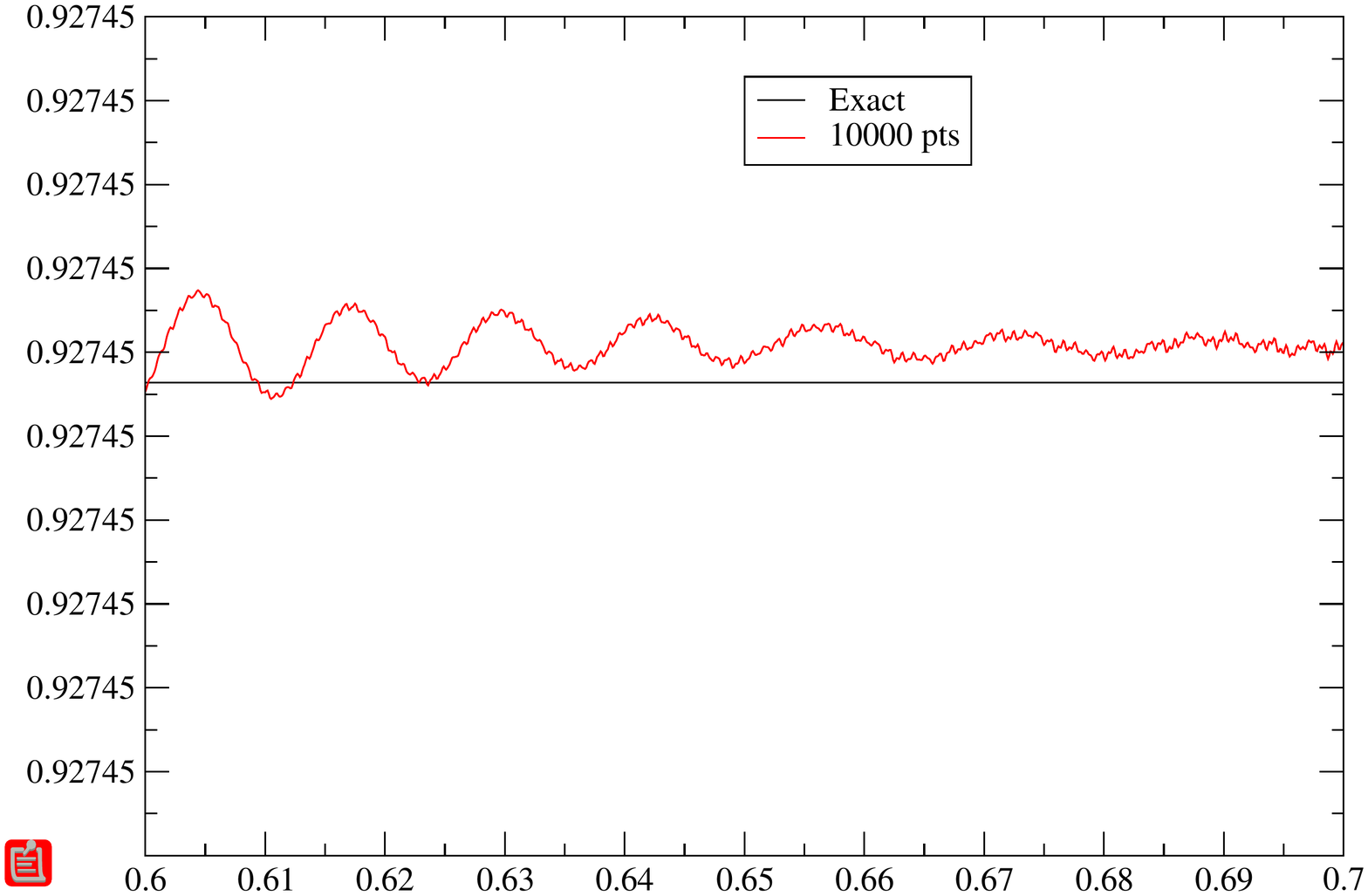}}
\end{center}
\caption{\label{fig:sod:entro_10000} Solution with the variables (s,u,p) for 10000 points, comparison with the exact solution. Cfl=0.1, no mood
The zoomed figures are for $x\in [0.6,0.7]$ and the ticks are for $10^{-7}$. We plot $u$ and $p$ across the contact}
\end{figure}

We can observe a numerical convergence to the exact one in all cases. In the appendix \ref{sec:irreg} we show some results on irregular meshes, with the same conclusions.
{
\subsection{A smooth case}
We consider a fluid with $\gamma=3$: the characteristics are straight lines. The initial condition is inspired from Toro: in $[-1,1]$, 
\begin{equation}\label{condition}
\begin{split}
\rho_0(x)&=1+\alpha\sin(2\pi x)\\
u_0(x)&=0\\
p_0(x0&=\rho_0(x)^\gamma
\end{split}
\end{equation}
The classical case is is for $\alpha=0.999995$ where vaccum is almost reached. Here, since we do not want to test the robustness of the method, we take $\alpha=\frac{3}{4}$.  The final time is set to $T=0.1$.

The exact density and velocity in this case can be obtained by the method of characteristics and is explicitly given by
\begin{equation*}
\rho(x,t) = \dfrac12\big( \rho_0(x_1) + \rho_0(x_2)\big), \quad u(x,t) = \sqrt{3}\big(\rho(x,t)-\rho_0(x_1) \big),
\end{equation*}
where for each coordinate $x$ and time $t$ the values $x_1$ and $x_2$ are solutions of the nonlinear equations
\begin{align*}
& x + \sqrt{3}\rho_0(x_1) t - x_1 = 0, \\
& x - \sqrt{3}\rho_0(x_2) t - x_2 = 0.
\end{align*}
A example of numerical solution, superimposed with the exact one, is shown on figure \ref{fig:smooth}. It is obtained with the third order (time and space) scheme, and here we have used the model $(\rho,u,p)$. The CFL number is set to $0.2$.
\begin{figure}[h]
\begin{center}
\subfigure[$\rho$]{\includegraphics[width=0.45\textwidth]{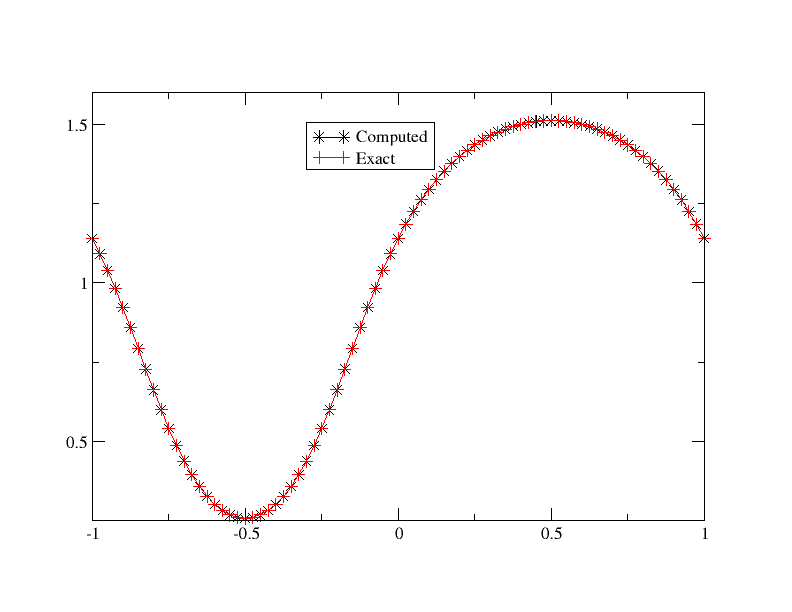}}
\subfigure[$u$]{\includegraphics[width=0.45\textwidth]{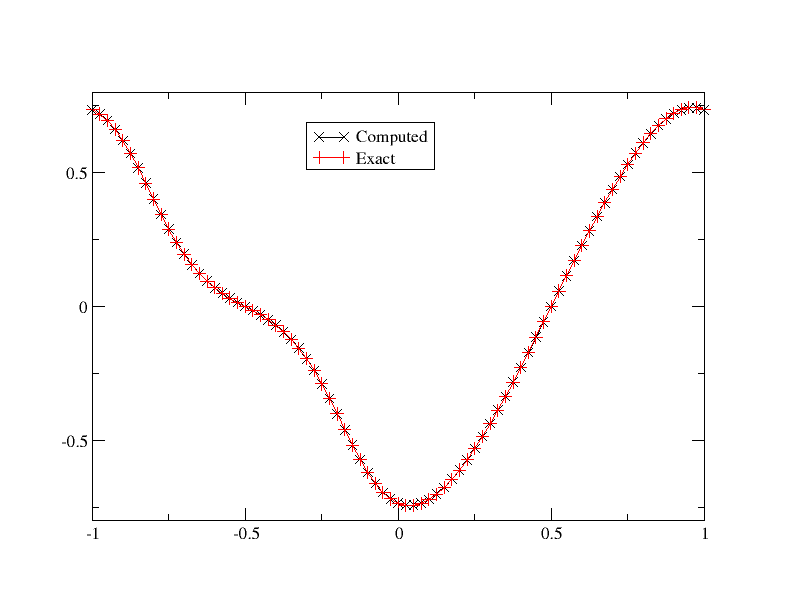}}
\subfigure[$p$]{\includegraphics[width=0.45\textwidth]{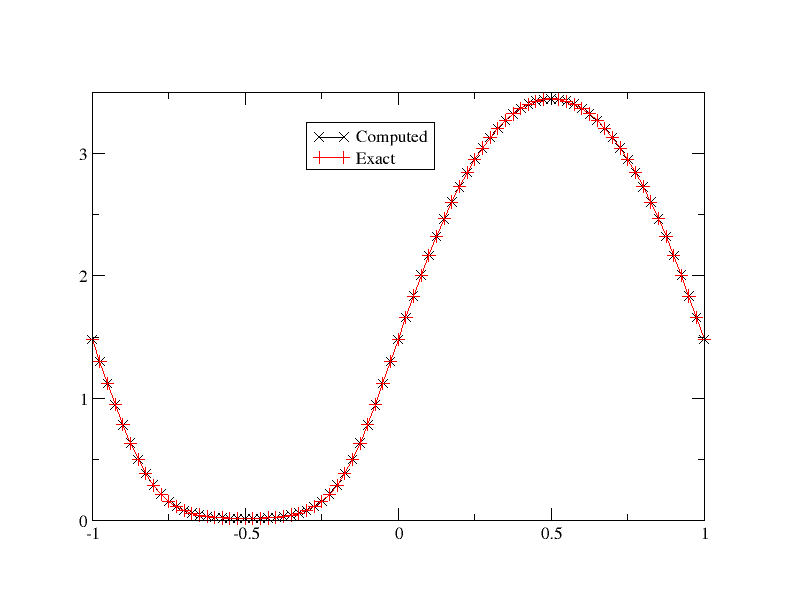}}
\end{center}
\caption{\label{fig:smooth} Solution (numerical and exact) for the conditions \eqref{condition}. The number of grid points is set to $80$, with periodic boundary conditions.}
\end{figure}
The errors are shown in table \ref{table:smooth}.
\begin{table}[h]
\begin{center}
\begin{tabular}{|c||c|c||c|c||c|c|}\hline
  $h=1/N$ & $L^1$ & & $L^2$ & & $L^\infty$ & \\\hline
  20   &$2.136\;10^{-4}$   & $-$        & $2.968\;10^{-4}$ & $-$      &$6.596\;10^{-4}$&$-$\\
  40   &$1.912\;10^{-5}$ & $-3.48$  &$2.702\;10^{-5}$ &$-3.45$  &$5.750\;10^{-5}$&$-3.52$\\
  80   & $1.398\;10^{-6}$ & $-3.77$  & $2.138\;10^{-6}$ & $-3.65$ &$4.673\;10^{-6}$&$-3.62$\\
 160  & $1.934\;10^{-7}$ & $-2.85$   &$2.595\;10^{-7}$ & $-3.04$ &$5.753\;10^{-7}$&$-3.02$\\
  320 & $ 3.641\;10^{-8}$& $-2.40$    &$5.523\;10^{-8}$& $-2.23$ &$1.276\;10^{-7}$&$-2.17$\\
    \hline
\end{tabular}
\end{center}
\caption{\label{table:smooth} $L^1$, $L^2$ and $L^\infty$ error for the initial conditions \eqref{condition} with the third order scheme.}
\end{table}
The errors, computed in $[-1,1]$ are in reasonable agreement with the $-3$ expected slopes.
We also have done the same test with the non linear stabilisation procedure described in section \ref{sec:mood}. \emph{Exactly} the same errors are obtained: the order reduction test are never activated.}
\subsection{Shu-Osher case}
The initial condition are:
$$(\rho, u, p)=\left \{\begin{array}{ll}
(3.857143, 2.629369, 10.3333333) &\text{if } x<-4\\
(1+0.2\sin(5x), 0, 1) &\text{else}
\end{array}\right .
$$
on the domain $[-5,5]$ until $T=1.8$. We have used the combination \eqref{eq:conservative}-\eqref{eq:primitive}, since the other one seems less robust. The density is compared to a reference solution (obtained with a standard finite volume scheme with $20\; 000$ points, and the solution obtained with the third order scheme with $CFL=0.3$ and $200$, $400$,  $800$  and $1600$ points. The mood procedure uses the first order upwind scheme if a PAD, a NaN or DMP is detected, the other cases use the 3rd order scheme.
The solutions are displayed in \ref{figshu}. With little resolution, the results are very close of the reference one.
\begin{figure}[h]
\subfigure[]{\includegraphics[width=0.5\textwidth]{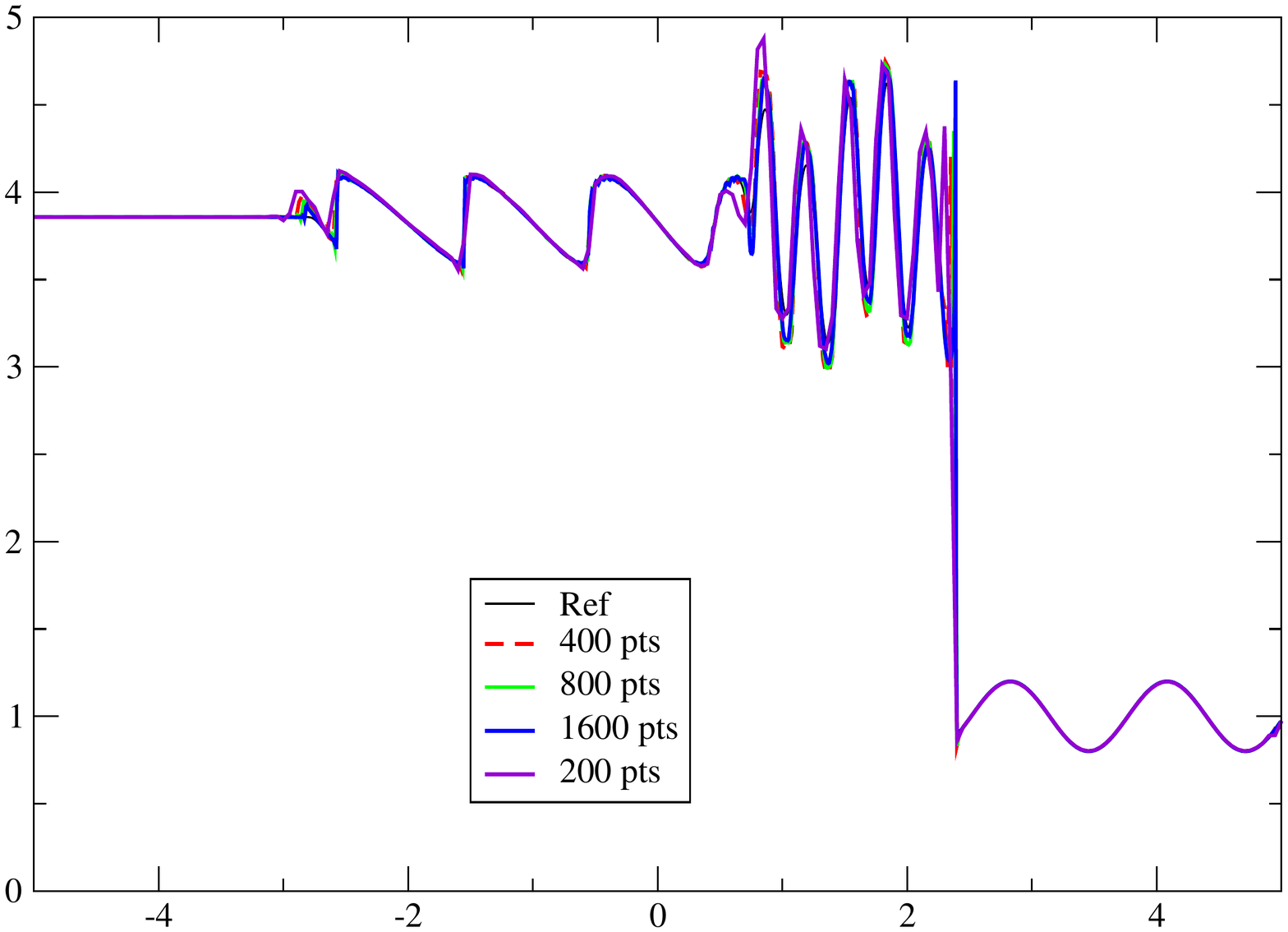}}
\subfigure[]{\includegraphics[width=0.5\textwidth]{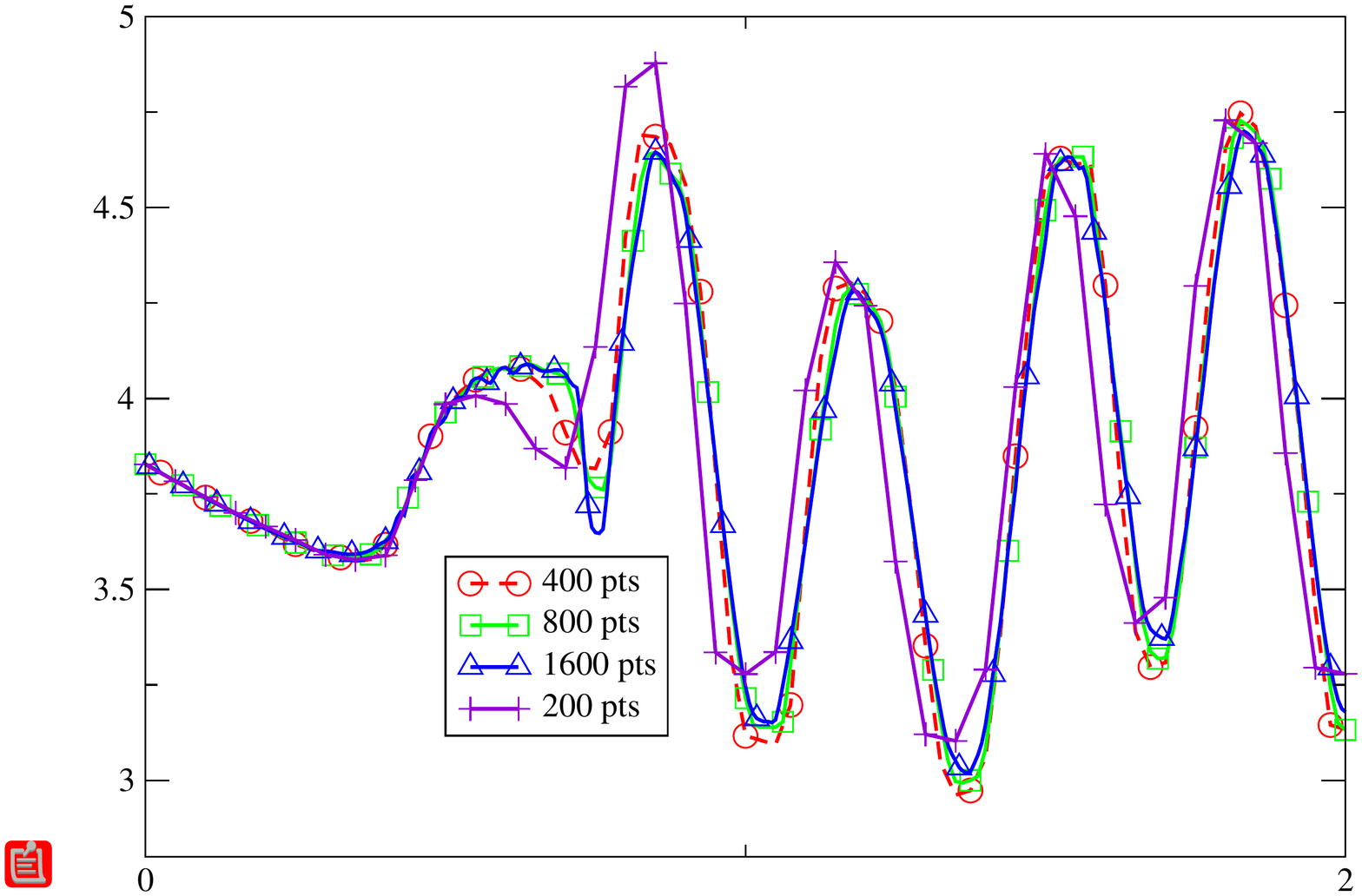}}
\caption{\label{figshu} (a): Solution of the Shu Osher problem , (b): zoom of the solution around the shock.}
\end{figure}
For figure \ref{figshu}, the second order scheme is used as a rescue scheme.

\subsection{Le Blanc case}
The initial conditions are
$$(\rho,u,e)=\left\{\begin{array}{ll}
(1,0,0.1) & \text{ if } x\in [-3,3]\\
(0.001,0.10^{-7} ) & \text{ if } x\in [3,6]
\end{array}
\right .
$$
where $e=(\gamma-1) p$ and $\gamma=\tfrac{5}{3}$. The final time is $t=6$.  This is a very strong shock tube. The combination \eqref{eq:conservative}-\eqref{eq:primitive}. It is not possible to run higher that first order without the MOOD procedure. We show the second and third order results are shown on figure \ref{fig:leblanc}, and zooms around the shocks and the fan are showed in \ref{fig:leblanc2}.
\begin{figure}[h]
\begin{center}
\subfigure[$\rho$]{\includegraphics[width=0.45\textwidth]{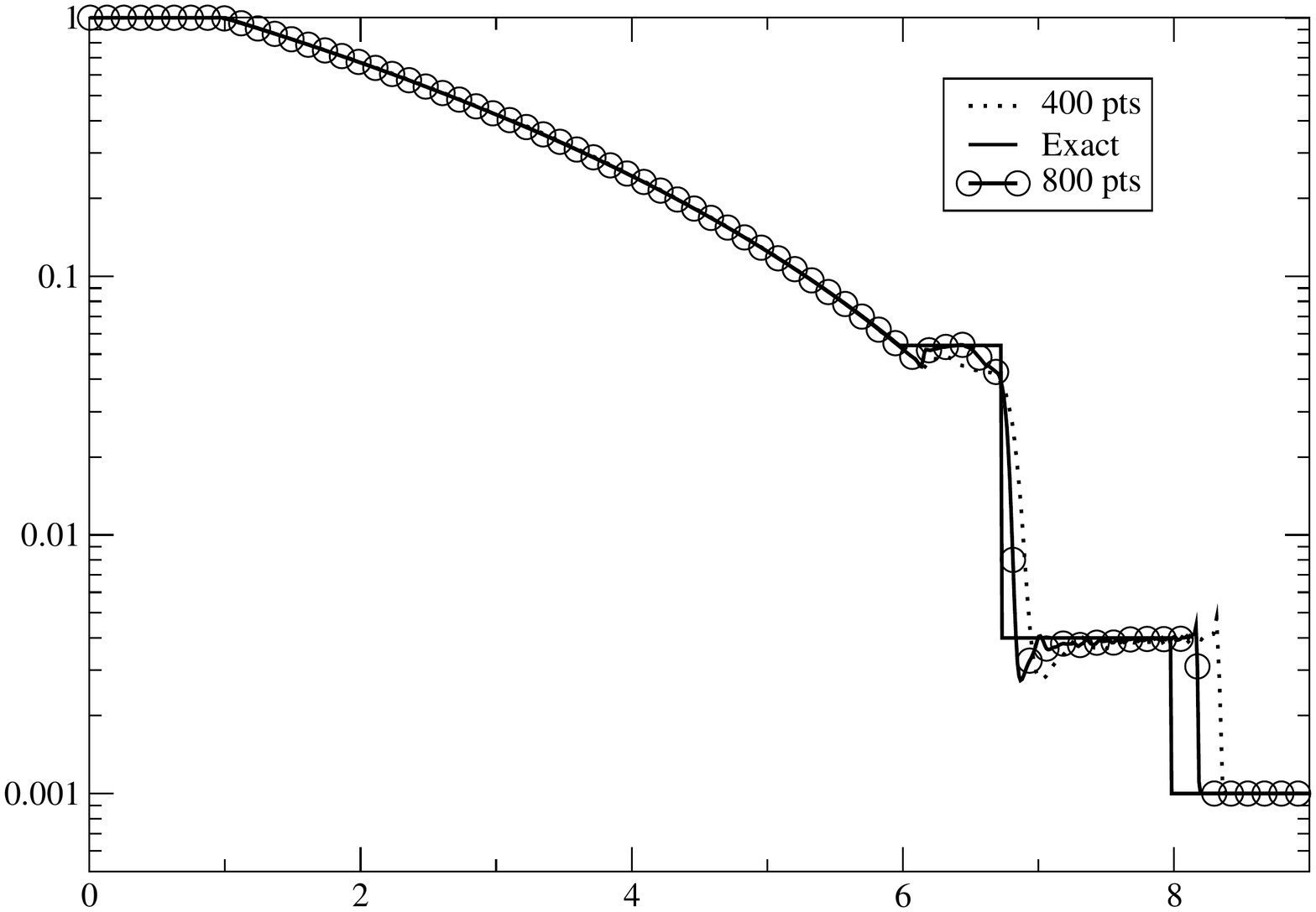}}
\subfigure[$\rho$]{\includegraphics[width=0.45\textwidth]{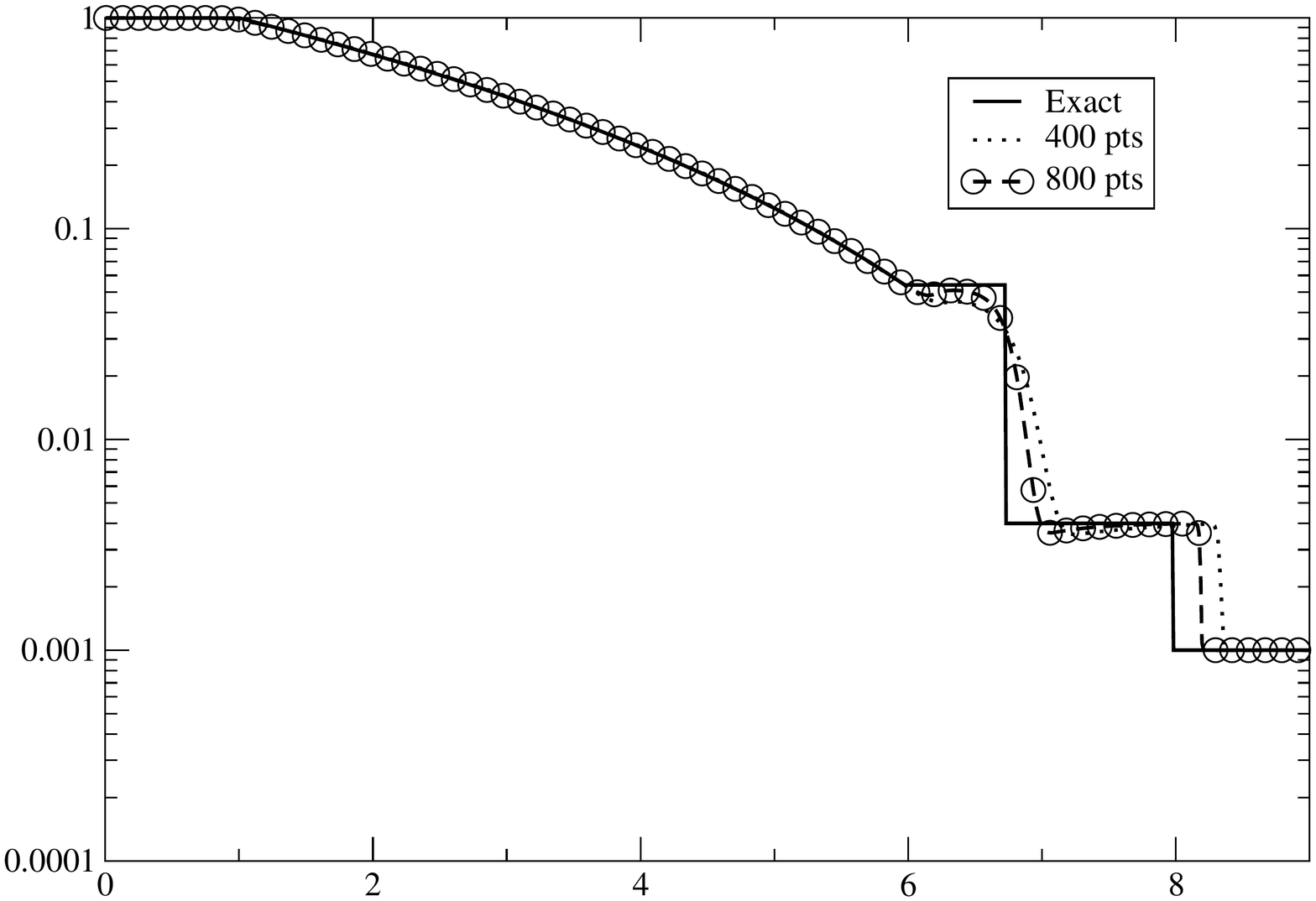}}
\subfigure[$p$]{\includegraphics[width=0.45\textwidth]{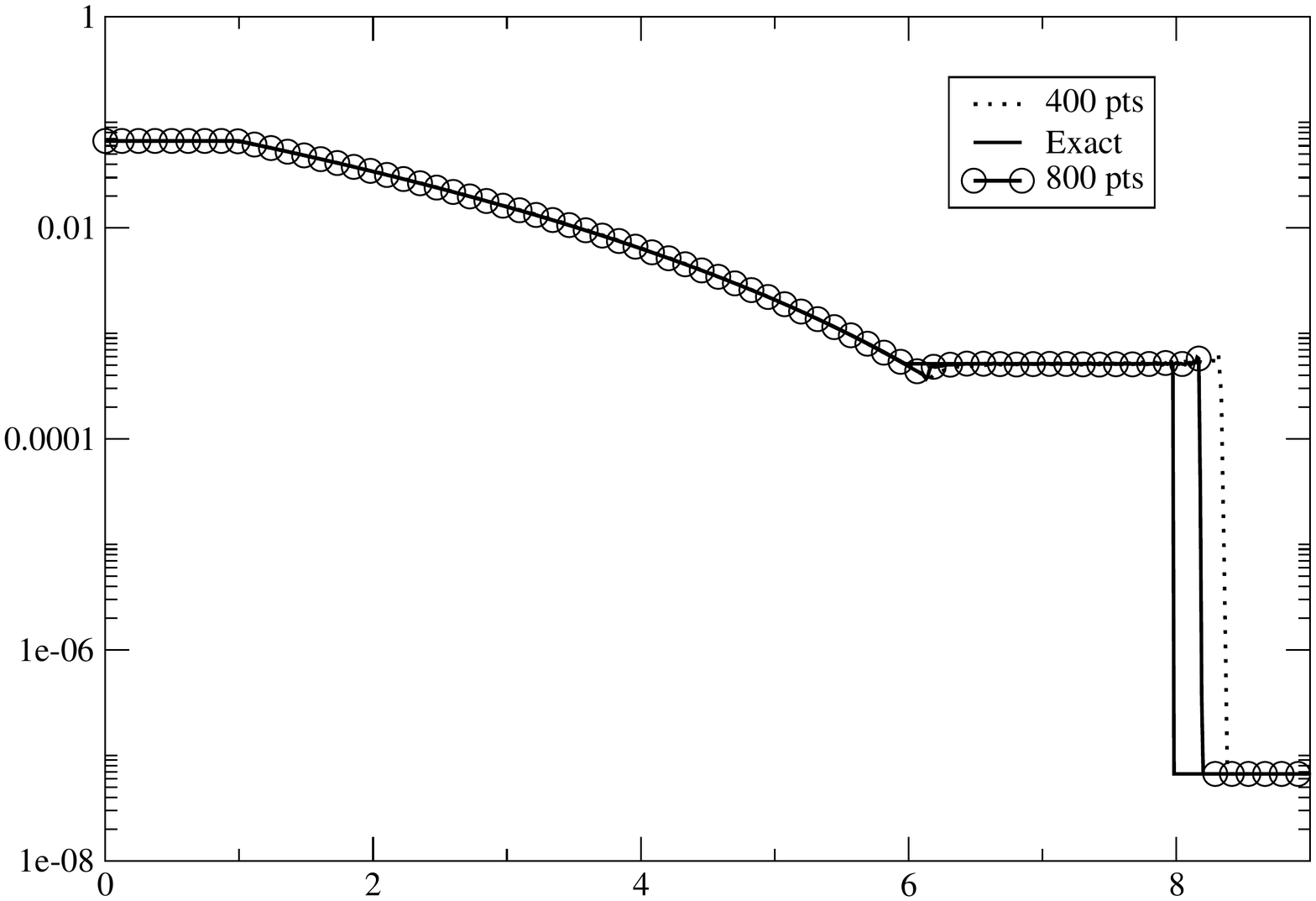}}
\subfigure[$p$]{\includegraphics[width=0.45\textwidth]{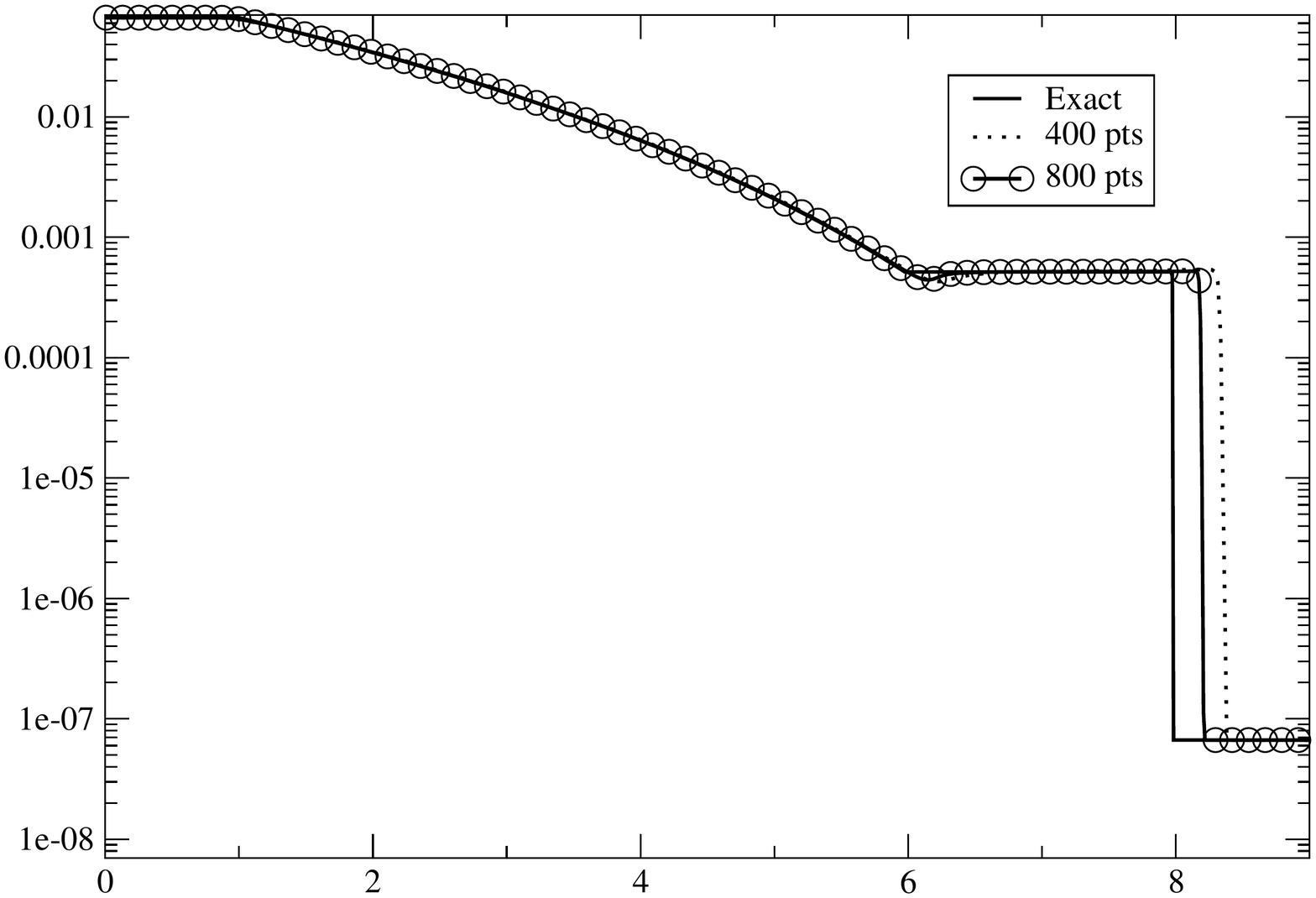}}
\subfigure[$u$]{\includegraphics[width=0.45\textwidth]{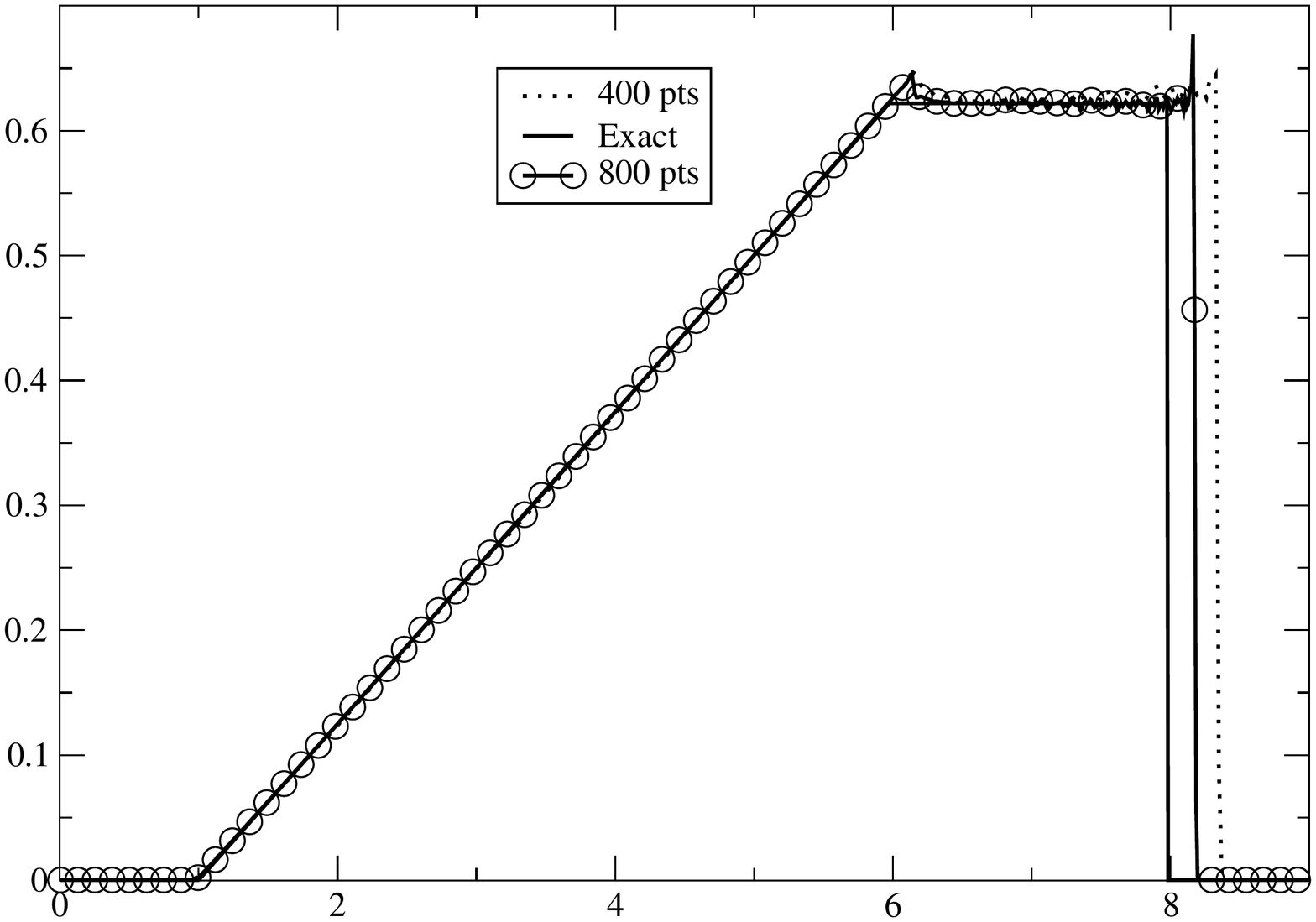}}
\subfigure[$u$]{\includegraphics[width=0.45\textwidth]{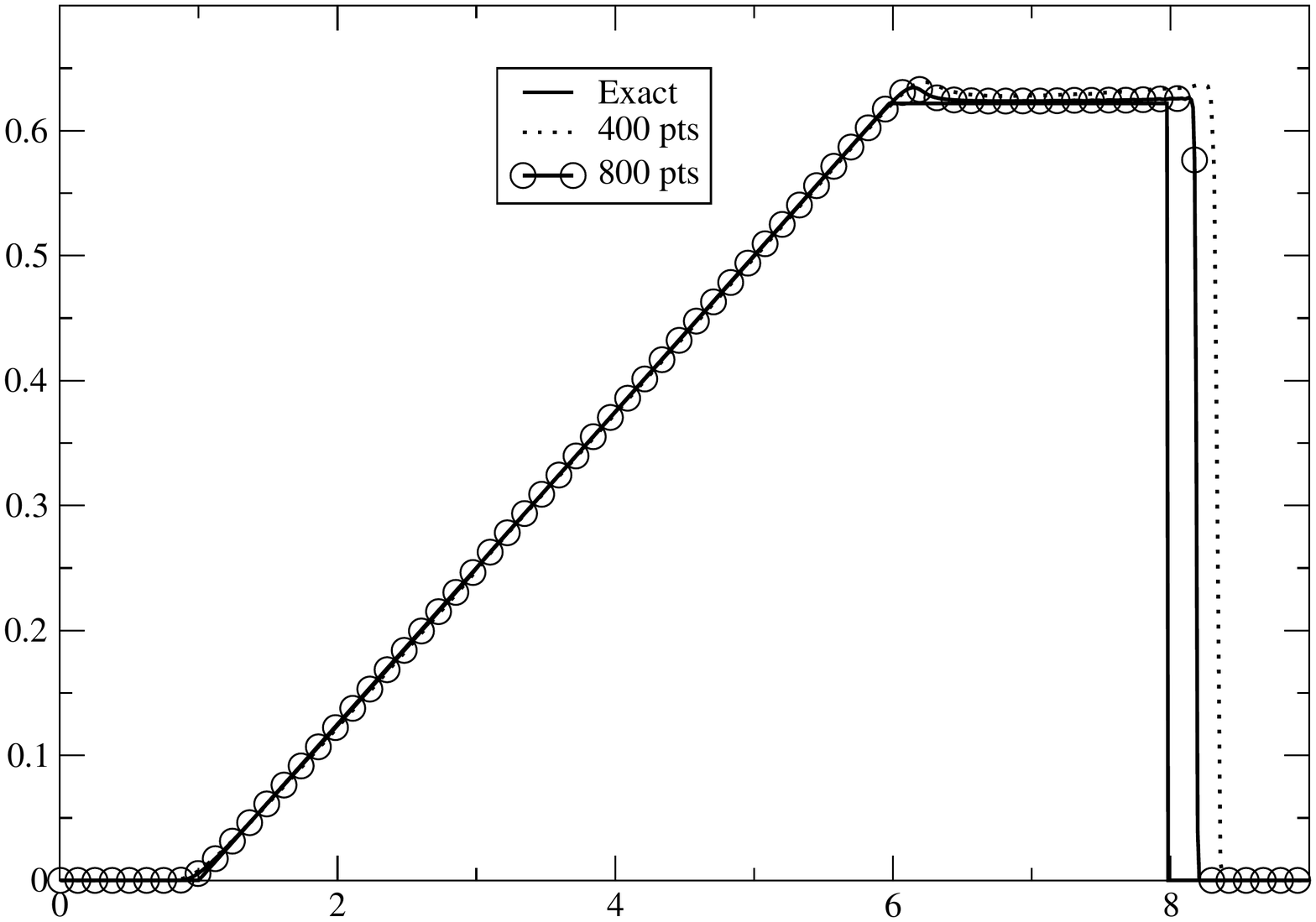}}
\end{center}
\caption{\label{fig:leblanc} Le Blanc test case, $CFL=0.1$, from 400 to 800 points. Left column: MOOD test on $\rho$ and $p$, second order, right column: MOOD test on $\rho$ and $p$, third order}
\end{figure}
\begin{figure}[h]
\begin{center}
\subfigure[]{\includegraphics[width=0.45\textwidth]{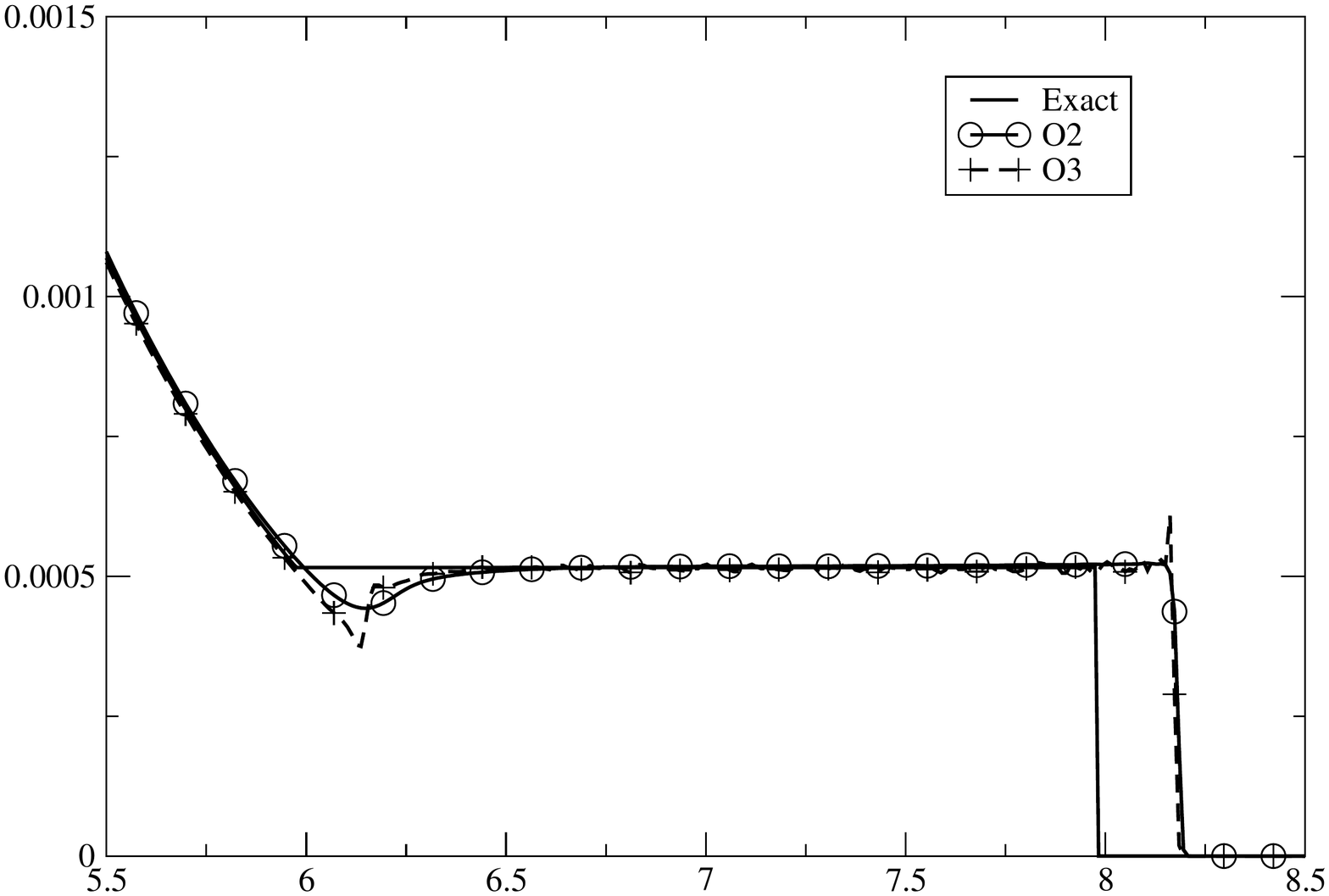}}
\subfigure[]{\includegraphics[width=0.45\textwidth]{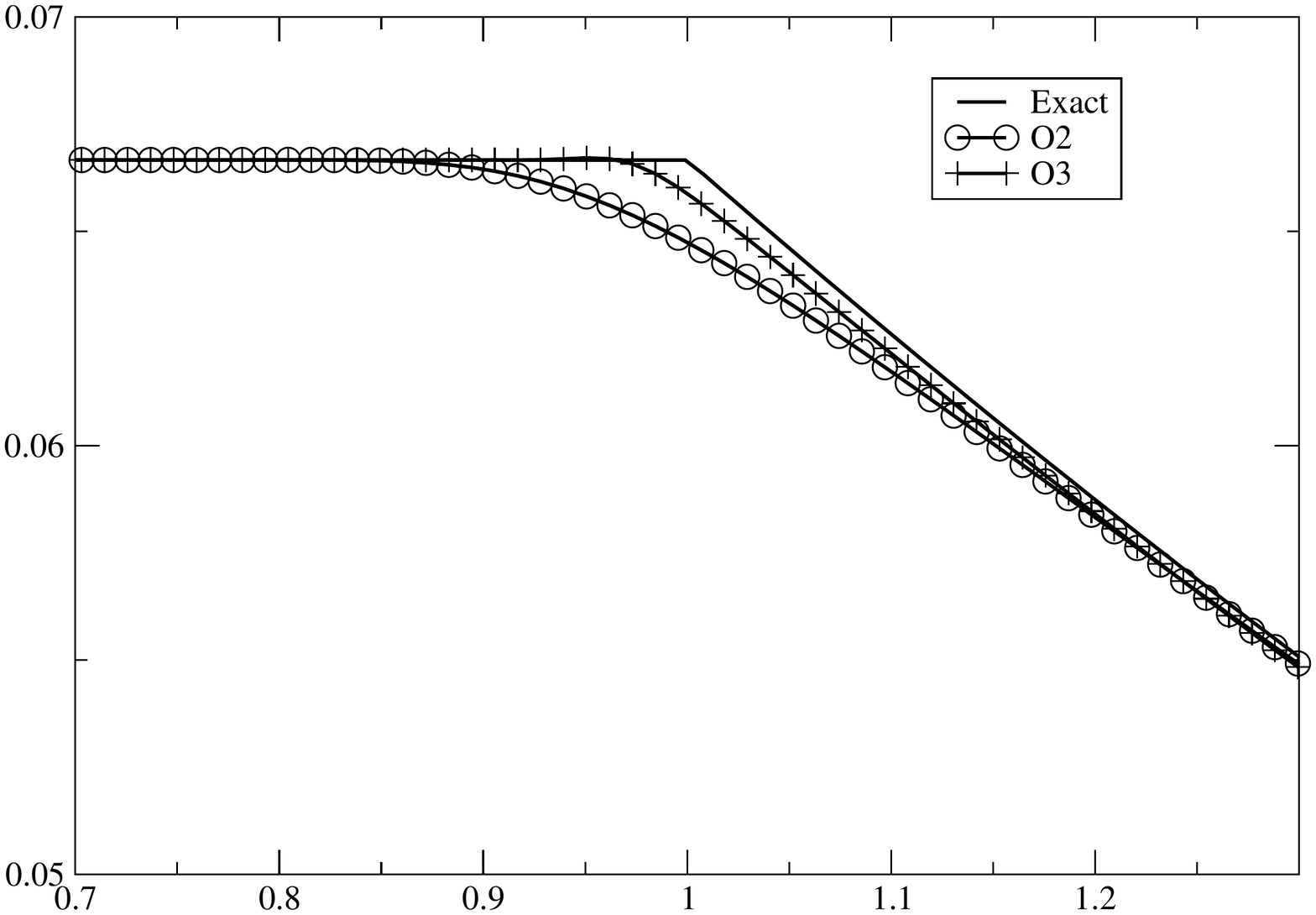}}
\end{center}
\caption{\label{fig:leblanc2} Le Blanc test case, zooms, comparison on the pressure between second order and third order with 400 points}
\end{figure}

At time $t=6$ the shock wave should be at $x=8$: in addition to the extreme conditions, it is generally difficult to get a a correct position of the shock wave; this is why a convergence study is shown in figure \ref{fig:leblanc_conv}. It is performed with 400, 800, 10000 grid points, and the third order SSPRK3 scheme with CFL=0.1. It is compared to the exact solution, and the results are good, see for example \cite{Ramani_2019} for a comparison with other methods, or \cite{loubere} for a comparison with Lagrangian methods.
\begin{figure}[h]
\begin{center}
\subfigure[]{\includegraphics[width=0.45\textwidth]{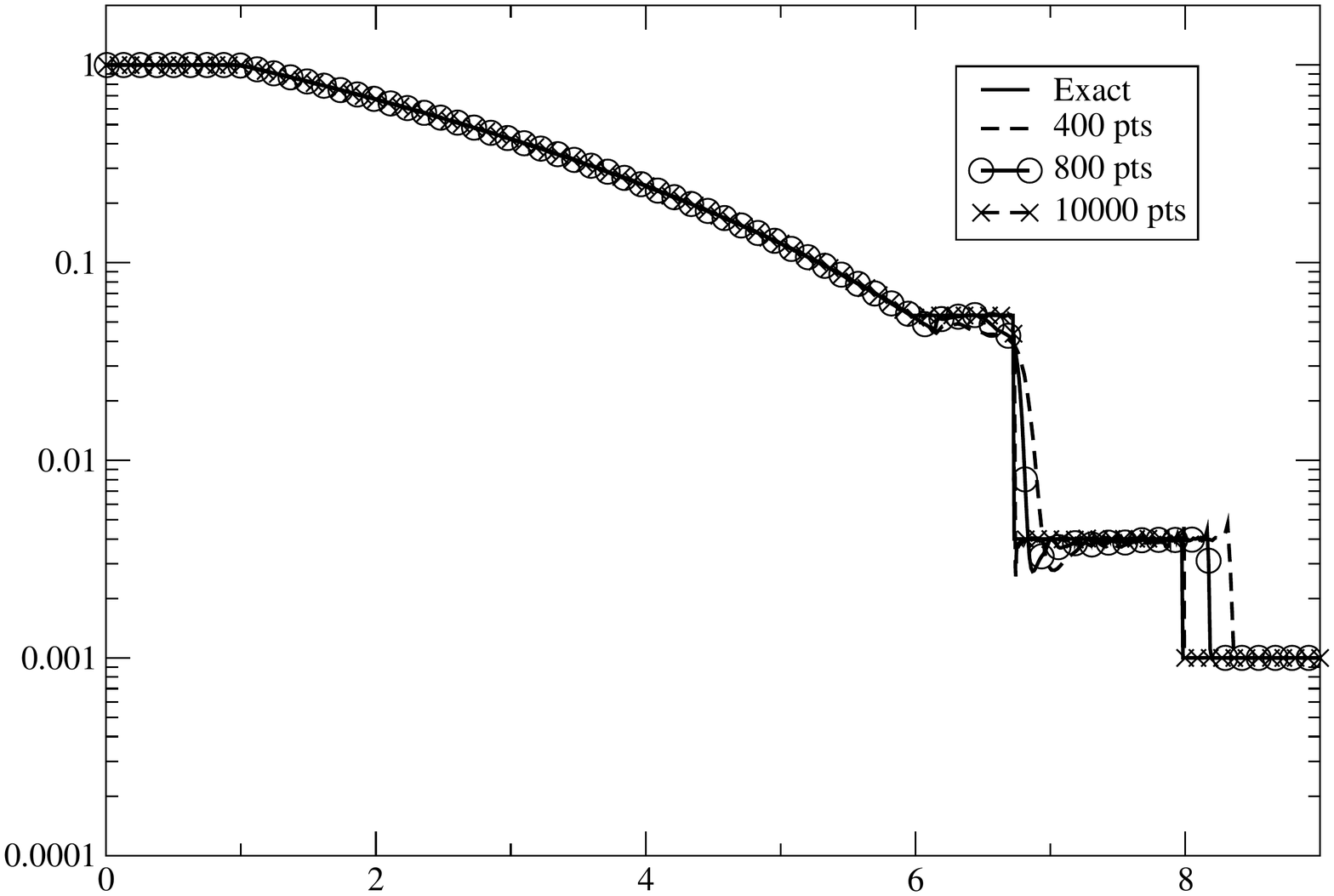}}
\subfigure[]{\includegraphics[width=0.45\textwidth]{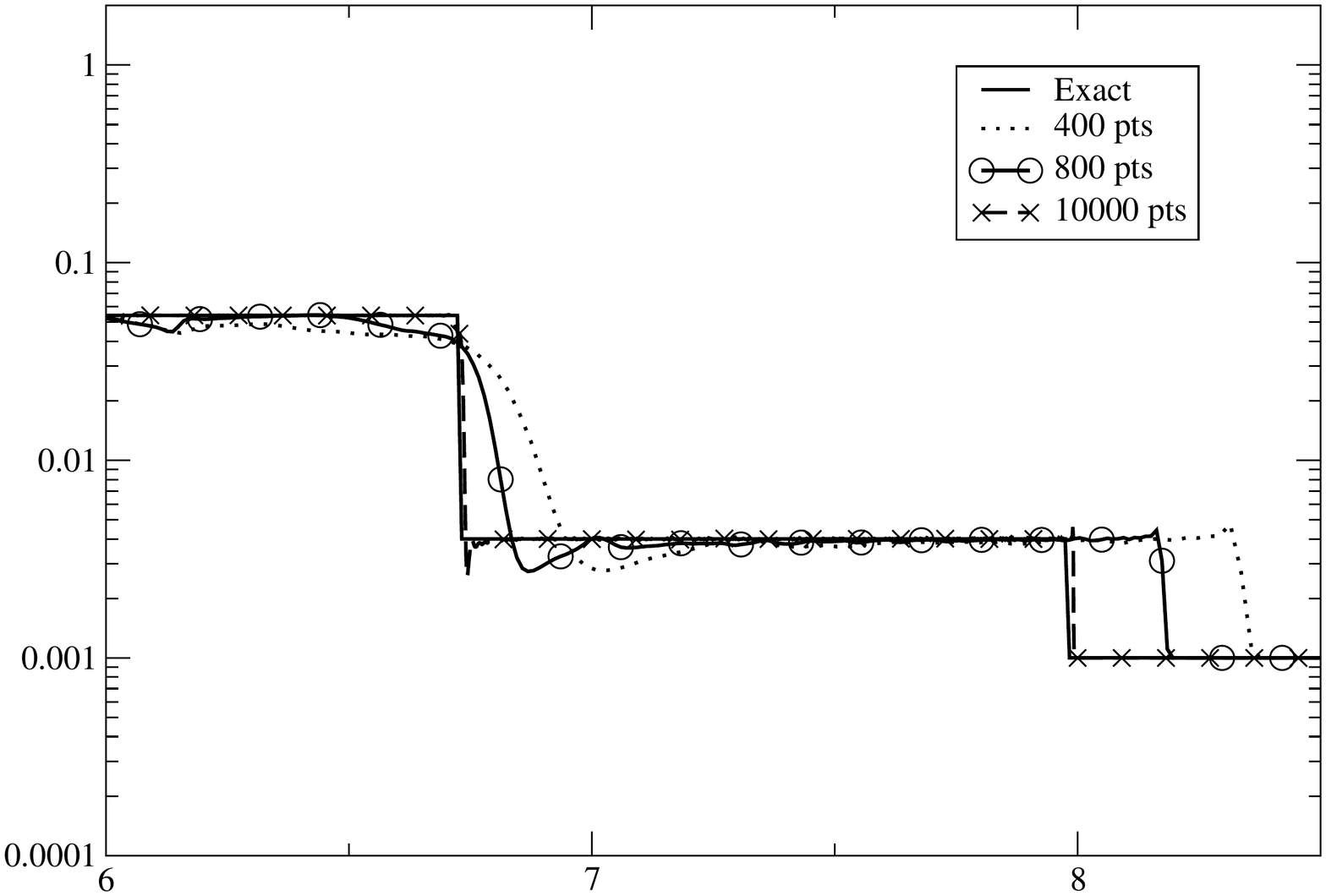}}
\subfigure[]{\includegraphics[width=0.45\textwidth]{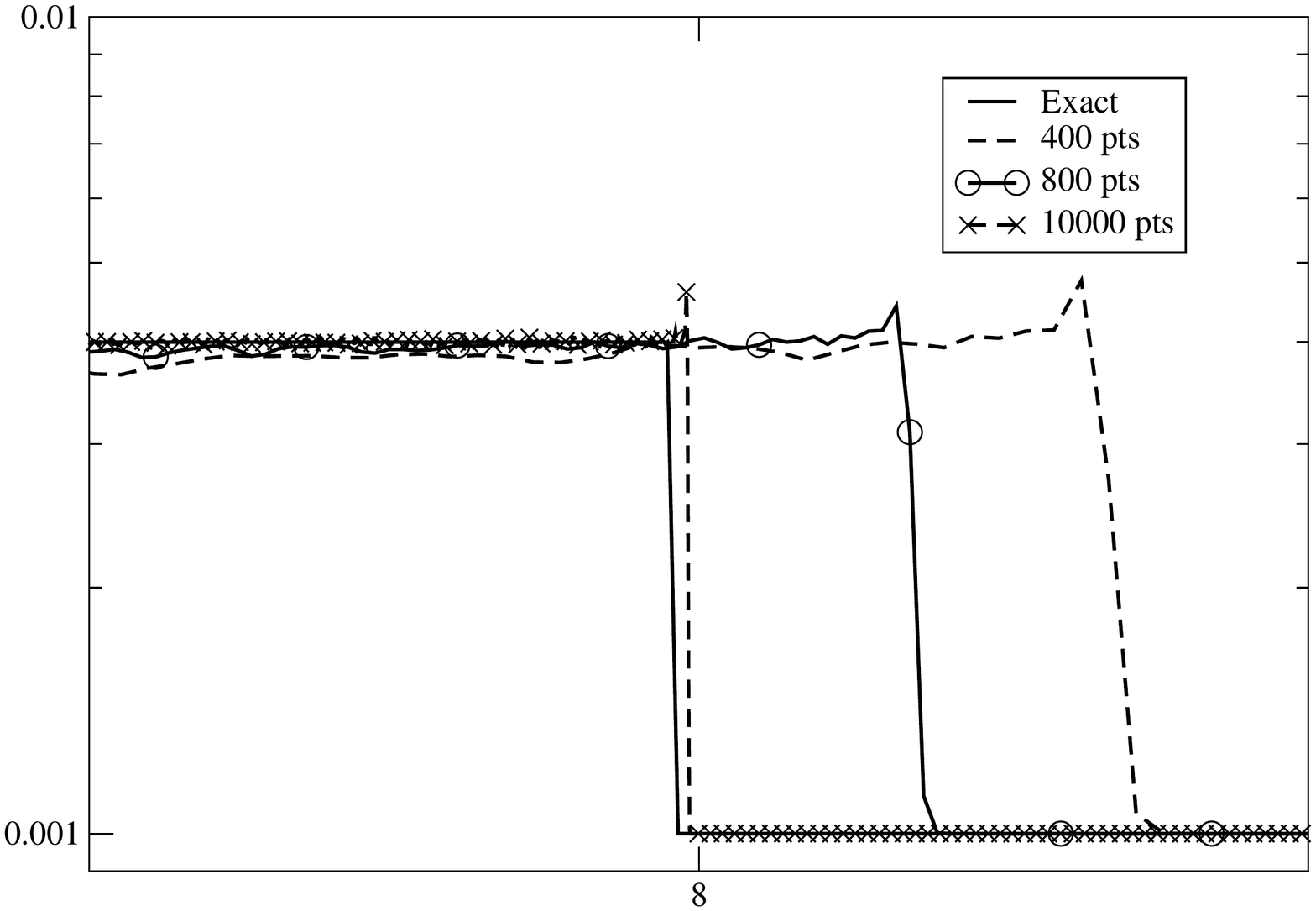}}
\end{center}
\caption{\label{fig:leblanc_conv} Convergence study on the density for the LeBlanc test case.}
\end{figure}

\section{Conclusion}
This study is preliminary and should be seen as  a proof of concept. We show how to combine, without any test, several formulations of the same problem, one conservative and the other ones in non conservative form, in order to compute the solution of hyperbolic systems. The emphasis is mostly put on the Euler equations.

We explain why the formulation leads to a method that satisfies a Lax-Wendroff like theorem. We also propose a way to provide non linearity stability, this method works well but is not yet completely satisfactory.

Besides the theoretical results, we also show numerically that  we get the convergence to the correct weak solution. This is done on standard benchmark problems, some being very challenging.

We intend to extend the method to several space dimensions, and improve the limiting strategy. Different systems, such as the shallow water system, will also be considered.

\section*{Acknowledgements.}
This work was done while the author was partially funded by SNF project  	200020$\_$175784. The support of Inria via the  International Chair of the author at Inria Bordeaux-Sud Ouest is also acknowledged. Discussions with Dr. Wasilij Barsukow are acknowledged, as well as the encouragements of Anne Burbeau (CEA DAM, France).
Last, I would like to thank, warmly, the two anonymous referees: their critical comments have led to big improvements.

\bibliographystyle{unsrt}
\bibliography{biblio}
\appendix
{ \section{Proof of proposition \ref{LxW}.}\label{appendix}
We show the proposition \ref{LxW} in the scalar case, the system case is identical.

We start by some notations: $\R$ is subdivided into intervals $K_{j+1/2}=[x_j,x_{j+1}]$ with $x_j<x_{j+1}$, and $h$ will be the maximum of the length of the $K_{j+1/2}$. On each interval, from the 
point values $u_i$ and $u_{i+1}$, as well as the average $\bar{u}_{j+1/2}$ we can construct a quadratic polynomial . From this and as above, we can construct a globally continuous piecewise quadratic function, that  for simplicity of notations we will denote by $R_{u _\Delta}$.

 Let $T>0$ and a time discretisation $0<t_1<\ldots <t_n<\ldots <t_N\leq T$ of $[0,T]$.  We define $\Delta t_n=t_{n+1}-t_n$ and $\Delta t=\max\limits_n \Delta t_n$.  We are given the sequences $\{u_j^p\}_{j\in \Z}^{p=0\ldots N}$ and $\{{\bar u}_{j+1/2}^n\}_{j\in \Z}^{p=0, \ldots, N}$. We can define  a function $u_{\Delta}$ by:
 $$\text{ if }(x,t)\in [x_j,x_{j+1}]\times [t_{n}, t_{n+1}[, \text{ then } u _\Delta(x,t)=R_{u^n_\Delta}(x).$$
 The set of these functions is denoted by $X _{\Delta}$ and is equipped with the $L^\infty$ and $L^2$ norms.

We have the following lemma
\begin{lemma}\label{weakBV}
Let $T>0$ , $\{t_n\}_{n=0, \ldots , N}$ an increasing  subdivision of $[0,T]$, $[a,b]$ a compact of $\R$. Let $(u _\Delta)_{h}$ a sequence of functions of $X _\Delta$ defined on $\R\times \R^+$. We assume that there exists $C\in \R$ independent of $\Delta$ and $\Delta t$, and $\bu\in L^2_{loc}([a,b]\times [0,T])$ such that
$$\sup\limits _\Delta\sup\limits_{x,t}\vert u _\Delta(x,t)\vert \leq C\quad \text{ and } \lim\limits _{\Delta,  \Delta t\rightarrow 0}\vert u _\Delta-u\vert_{L^2([a,b]\times[0,T])}=0.$$
 Then
\begin{equation}
\label{appendix:1}\lim\limits_{\Delta, \Delta t\rightarrow 0}\sum_{n=0}^N\Delta t_n \bigg [\sum_{j\in \Z}\Delta_{j+1/2}\bigg ( \vert u_{j}^n-\bar{u} _{j+1/2}^{n}\vert + \vert u_{j+1}^n-\bar{u} _{j+1/2}^{n}\vert+\vert u_{j}^n-u_{j+1}^n\vert \bigg )\bigg ] =0.\end{equation}
\end{lemma}
\begin{proof}
First, because the vector space of polynomials of degree 3 on $[x_j,x_{j+1}]$ is finite dimensional and with a dimension independent of $j$, there exists $C_1$ and $C_2$ such that 
{\begin{equation*}
\begin{split}C_1 \Delta_{j+1/2}\bigg ( \vert u_{ j}^n-\bar{u} _{j+1/2}^{n}\vert + \vert u_{ j+1}^n-\bar{u} _{j+1/2}^{n}\vert\bigg )\leq  \int_{x_j}^{x_{j+1}}& \vert u _\Delta(x,t_n)-{\bar u} _{j+1/2}^n\vert  \; dx\\
&\leq C_2  \Delta_{j+1/2}\bigg ( \vert u_{ j}^n-\bar{u} _{j+1/2}^{n}\vert + \vert u_{j+1}^n-\bar{u} _{j+1/2}^{n}\vert\bigg )
\end{split}
\end{equation*} }
so that
$$\sum_{n=0}^K \Delta t_n \sum\limits_{j, K_{j+1/2}\subset [a,b]} \Delta_{j+1/2}\bigg ( \vert u_{ j}^n-\bar{u} _{j+1/2}^{n}\vert + \vert u_{ j+1}^n-\bar{u} _{j+1/2}^{n}\vert\bigg )\leq C_1^{-1} \int_0^T\int_a^b \vert  u _\Delta-\bar u _\Delta\vert \; dx,$$
where for simplicity we have introduced $\bar u _\Delta$ the function defined by:
$$\text{ if }(x,t)\in [x_j,x_{j+1}[\times [t_n, t_{n+1}[, \bar u _\Delta(x,t)={\bar u}_{j+1/2}^{n}.$$
Then we rely on classical arguments of functional analysis: since $(u _\Delta)$ is bounded, and since $L^\infty([a,b]\times [0,T]) \subset L^1([a,b]\times [0,T]) $, there exists $u'\in L^\infty([a,b]\times [0,T]$ such that $u _\Delta\rightarrow u'$ in the weak-$\star$ topology. Similarly, there exists $\bar u\in L^\infty([a,b]\times [0,T])$  such that ${\bar u} _\Delta\rightarrow \bar u$ for the weak-$\star$ topology.

Since $u _\Delta\rightarrow u$ in $L^2_{loc}$, we have $u'=u$ because $[a,b]\times [0,T]$ is bounded and $C_0^\infty([a,b]\times[0,T])$ is dense in $L^1([a,b]\times [0,T])$. Let us show that $\bar u=u$. let $\varphi\in C_0^\infty(\R\times \R^+)$. We have, setting {
{$$\bar \varphi_{j+1/2}^n=\dfrac{1}{\Delta_{j+1/2}\Delta t_n}\int_{t_n}^{t_{n+1}} \int_{x_j}^{x_{j+1}} \varphi(x,t)\; dxdt,$$}
\begin{equation*}
\begin{split}
\int_0^T \int_a^b \big ( \bar u _\Delta-u _\Delta\big ) \varphi\; dx dt&=\sum_n \sum\limits_{a\leq x_j<x_{j+1}\leq b}
\int_{t_n}^{t_{n+1}}\int_{x_j}^{x_{j+1}} \big ( \bar u _\Delta-u_\Delta\big ) \varphi\; dx dt\\
&=\sum_n \sum\limits_{a\leq x_j<x_{j+1}\leq b}\Bigg ( \int_{t_n}^{t_{n+1}}\int_{x_j}^{x_{j+1}} \big ( \bar u _\Delta-u_\Delta\big ) \varphi\; dx dt -
\int_{t_n}^{t_{n+1}}\int_{x_j}^{x_{j+1}} \big ( \bar u _\Delta-u_\Delta\big ) \bar \varphi_{j+1/2}^n\; dx dt\Bigg )\\
&=\sum_n \sum\limits_{a\leq x_j<x_{j+1}\leq b}\int_{t_n}^{t_{n+1}}\int_{x_j}^{x_{j+1}}\big ( \bar u _\Delta-u_\Delta\big )\big ( \varphi-\bar\varphi_{j+1/2}^n\big ) \; dx dt
\end{split}
\end{equation*}
using the fact that for any $[x_j, x_{j+1}]\times [t_n, t_{n+1}]$, we have $ \int_{t_n}^{t_{n+1}}\int_{x_j}^{x_{j+1}} \big ( \bar u _\Delta-u_\Delta\big ) \;dx\; dt=0$.}

Since $\varphi\in C_0^\infty(\R\times \R^+)$, there exists $C$ that depends only on $\Vert \dfrac{d\varphi}{dx}\Vert_{L^\infty(\R\times \R+)}$ such that
$$\bigg \vert \int_{x_j}^{x_{j+1}} \big (\bar u _\Delta-u_\Delta\big )\big ( \varphi-\bar\varphi\big ) \; dx dt\bigg \vert\leq \Delta t \Delta_{j+1/2}\; \Delta  \max\limits_{j\in \Z, n\leq N}\big ( \vert u_j^n\vert , \vert {\bar u}_{j+1/2}^n\vert \big )\leq \Delta t \Delta^2\max\limits_{j\in \Z, n\leq N}\big ( \vert u_j^n\vert , \vert {\bar u}_{j+1/2}^n\vert \big ) $$
and then,
$$\bigg \vert \int_0^T \int_a^b \big ({\bar u} _\Delta-u _\Delta\big ) \varphi \; dx dt\bigg \vert \leq C \; \Delta$$
and passing to the limit, $\bar u=u'$. Since a subsequence of $u _\Delta$ converges to $u$ in $L^2$, we have $\bar u=u'=u$.

The same method shows that $(u _\Delta^2)$ and $({\bar u} _\Delta^2)$ have the same weak-$\star$ limit. Let us show it is $u^2$. Since $C_0^\infty([a,b]\times [0,T]$ is dense in $L^1([a,b]\times [0,T]$), and since $u _\Delta^2$ is bounded independently of $\Delta$ and $\Delta t$, we can choose functions $\phi$ in $C_0^\infty([a,b]\times [0,T])$. This test function is bounded in $[a,b]\times [0,T]$ and then, we  have, at least for a subsequence,
$$\int_{a}^b\int_{0}^T \vert u-u _\Delta\vert^2 \; dx dt \rightarrow 0,$$
and then
$$
\int_a^b\int_{0}^T u _\Delta^2\; dx dt-2\int_a^b\int_{0}^T u _\Delta \, u\; dx dt+ \int_a^b\int_{0}^T u^2\; dx dt\rightarrow 0.$$
By the Cauchy-Schwarz inequality, $u\phi\in L^1([a,b]\times [0,T])$: the second term tends towards
$$\int_a^b\int_{0}^T u^2\; dxdt,$$ so that
$$\int_a^b\int_{0}^T u _\Delta^2\; dx dt- \int_a^b\int_{0}^T u^2\; dx dt\rightarrow 0.$$
and $u _\Delta^2\rightarrow u^2$ in $L^\infty$ weak-$\star$.

Last, again by the same argument for $\phi=1$, since $u _\Delta^2\rightarrow u^2$ in  $L^\infty$ weak-$\star$, we get
$$\int_a^b\int_{0}^T\vert \bar u _\Delta-u\vert^2 \; dx dt\rightarrow 0,$$
and finally
$$\int_a^b\int_{0}^T\vert \bar u _\Delta-u _\Delta\vert^2 \; dx dt\rightarrow 0.$$
Since $[a,b]\times [0,T]$ is bounded, $L^1([a,b]\times [0,T])\subset L^2([a,b]\times [0,T])$, we obtain
$$
\lim\limits_{ \Delta, \Delta t\rightarrow 0}\sum_{n=0}^N\Delta t_n \bigg [\sum_{j\in \Z}\Delta_{j+1/2}\bigg ( \vert u_{ j}^n-\bar{u} _{j+1/2}^{n}\vert + \vert u_{j+1}^n-\bar{u} _{j+1/2}^{n}\vert \bigg )\bigg ] =0$$
From this we get  \eqref{appendix:1} because 
$$\vert u_{ j}^n-u_{ j+1}^n\vert \leq \vert u_{ j}^n-\bar{u} _{j+1/2}^{n}\vert+\vert u_{ j+1}^n-\bar{u} _{j+1/2}^{n}\vert.$$
\end{proof}

Then we can proof proposition \ref{LxW}. We proceed the proof in several lemma.
\begin{lemma}
Under the conditions of proposition \ref{LxW}, for any $\varphi\in C_0^\infty(\R\times \R^+)$ we have 
\begin{equation*}
\begin{split}
\lim\limits_{\Delta t\rightarrow 0, \Delta\rightarrow 0} \sum\limits_{n=0}^\infty\sum\limits _{[x_j, x_{j+1}], j\in \Z}&\dfrac{\Delta_{j+1/2}}{6}\bigg ( \varphi_{j+1}^n(\bu_{j+1}^{n+1}-\bu_{j+1}^n)+4\varphi_{j+1/2}^n (\bu_{j+1/2}^{n+1}-\bu_{j+1/2}^n)+ \varphi_{j}^n(\bu_{j}^{n+1}-\bu_j^n)\bigg )\\&=
-\int_{\R\times \R^+} \dpar{\varphi}{t} u \; dx dt+\int_\R \varphi(x,0) u_0\; dx dt.\end{split}\end{equation*}
\end{lemma}
\begin{proof}
This is a simple adaptation of the classical proof, see for example \cite{MR1304494}. We have, using that
$$\delta u_{j+1/2}=\frac{3}{2}\overline{\delta u}_{j+1/2}-\frac{\delta u_j+\delta u_{j+1}}{4}$$
and the compactness of the support of $\varphi$, 
\begin{equation*}
\begin{split}
\sum\limits_{n=0}^\infty\sum\limits _{[x_j, x_{j+1}], j\in \Z}&\dfrac{\Delta_{j+1/2}}{6}\bigg ( \varphi_{j+1}^n(\bu_{j+1}^{n+1}-\bu_{j+1}^n)+4\varphi_{j+1/2}^n (\bu_{j+1/2}^{n+1}-\bu_{j+1/2}^n)+ \varphi_{j}^n(\bu_{j}^{n+1}-\bu_j^n)\bigg )\\&=
\underbrace{\sum\limits_{n=0}^\infty \sum\limits _{[x_j, x_{j+1}], j\in \Z}\Delta_{j+1/2} \varphi_{j+1/2}\overline{\delta u}_{j+1/2}}_{(I)} \\
&\qquad + \underbrace{\sum\limits_{n=0}^\infty
\sum\limits_{j\in \Z}\bigg ( \frac{\Delta_{j+1/2}}{6} \big ( \varphi_{j}-\varphi_{j+1/2}\big ) +\frac{\Delta_{j-1/2}}{6}\big ( \varphi_{j}-\varphi_{j-1/2}\big )\bigg )\delta u_j
}_{(II)}
\end{split}
\end{equation*}
where $\overline{\delta \bu}_{j+1/2}=\bar\bu_{j+1/2}^{n+1}-\bar\bu_{j+1/2}^{n}$ and $\delta \bu_{j} =\bu_j^{n+1}-\bu_j^n$.
The first part, $(I)$,  is the classical term, and under the condition of the lemma, converges to 
$$-\int_{\R\times \R^+} \dpar{\varphi}{t} u \; dx dt+\int_\R \varphi(x,0) u_0\; dx dt.$$
Since $\varphi\in C_0^\infty(\R\times \R^+)$, there exists $C$ that depends only on the $L^\infty$ norm of the first derivative of $\varphi$ such that the term $(II)$  can be bounded by
\begin{equation*}
\begin{split}
\bigg \vert \sum\limits_{n=0}^N\Delta t_n
\sum\limits_{j\in \Z}\bigg ( \frac{\Delta_{j+1/2}}{6} \big ( \varphi_{j}-\varphi_{j+1/2}\big ) &+\frac{\Delta_{j-1/2}}{6}\big ( \varphi_{j}-\varphi_{j-1/2}\big )\bigg )\delta u_j
\bigg \vert \leq C
\sum\limits_{n=0}^N \Delta t \;  \Delta
\sum\limits_{j\in \Z}\Delta_{j+1/2} \vert \delta u_j\vert \\
& \quad \leq C T(b-a) \Delta \max\limits_{j,p\in \N}\vert\bu_j^p\vert.
\end{split}
\end{equation*}
This tends to zero because $\max\limits_{j\in \Z,p\in \N}\vert\bu_j^p\vert$ is finite.
\end{proof}
\begin{lemma}
Under the assumptions of proposition \ref{LxW},
$$\lim\limits_{\Delta t, \Delta\rightarrow 0}\sum\limits_{n=0}^\infty\sum\limits _{[x_j, x_{j+1}], j\in \Z}\Delta t_n \varphi_{j+1/2} \delta_{j+1/2} \bbf=-\int_{\R\times \R^+} \dpar{\varphi}{x}f(u) \; dx dt.$$
\end{lemma}
\begin{proof}
This is again a simple adaptation of the classical proof since $\delta_{j+1/2}f=f(u_{j+1})-f(u_j)$. We have {
$$\sum\limits _{[x_j, x_{j+1}], j\in \Z}\varphi_{j+1/2} \delta_{j+1/2} \bbf=-\sum\limits_{ j\in \Z}\bbf(\bu_j)\big ( \varphi_{j+1/2}-\varphi_{j-1/2}\big ) ,$$}
Then using the boundedness of the solution and Lebesgue dominated convergence theorem, we get the result.
\end{proof}

Then we have
\begin{lemma}
Under the conditions of proposition \ref{LxW}, we have
$$\lim\limits_{\Delta t, \Delta \rightarrow 0}\bigg (\sum\limits_{n\in \N}\sum\limits_{j\in \Z} \delta_{j}^{n+1/2} \bu
\bigg \{\Delta_{j+1/2}\big ( \varphi_j-\varphi_{j+1/2}\big ) +\Delta_{j-1/2}\big (\varphi_j-\varphi_{j-1/2}\big )\bigg \}\Bigg )=0$$
\end{lemma}

\begin{proof}
Since $\varphi\in C_0^\infty(\R\times \R^+)$, there exists $C$ that depends only on the first derivative of $\varphi$ such that
$$\vert \Delta_{j+1/2}\big ( \varphi_j-\varphi_{j+1/2}\big ) +\Delta_{j-1/2}\big (\varphi_j-\varphi_{j-1/2}\big )\vert \leq C  \Delta\; (\Delta_{j+1/2}+\Delta_{j-1/2}).$$
Then using \eqref{schemedisc:3}, we see that 
$$\delta_j^{n+1/2}:=\bu_j^{n+1}-\bu_j^n=-\dfrac{\Delta t_n}{\Delta_j } \delta_x\bu_{j},$$
so that 
\begin{equation*}
\begin{split}
\Vert \sum\limits_{j\in \Z} \delta_{j}^{n+1/2} \bu
\bigg \{\Delta_{j+1/2}\big ( \varphi_j-\varphi_{j+1/2}\big ) &+\Delta_{j-1/2}\big (\varphi_j-\varphi_{j-1/2}\big )\bigg \}\Bigg )\Vert 
\leq C  \Delta t_n\sum\limits_{j\in \Z} \Vert \delta_{x}^{n+1/2} \bu\Vert \Delta_j\\
&=C  \Delta \Delta t_n\sum\limits_{j\in \Z}   \Vert \delta_x\bu_{j}\Vert=C\Delta \sum\limits_{j\in Z}\sum\limits_{l=-p}^{l=p}\vert \bu_{j+l}-\bar\bu_{j+l+1/2}\vert 
\end{split}
\end{equation*}
using the Lipschitz continuity of the fluctuations and the regularity of the transformation $\bv\mapsto \bu$ together with the boundedness of the solution.
Then, from lemma \ref{weakBV}, we see that the results holds true.
\end{proof}

Then ends the proof of proposition \ref{LxW}.
}

{\section{Linear stability analysis}\label{appendix:linearstability}

The scheme writes, setting $\lambda=a\tfrac{\Delta t}{\Delta x}$ and assuming $a>0$,
\begin{equation*}
\begin{split}
u_{j}^{n+1}&=u^n_j-2\lambda \delta_ju^n\\
\bar u_{j+1/2}^{n+1}&=\bar u_{j+1/2}^n-\lambda \big ( u_{j+1}^n-u_j^n\big )
\end{split}
\end{equation*}
on  with periodicity $1$. We set $\Delta x=\tfrac{1}{N}$
It is more convenient to work with point values only, and we will use the form:
\begin{equation*}
\begin{split}
u_{j}^{n+1}&=u^n_j-2\lambda \delta_ju^n\\
u_{j+1/2}^{n+1}&=u_{j+1/2}^n -\lambda \bigg ( \frac{3}{2}\big ( u_{j+1}^n-u_j^n\big )-\frac{1}{4}\big ( \delta_j u^n+\delta_{j+1}u^n\big )\bigg ).
\end{split}
\end{equation*}
We perform a linear stability analysis by Fourier analysis. What is not completely standard is that the grid points do not play the same role. For ease of notations, we double the indices, this avoids to  have half integer in the Fourier analysis. In other points, the quantities $u_j$ associated to the grid points $x_j$ are denoted by $u_{2j}$: this will be the even terms. Those associated to the intervals $[x_j,x_{j+1}]$, i.e. $\bar u_{j+1/2}$ and $u_{j+1/2}$ will be denoted as $\bar u_{2j+1}$ and $u_{2j+1}$, they are the odd  terms, so that we use
{\color{green}\begin{equation}\label{stab:forward}
\begin{split}
u_{2j}^{n+1}&=u^n_j-2\lambda \delta_{2j}u^n\\
u_{2j+1}^{n+1}&=u_{2j+1}^n -\lambda \bigg ( \frac{3}{2}\big ( u_{2j+2}^n-u_{2j}^n\big )-\frac{1}{4}\big ( \delta_{2j}u^n+\delta_{2j+2}u^n\big )\bigg ).
\end{split}
\end{equation} }
We have the Parseval equality, 
$$\frac{1}{2N}\sum\limits_{0}^{2N-1} u_j^2=\sum_{k=0}^{2N-1} \vert \hat{u}(k)\vert^2,$$
with
$$\hat{u}(k)=\frac{1}{2N}\sum\limits_{j=0}^{2N-1} u_je^{2i\pi\frac{kj}{2N}}=\frac{1}{2}\big ( \hat{u}_o(k)+\hat{u}_e(k)\big )$$
with, setting $\omega=e^{\frac{i\pi}{N}}$
\begin{equation*}
\hat{u}_o(k)=\frac{1}{N}\sum\limits_{j=0}^{N-1} u_{2j+1} \omega^{(2j+1)k}, \quad
\hat{u}_e(k)=\frac{1}{N}\sum\limits_{j=0}^{N-1} u_{2j} \omega^{(2j)k}
\end{equation*}
The usual shift operator $[S(u)]_j=u_{j+1}$ gives:
\begin{equation*}
\hat{S(u)}_o=\omega^{-k} \hat{u}_e, \quad \hat{S(u)}_e=\omega^{-k} \hat{u}_o.
\end{equation*} 
Using this, we see that the Euler forward method \eqref{stab:forward} gives
\begin{equation}\label{stab:2}
\begin{pmatrix}
\hat{u}_o^{n+1}\\
\hat{u}_e^{n+1}\end{pmatrix}=\bigg ( \text{Id}-\lambda H\bigg ) \begin{pmatrix}
\hat{u}_o^{n}\\
\hat{u}_e^{n}\end{pmatrix}
\end{equation}
with
$$H_1(k)=\begin{pmatrix}
\frac{1+\omega^{2k}}{4}& \frac{5\omega^{-k}+7\omega^k}{4}\\
0& 2\big (1-\omega^k\big )\end{pmatrix}
$$ for the first order in space scheme,
$$H_2(k)=
\begin{pmatrix}
\frac{1+\omega^{2k}}{2} & \frac{9}{8}\omega^{-k}-2\omega^{-k}-\frac{\omega^{3k}}{8}\\
0& 2\big (1-\omega^k\big )\end{pmatrix}
$$ for the second order scheme and
$$H_3(k)
=
\begin{pmatrix}
\frac{\omega^{-2k}}{12}+\frac{1}{24}+\frac{\omega^{2k}}{12}-\frac{\omega^{4k}}{24} &
\frac{31}{24}\omega^{-k}-\frac{3}{2}\omega^k+\frac{\omega^{2k}}{12}+\frac{5}{24}\omega^{3k}\\
0& 2\big (1-\omega^k\big )\end{pmatrix}
$$ for the third order in time space.

Combined with the RK time stepping we end up with an update of the form
$$\begin{pmatrix}
\hat{u}_o^{n+1}\\
\hat{u}_e^{n+1}\end{pmatrix}=G_k\begin{pmatrix}
\hat{u}_o^{n}\\
\hat{u}_e^{n}\end{pmatrix}$$
and writing $$G_k=\begin{pmatrix}\alpha_k & \beta_k\\
\gamma_k & \delta_k \end{pmatrix}$$ we end up with
\begin{equation*}
\begin{split}
\hat{u}_o^{n+1}&=\alpha_k \hat{u}_o^{n}+\beta_k \hat{u}_e^{n}\\
\hat{u}_e^{n+1}&=\gamma_k \hat{u}_o^{n}+\delta_k \hat{u}_e^{n}\\
\end{split}
\end{equation*}
so that
$$\hat{u}^{n+1}=\frac{\alpha_k+\gamma_k}{2}\hat{u}_o^n+\frac{\beta_k+\delta_k}{2}\hat{u}_e^n,$$
from which we get
$$\vert \hat{u}^{n+1}\vert^2=\frac{1}{4}\overline{  \begin{pmatrix} \hat{u}_o^n & \hat{u}_e^n\end{pmatrix} } M_k
\begin{pmatrix} \hat{u}_o^n \\ \hat{u}_e^n\end{pmatrix}
$$
with
$$M_k= \frac{1}{4}\begin{pmatrix} \vert \alpha_k+\gamma_k\vert^2 & \big(\alpha_k+\gamma_k\big )\overline{ \big (\beta_k+\delta_k \big )}\\

\overline{ \big(\alpha_k+\gamma_k\big )} { \big (\beta_k+\delta_k \big )} & \vert  \beta_k+\delta_k \vert^2\end{pmatrix}
$$
We have stability if the spectral radius of these matrices is always $\leq 1$, and we immediately see that
$$\rho(M_k)=\frac{1}{4}\bigg ( \vert \alpha_k+\gamma_k\vert^2+\vert  \beta_k+\delta_k \vert^2\bigg ).$$

After calculations, we see that the stability limits are:
\begin{itemize}
\item First order scheme, $\vert \lambda\vert\leq 0.92$,
\item Second order scheme, $\vert\lambda\vert\leq 0.6$,
\item Third order scheme, $\vert\lambda\vert\leq 0.5$.
\end{itemize}
}
\section{Some numerical results on irregular meshes}\label{sec:irreg}
In order to support the theoretical analysis of the method, we have applied it  on irregular meshes. The goal is to show that even here, one \remiIII{gets} convergence of the solution to a weak solution that appears to be the right one. Since we use the same schemes, there is no hope to get anything but first order accuracy. Accuracy on irregular meshes will be the topic of future work.
The mesh is defined by: for $0\leq i\leq N$, 
$$ y_0=0, \quad y_{i+1}=y_i+\Delta y, \quad \Delta y=\dfrac{1+\epsilon_i/2}{N}$$
and 
$$\epsilon_0=-1, \epsilon_{i+1}=-\epsilon_i.$$
Then we define the actual mesh by
$$x_i=\dfrac{y_i}{y_N}.$$

On the Sod problem, with $N=10000$, we get the results of the figure \ref{sod:irregular}. 
\begin{figure}[h]
\subfigure[$\rho$]{\includegraphics[width=0.45\textwidth]{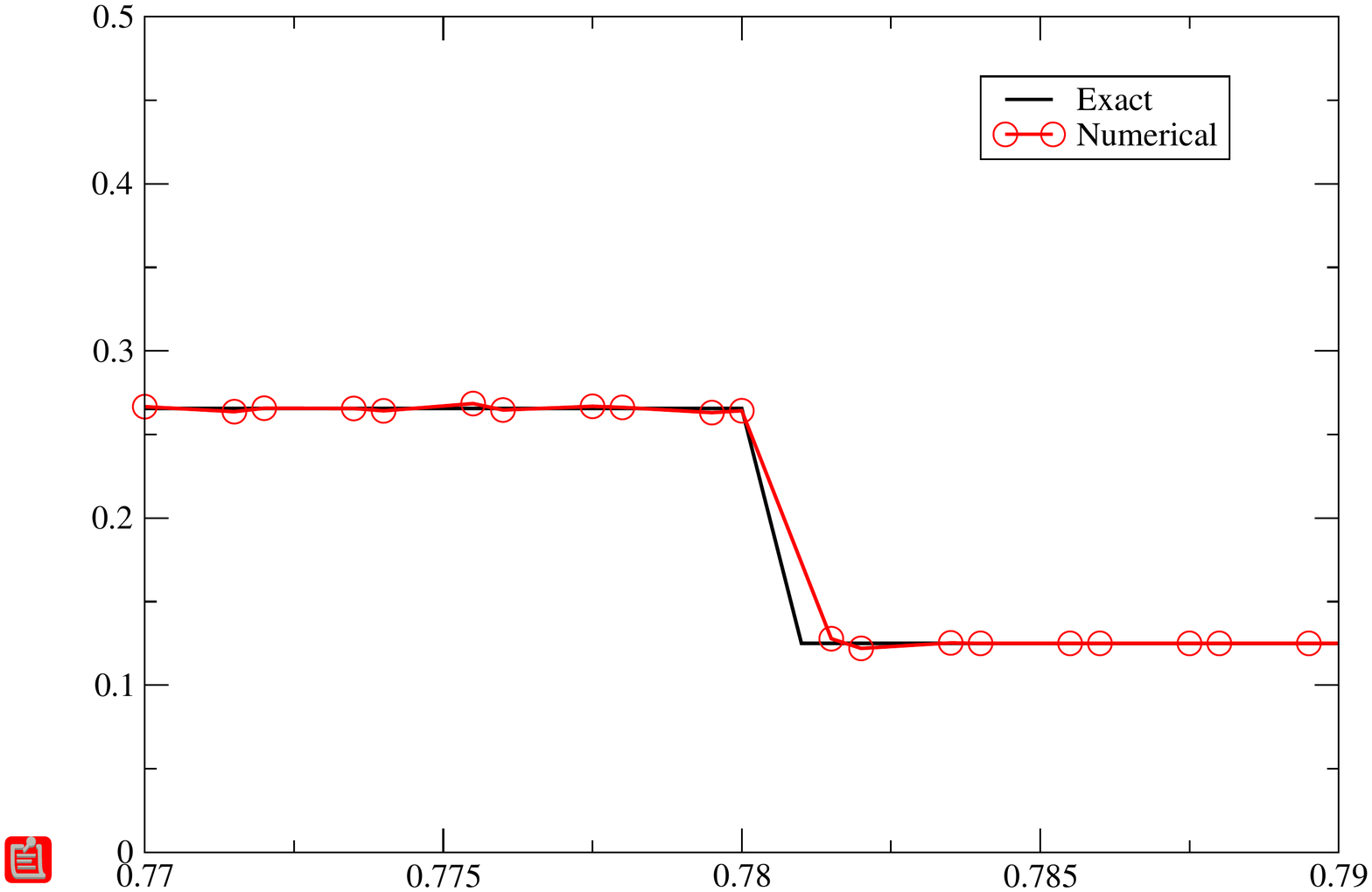}}
\subfigure[$p$]{\includegraphics[width=0.45\textwidth]{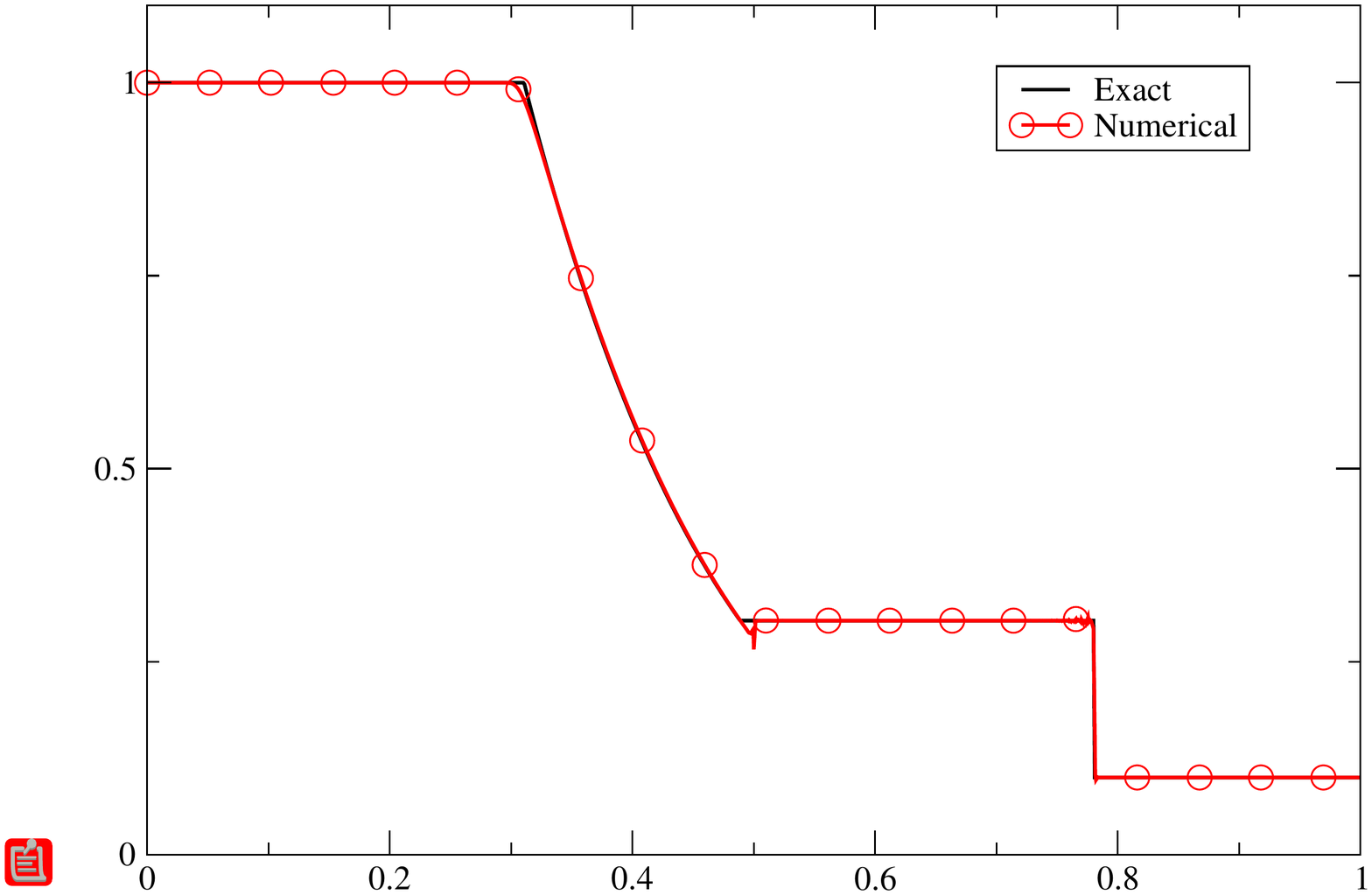}}
\subfigure[$u$]{\includegraphics[width=0.45\textwidth]{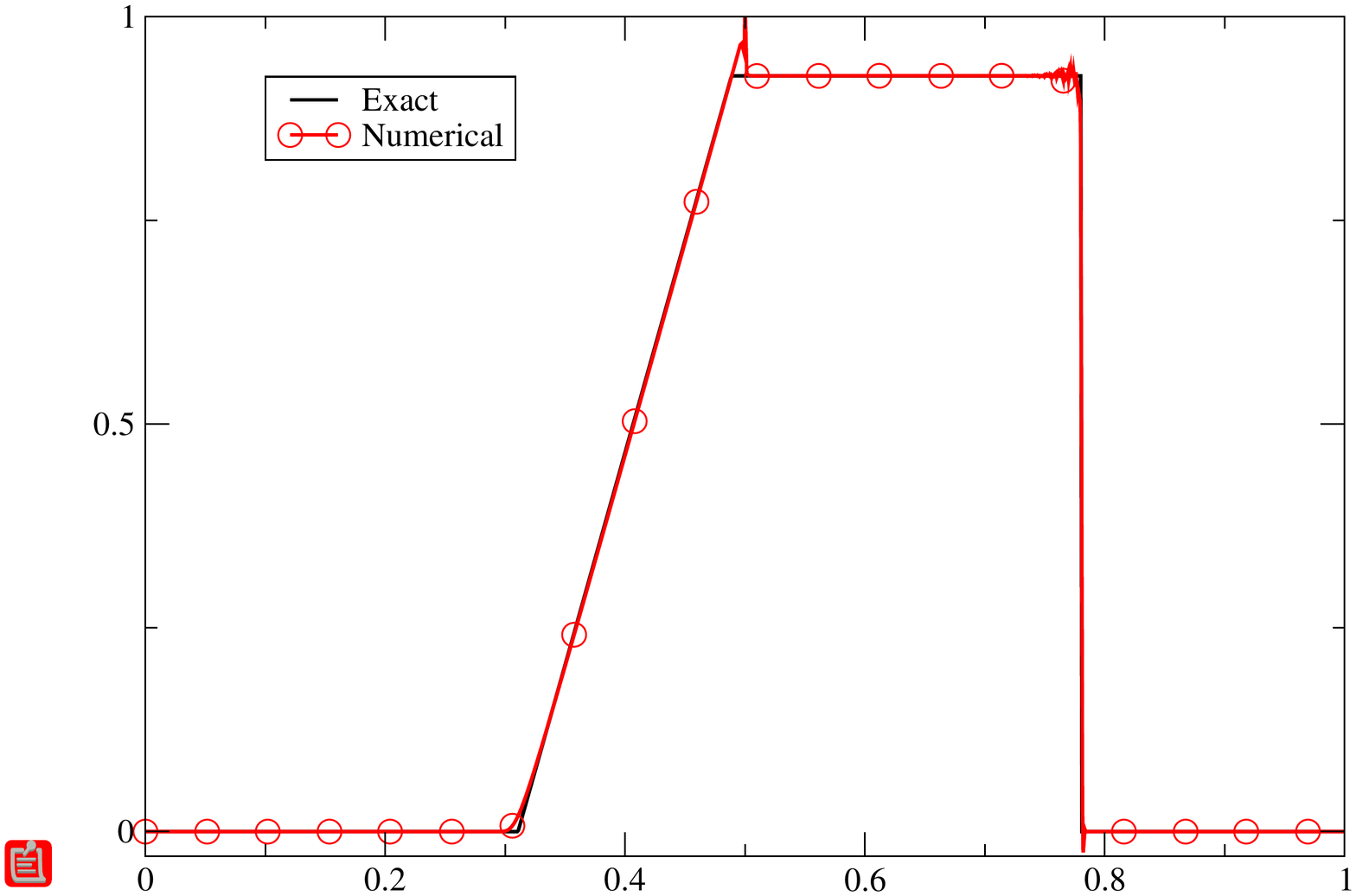}}
\subfigure[$s$]{\includegraphics[width=0.45\textwidth]{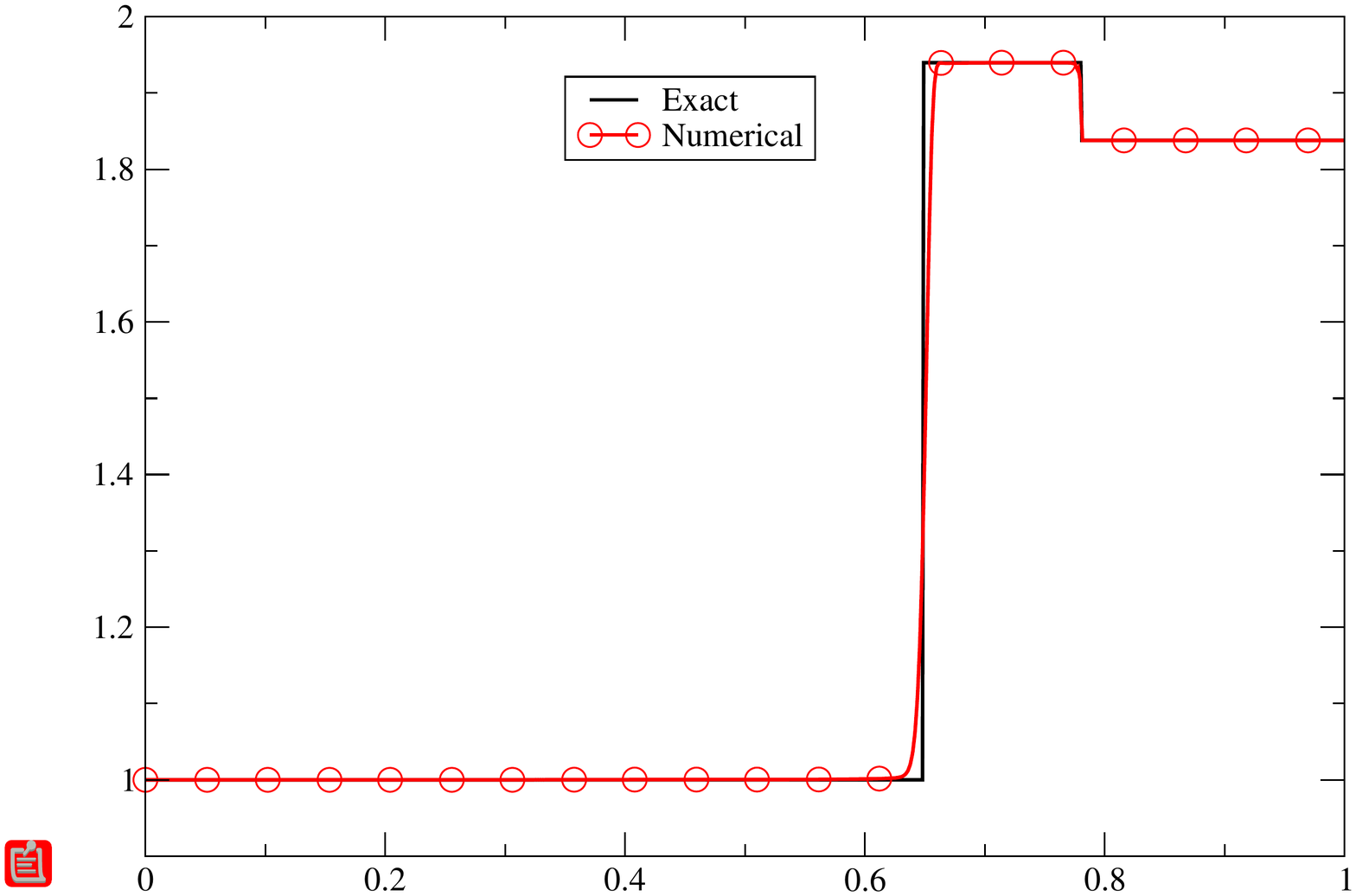}}
\caption{\label{sod:irregular} Plot of the density, pressure, velocity and the pressure. This is obtained with the "third order" with MOOD test on the density and the pressure.}
\end{figure}
On figure \ref{sod:irregular:zoom}, we show a zoom of the density around the shock wave. The discretisation points as well as the numerical and exact solution are shown.
\begin{figure}[h]
\begin{center}
{\includegraphics[width=0.45\textwidth]{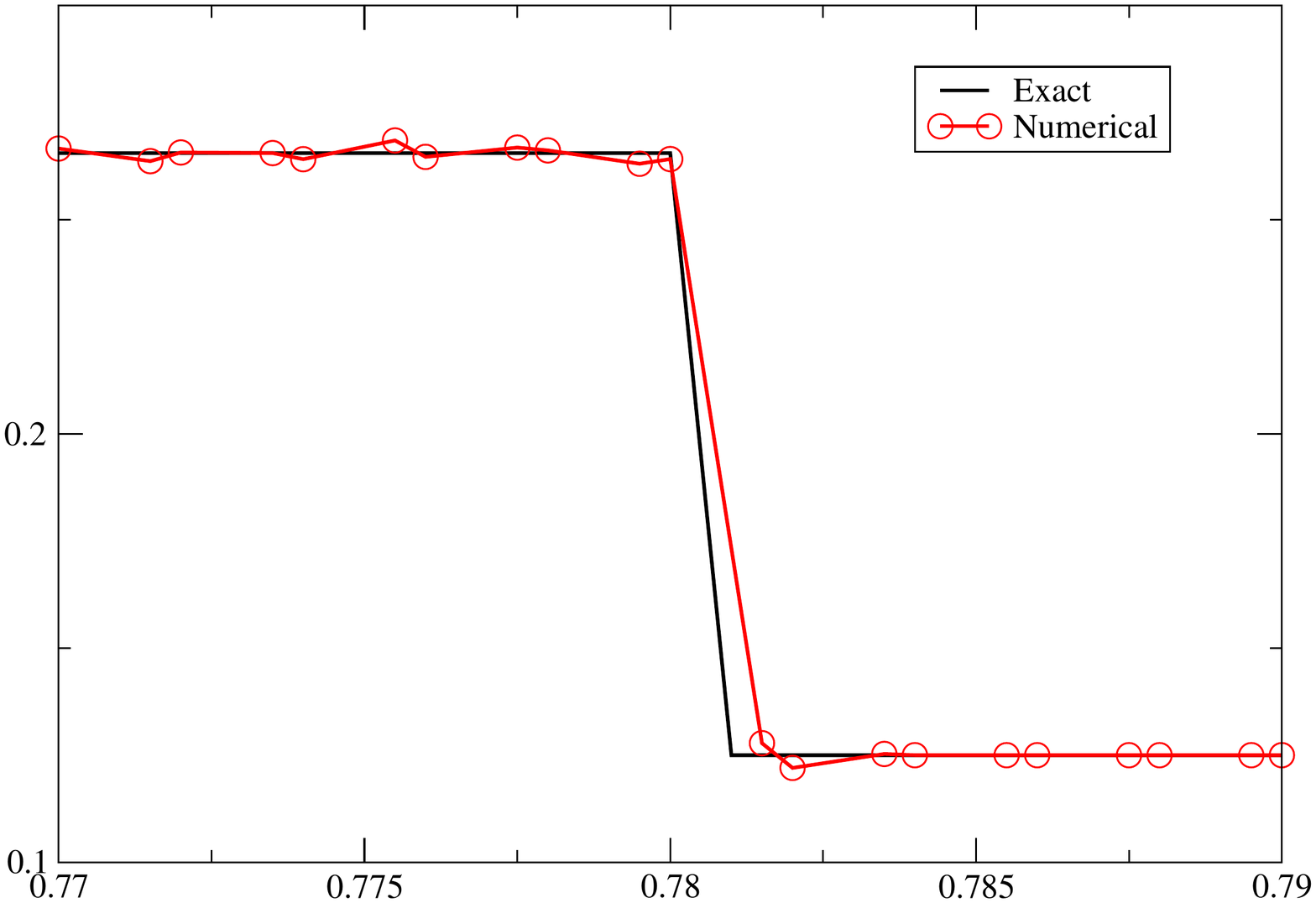}}
\end{center}
\caption{\label{sod:irregular:zoom} Zoom of the density around the shock wave.}
\end{figure}
\end{document}